\documentclass[oneside,pdflatex,sn-mathphys-num]{sn-jnl}

\ifdefined\pdfsuppresswarningpagegroup
\fi

\usepackage{amsmath,amssymb,mathtools}
\usepackage{amsthm}
\usepackage{booktabs}
\usepackage{graphicx}
\usepackage{tabularx}
\usepackage{placeins}
\usepackage{microtype}
\usepackage{xcolor}
\usepackage{multirow}

\usepackage{pgfplots}
\usepgfplotslibrary{groupplots}
\usetikzlibrary{calc}
\pgfplotsset{compat=1.18}

\hypersetup{
  pdfauthor={Zhengjie Sun},
  pdfsubject={Weight-adapted ANOVA kernel},
}
\usepackage[nameinlink,noabbrev]{cleveref}
\providecommand{\burl}[1]{\url{#1}}

\theoremstyle{thmstyleone}
\newtheorem{theorem}{Theorem}[section]
\newtheorem{proposition}[theorem]{Proposition}
\newtheorem{corollary}[theorem]{Corollary}
\newtheorem{lemma}[theorem]{Lemma}
\theoremstyle{thmstylethree}
\newtheorem{assumption}[theorem]{Assumption}
\newtheorem{definition}[theorem]{Definition}

\theoremstyle{thmstyletwo}
\newtheorem{remark}[theorem]{Remark}
\newtheorem*{remark*}{Remark}
\numberwithin{equation}{section}

\newcommand{\T}{\mathbb T}
\newcommand{\Z}{\mathbb Z}
\newcommand{\E}{\mathbb E}
\newcommand{\brho}{\boldsymbol{\rho}}
\newcommand{\bk}{\boldsymbol{k}}
\newcommand{\bell}{\boldsymbol{\ell}}

\newcommand{\bx}{\boldsymbol{x}}

\newcommand{\calK}{\mathcal{K}}
\DeclareMathOperator{\supp}{supp}

\begin{document}

\title[Weight-adapted ANOVA quasi-interpolation]
{Weight-adapted ANOVA kernels for high-dimensional periodic approximation}

\author*{\fnm{Zhengjie} \sur{Sun}}\email{zhengjiesun@njust.edu.cn}
\affil{\orgdiv{School of Mathematics and Statistics},
\orgname{Nanjing University of Science and Technology},
\orgaddress{\city{Nanjing}, \postcode{210094}, \country{China}}}

\abstract{In this paper, we construct normalized periodic kernels from univariate functions satisfying suitable moment conditions and derive sufficient conditions for high-order convolution approximation.  Two computable quasi-interpolation schemes are obtained by discretizing the convolution integral using either a full tensor grid or random samples. In each case, the analysis separates the continuous convolution error from the error introduced by discretization. The contribution of kernel energy to the error exposes the dimensional dependence of tensor-product kernels and motivates our weight-adapted ANOVA constructions. Under suitable summability conditions on the weights, we establish convolution error and kernel energy estimates that yield dimension-uniform error bounds for the resulting quasi-interpolants. Numerical experiments corroborate the theoretical convergence rates and illustrate the effect of coordinate weighting on dimensional dependence.}


\maketitle

\section{Introduction}
\label{sec:journal-introduction}

High-dimensional function approximation arises in uncertainty quantification \cite{smith_Book_2013},
parametric partial differential equations
\cite{cohen_2015ActaNumer_approximation}, molecular dynamics \cite{wang_2022NatComputSci_efficient}, and electronic-structure calculations \cite{kuhne_2020JCP_cp2k}. We consider the approximation of a periodic function \(f:\T^d\to\mathbb R\) from finitely many values
\(f(\bx_1),\ldots,f(\bx_M)\). Since a full tensor grid with \(n\) nodes per coordinate requires \(M=n^d\) evaluations, approximation in high
dimensions must exploit additional structure, such as mixed smoothness, anisotropy, or low-order coordinate interactions.
For functions with dominating mixed smoothness, hyperbolic-cross approximation selects Fourier frequencies subject to a product constraint \cite{dung_book2018}. Rank-1 and multiple rank-1 lattices provide structured sampling sets for the resulting trigonometric approximation problems
\cite{byrenheid_2017NM_tight,dick_book2022,Kammerer2013,KammererPottsVolkmer2015,kammerer_2019JAT_approximation,kuo_2021MCoM_function}. Sparse grids exploit mixed smoothness through combinations of anisotropic resolution levels \cite{AddyLatzTeckentrup2026,bungartz_2004Acta_sparse,gerstner_1998NumerAlgor_numerical,griebel_2009MCoM_optimized,GriebelSchneiderZenger1992}. Related approaches include approximation in Korobov spaces, B-spline quasi-interpolation, periodic Smolyak
operators, multilevel kernel methods, and periodic kernel interpolation \cite{Dung2016,DungMicchelli2013,KaarniojaEtAl2022,
KempfWendland2023,Kolomoitsev2022,ZengKritzerHickernell2009}.

Here we study kernel quasi-interpolation, which constructs a linear combination of kernel translates with coefficients obtained directly from sampled function values. Its construction therefore requires
no interpolation solve. The classical theory begins with spline quasi-interpolation \cite{deBoorFix1973} and extends to multivariate approximation by integer translates \cite{DynJacksonLevinRon1992}, kernel conditions ensuring prescribed approximation orders
\cite{CheneyLightXu1992}, and approximation from irregular point sets \cite{CheneyLei1994}. For radial kernels, approximation orders have been related to Fourier decay, polynomial reproduction, and generalized Strang--Fix conditions
\cite{NarcowichSunWard2007,wu_2005Adv_generalized}. Sparse grid and multilevel variants have also been developed \cite{FranzWendland2023,JeongKerseyYoon2021,KempfWendland2023,UstaLevesley2018}.

Error estimates for convolution-based quasi-interpolation must account for both approximation by convolution and discretization
of the integral. The convolution error is governed by the moment and decay properties of the kernel and the smoothness of the target
\cite{FranzWendland2023,GaoWangSunFasshauer2024,
sun_2025IMAJNA_scaled}. Replacing the integral by a sum over sampled values introduces a further error, which also depends on the sampling rule and quantitative bounds for the kernel
\cite{GaoFasshauerSunZhou2020,KolomoitsevLomakoTikhonov2023,KolomoitsevPrestin2021,sun_2026SINUM_monte,XiaoBoyd2014}. Hubbert, J\"ager, and Levesley establish convergence rates for Gaussian convolution and sparse grid combinations of anisotropic
tensor-product Gaussian convolution operators
\cite{HubbertJagerLevesley2023}. Although they discuss a discrete analogue, their convergence estimates do not include the additional
discretization error. Periodic full-grid and sparse grid quasi-interpolants have also been analyzed for a particular class of generalized Gaussian kernels \cite{GaoHuSunWang2026}. Our analysis treats both sources of error for a broader class of periodic kernels characterized by moment and decay conditions.

Independent random sampling allows arbitrary sample sizes, but does not eliminate the dependence of the error on dimension. For tensor-product kernels, the kernel energy contributes to the sampling variance and typically increases as the scale decreases or the
dimension grows \cite{GaoFasshauerSunZhou2020,sun_2026SINUM_monte}.
To control this dependence, we exploit coordinate importance and interaction structure. Weighted tensor-product spaces provide a standard framework for quantifying unequal coordinate influence \cite{WasilkowskiWozniakowski1999}. The tractability of approximation in weighted Korobov spaces has been studied in \cite{EbertPillichshammer2021,NovakSloanWozniakowski2004}, and constructive lattice algorithms have been developed for product, order-dependent, and more general weights \cite{CaiGoda2026,CaiGodaKazashi2026,CoolsKuoNuyensSloan2020,CoolsKuoNuyensSloan2021,KuoSloanWozniakowski2006}. ANOVA decompositions complement coordinate weights by resolving a function into components indexed by coordinate subsets. Restricting
the size of these subsets limits the interaction order while allowing all coordinates to remain active. ANOVA kernels express this decomposition in a reproducing kernel framework \cite{DurrandeEtAl2013}. In the periodic setting, sparse ANOVA support has been exploited in Fourier approximation \cite{PottsSchmischke2021} and hyperbolic wavelet regression \cite{LippertPottsUllrich2023}. Sampling inequalities and constructive approximation results for functions assembled from lower-dimensional
components are established in \cite{rieger_2024SINUM_approximability,RiegerWendland2026}. These results motivate the incorporation of coordinate weights and interaction structure directly into the quasi-interpolation kernel.

The paper makes two main contributions. First, we construct normalized periodic kernels from univariate functions satisfying prescribed moment and decay conditions and prove high-order
convolution estimates. We then derive separate discretization estimates for full tensor grids and independent uniform samples. For independent sampling, an exact mean-square error identity
separates the squared convolution bias from the sampling variance and makes the contribution of kernel energy explicit. Second, to control the dimensional dependence of this energy, we construct single-level and multilevel weight-adapted ANOVA kernels. Under appropriate summability conditions on the coordinate weights and a fixed maximal interaction order, we obtain convolution error
and kernel-energy estimates with constants independent of the ambient dimension. Together with the mean-square identity, these estimates yield dimension-uniform error bounds for independently
sampled quasi-interpolants on the specified weighted function classes. For the multilevel construction, the algebraic convergence exponent is also independent of the interaction order, although the
constants and logarithmic factors may depend on it. Numerical experiments test the approximation orders, the effect of coordinate weighting, and alternative discretizations of the convolution integral.

The remainder of the paper is organized as follows. Section~2 introduces the function spaces and kernel assumptions. Section~3 constructs the periodic kernels and analyzes the associated convolution operators. Section~4 derives the full-grid and random discretization estimates. Sections~5 and~6 develop the single-level and multilevel weight-adapted ANOVA methods,
respectively. Section~7 presents the numerical experiments. Section~8 concludes the paper.

\section{Preliminaries}
\label{sec:jfg-preliminaries}

We use the following conventions. Let \(\T^d:=\mathbb R^d/\mathbb Z^d\), identified with \([0,1)^d\) and equipped with normalized Lebesgue measure. Addition and subtraction
on \(\T^d\) are understood modulo \(\mathbb Z^d\). 
For \(\bk=(k_1,\ldots,k_d)\in\mathbb Z^d\), let
\(\chi_{\bk}(\boldsymbol x):=
\exp(2\pi\mathrm{i}\,\bk\cdot\boldsymbol x)\).
The family \((\chi_{\bk})_{\bk\in\mathbb Z^d}\) is an orthonormal basis of
\(L_2(\T^d)\). For \(f\in L_1(\T^d)\), the Fourier coefficient is
$$\widehat f(\bk):=
\int_{\T^d}f(\boldsymbol x)
\overline{\chi_{\bk}(\boldsymbol x)}\,d\boldsymbol x.$$
In particular, Parseval's identity gives
\(\|f\|_2^2=\sum_{\bk\in\mathbb Z^d}|\widehat f(\bk)|^2\) for
\(f\in L_2(\T^d)\). We use the standard notation
\(\langle t\rangle:=(1+t^2)^{1/2}\) for \(t\in\mathbb R\).

For \(1\leq p\leq\infty\), \(f\in L_p(\T^d)\), and
\(g\in L_1(\T^d)\), their periodic convolution is defined almost everywhere by
\begin{equation}
	(f*g)(\boldsymbol x)
	:=
	\int_{\T^d}
	f(\boldsymbol y)
	g(\boldsymbol x-\boldsymbol y)
	\,d\boldsymbol y.
	\label{eq:periodic-convolution}
\end{equation}
Young's inequality yields
\(\|f*g\|_p\leq\|f\|_p\|g\|_1\), so convolution with \(g\) defines a
bounded operator on \(L_p(\T^d)\). Moreover,
\(\widehat{f*g}(\bk)=\widehat f(\bk)\widehat g(\bk)\) for any
\(\bk\in\mathbb Z^d\). Thus, periodic convolution acts diagonally in the
Fourier basis, with multiplier sequence
\((\widehat g(\bk))_{\bk\in\mathbb Z^d}\).

\subsection{Function spaces}
\label{subsec:function-spaces}

For \(s>1/2\), the periodic Sobolev space of dominating mixed
smoothness is defined by
\begin{equation}
	H^s_{\mathrm{mix}}(\T^d)
	:=
	\Big\{
	f\in L_2(\T^d):
	\|f\|_{s,\mathrm{mix}}^2
	:=
	\sum_{\bk\in\Z^d}
	|\widehat f(\bk)|^2
	\prod_{j=1}^d \langle k_j\rangle^{2s}
	<\infty
	\Big\}.
	\label{eq:jfg-mixed-sobolev}
\end{equation}
Let
\(A_s:=\sum_{\ell\in\Z\setminus\{0\}}\langle\ell\rangle^{-2s}\).
Because \(s>1/2\), we have \(A_s<\infty\), and the
Cauchy--Schwarz inequality yields
\[
\sum_{\bk\in\Z^d}|\widehat f(\bk)|
\leq
(1+A_s)^{d/2}\|f\|_{s,\mathrm{mix}}.
\]
Consequently, the Fourier series of \(f\) converges absolutely and uniformly,
and \(H^s_{\mathrm{mix}}(\T^d)\hookrightarrow C(\T^d)\) with an
embedding constant that depends on \(d\).

To obtain estimates that are uniform in \(d\), we introduce product
weights that quantify the relative importance of the coordinate
directions. Let \([d]:=\{1,\ldots,d\}\), and, for
\(\bk\in\Z^d\), set
\(\supp\bk:=\{j\in[d]:k_j\neq0\}\). Throughout the paper, let
\(1\geq\gamma_1\geq\gamma_2\geq\cdots>0\) be a nonincreasing sequence
of coordinate weights, and 
$$G_\alpha:=\sum_{j=1}^\infty\gamma_j^{1/(2\alpha)}.$$
For \(\boldsymbol u\subseteq[d]\), write
\(\gamma_{\boldsymbol u}:=\prod_{j\in\boldsymbol u}\gamma_j\), with
the convention \(\gamma_{\emptyset}:=1\).

\begin{definition}
	\label{def:jdaq-weighted-korobov}
	Let \(s>1/2\). The product-weighted Korobov space
	\(H^s_{\mathrm{Kor},\boldsymbol\gamma}(\T^d)\) consists of all
	\(f\in L_2(\T^d)\) for which
	\begin{equation}
		\|f\|_{s,\boldsymbol\gamma}^2
		:=
		\sum_{\bk\in\Z^d}
		\gamma_{\supp\bk}^{-1}
		\prod_{j\in\supp\bk}
		\langle k_j\rangle^{2s}
		|\widehat f(\bk)|^2
		<\infty.
		\label{eq:jdaq-weighted-korobov-norm}
	\end{equation}
\end{definition}

%For fixed \(d\), the weighted and unweighted spaces coincide as
%sets, and their norms satisfy
%\[
%\|f\|_{s,\mathrm{mix}}
%\leq
%\|f\|_{s,\boldsymbol\gamma}
%\leq
%\gamma_{[d]}^{-1/2}\|f\|_{s,\mathrm{mix}}.
%\]

Fourier projections decompose the space according to
coordinate interactions. For \(\boldsymbol u\subseteq[d]\), let
\(P_{\boldsymbol u}\) denote the Fourier projection defined by
\[
\widehat{P_{\boldsymbol u}f}(\bk)
:=
\mathbf 1_{\{\supp\bk=\boldsymbol u\}}\widehat f(\bk),
\quad \bk\in\Z^d.
\]
The ranges of the projections \(P_{\boldsymbol u}\) are mutually
orthogonal in both \(L_2(\T^d)\) and
\(H^s_{\mathrm{Kor},\boldsymbol\gamma}(\T^d)\), and
\[
f=\sum_{\boldsymbol u\subseteq[d]}P_{\boldsymbol u}f.
\]
This Fourier decomposition coincides with the functional ANOVA
decomposition associated with normalized product measure on
\(\T^d\).
For \(q\in\mathbb N_0\), define
\(P_q:=\sum_{|\boldsymbol u|\leq q}P_{\boldsymbol u}\).
We adopt the convention \(P_q=I\) whenever \(q\geq d\). The
corresponding finite superposition space is
\begin{equation}
\begin{aligned}
	H^s_{\mathrm{Kor},\boldsymbol\gamma,q}(\T^d)
	&:=
	P_qH^s_{\mathrm{Kor},\boldsymbol\gamma}(\T^d)\\
	&=
	\left\{
	f\in H^s_{\mathrm{Kor},\boldsymbol\gamma}(\T^d):
	\widehat f(\bk)=0
	\text{ whenever }|\supp\bk|>q
	\right\}.
\end{aligned}
	\label{eq:jdaq-hard-superposition-space}
\end{equation}
Because \(P_q\) is orthogonal with respect to both the
\(L_2\)-inner product and the weighted Korobov inner product, it is
contractive on both spaces. In particular,
\(H^s_{\mathrm{Kor},\boldsymbol\gamma,q}(\T^d)\) is a closed subspace
of \(H^s_{\mathrm{Kor},\boldsymbol\gamma}(\T^d)\).

The following lemma makes the dependence of the
\(L_\infty\)-embedding constant on the product weights explicit.

\begin{lemma}
	\label{lem:jdaq-uniform-embedding}
	Let \(s>1/2\). Then, for every \(d\in\mathbb N\) and
	\(f\in H^s_{\mathrm{Kor},\boldsymbol\gamma}(\T^d)\), 
	\begin{equation}
		\|f\|_\infty
		\leq
		\Big(
		\prod_{j=1}^d(1+A_s\gamma_j)
		\Big)^{1/2}
		\|f\|_{s,\boldsymbol\gamma}.
		\label{eq:jdaq-finite-dimensional-embedding}
	\end{equation}
	Moreover, if \(\sum_{j=1}^\infty\gamma_j<\infty\), the
	canonical embeddings
	$\iota_{s,\boldsymbol\gamma}^{(d)}:
	H^s_{\mathrm{Kor},\boldsymbol\gamma}(\T^d)
	\longrightarrow C(\T^d)$
	are bounded uniformly in \(d\) with
	$$\sup_{d\in\mathbb N}
	\big\|\iota_{s,\boldsymbol\gamma}^{(d)}\big\|
	\leq
	\exp\big(
	\frac{A_s}{2}\sum_{j=1}^\infty\gamma_j
	\big).$$
	In particular, if \(G_\alpha<\infty\) with \(\alpha>1/2\), then
	\begin{equation}
		\|f\|_\infty
		\leq
		\exp\left(\frac{A_sG_\alpha}{2}\right)
		\|f\|_{s,\boldsymbol\gamma},
		\quad d\in\mathbb N.
		\label{eq:jdaq-uniform-embedding}
	\end{equation}
\end{lemma}

\begin{proof}
	For \(f\in H^s_{\mathrm{Kor},\boldsymbol\gamma}(\T^d)\), the
	Cauchy--Schwarz inequality gives
	\begin{align*}
		\sum_{\bk\in\Z^d}|\widehat f(\bk)|
		&\leq
		\|f\|_{s,\boldsymbol\gamma}
		\Big(
		\sum_{\bk\in\Z^d}
		\gamma_{\supp\bk}
		\prod_{j\in\supp\bk}
		\langle k_j\rangle^{-2s}
		\Big)^{1/2}
		\\
		&=
		\|f\|_{s,\boldsymbol\gamma}
		\Big(
		\prod_{j=1}^d(1+A_s\gamma_j)
		\Big)^{1/2},
	\end{align*}
	where the last identity follows by separating the zero and nonzero
	frequencies in each coordinate. The Fourier series
	of \(f\) converges absolutely and uniformly and defines
	a continuous representative satisfying
	\eqref{eq:jdaq-finite-dimensional-embedding}.
	
	If \(\sum_{j=1}^\infty\gamma_j<\infty\), then
	\[
	\prod_{j=1}^d(1+A_s\gamma_j)
	\leq
	\exp\Big(A_s\sum_{j=1}^d\gamma_j\Big)
	\leq
	\exp\Big(A_s\sum_{j=1}^\infty\gamma_j\Big),
	\]
	which proves the dimension-uniform boundedness of the embeddings.
	
	Finally, let \(\alpha>1/2\) and \(G_\alpha<\infty\).
	Since \(0<\gamma_j\leq1\), we have
	\(\gamma_j\leq\gamma_j^{1/(2\alpha)}\). Hence
	\(\sum_{j=1}^\infty\gamma_j\leq G_\alpha\). Substituting it into the preceding exponential bound proves
	\eqref{eq:jdaq-uniform-embedding}.
\end{proof}

\subsection{Kernels}
We introduce the moment condition imposed on the univariate function from which the periodic kernels will be constructed. This differs from the radial formulation on $\mathbb R^d$ commonly used in the literature \cite{FranzWendland2023}.

\begin{definition}(Moment condition)
	\label{def:moment-conditions}
	Let $m=2p$, where $p\geq1$ is an integer. A bounded continuous function $\varphi\in L_1([0,\infty))$ is said to satisfy the moment condition of order $m$ if
	\begin{align}
		\int_0^\infty \varphi(t)\,dt
		&=1,
		\label{eq:jfg-profile-mass}\\
		\int_0^\infty \varphi(t)t^{2j}\,dt
		&=0,
		\quad 1\leq j\leq p-1,
		\label{eq:jfg-profile-vanishing-moments}\\
		\int_0^\infty |\varphi(t)|t^m\,dt
		&<\infty.
		\label{eq:jfg-profile-moments}
	\end{align}
\end{definition}

The first condition provides the normalization required for convolution approximation, whereas the remaining moment conditions ensure the cancellations necessary to achieve high-order approximation
\cite{deBoorFix1973,DynJacksonLevinRon1992,CheneyLightXu1992}.
High-order moment conditions can be imposed by finite linear combinations of dilated radial functions; see, for example, \cite{FranzWendland2023,GaoZhou2020}. A related Vandermonde cancellation argument for dilated zonal kernels was established in
\cite[Proposition~3.10]{sun_2025IMAJNA_scaled}. The following proposition adapts this construction to the present setting.

\begin{proposition}
	\label{prop:jfg-compact-moment-profile}
	Let \(m=2p\), where \(p\in\mathbb N\), and let
	\(\eta\in W^{2,1}(0,\infty)\) be real-valued. Assume that
    \begin{equation}\label{eq:moment_eq_assump}
        \mu_0:=\int_0^\infty\eta(t)\,dt\neq0 \quad \text{and}\quad \int_0^\infty|\eta(t)|t^m\,dt<\infty.
    \end{equation}
	Choose distinct numbers
	\(\delta_1,\ldots,\delta_p>0\), set \(\zeta_q:=\delta_q^2\), and define
	\begin{equation}
		\omega_q:=
		\prod_{\substack{1\leq v\leq p\\v\neq q}}
		\frac{-\zeta_v}{\zeta_q-\zeta_v},
		\quad
		\varphi(t):=\frac1{\mu_0}
		\sum_{q=1}^p\omega_q\delta_q^{-1}
		\eta\Bigl(\frac{t}{\delta_q}\Bigr).
		\label{eq:jfg-compact-profile}
	\end{equation}
	Then \(\varphi\in W^{2,1}(0,\infty)\) and satisfies the moment condition of order \(m\).
\end{proposition}

\begin{proof}
	The coefficients $\omega_q$ are the Lagrange interpolation weights for evaluation at zero at the nodes
	$\zeta_1,\ldots,\zeta_p$. Consequently,
	\begin{equation}
		\sum_{q=1}^p\omega_q\zeta_q^j
		=
		\begin{cases}
			1, & j=0,\\
			0, & 1\leq j\leq p-1.
		\end{cases}
		\label{eq:jfg-compact-lagrange-identities}
	\end{equation}
	For $0\leq j\leq p-1$, let
	$\mu_{2j}:=\int_0^\infty\eta(t)t^{2j}\,dt$. These moments are absolutely convergent by \eqref{eq:moment_eq_assump} and the fact that $\eta\in L_1(0,\infty)$. A change of variables in each term of \eqref{eq:jfg-compact-profile} gives
	\[
	\int_0^\infty \varphi(t)t^{2j}\,dt
	=
	\frac{\mu_{2j}}{\mu_0}
	\sum_{q=1}^p\omega_q\delta_q^{2j}.
	\]
	Since $\delta_q^{2j}=\zeta_q^j$, identity
	\eqref{eq:jfg-compact-lagrange-identities} implies that
	\[
	\int_0^\infty \varphi(t)t^{2j}\,dt
	=
	\begin{cases}
		1, & j=0,\\
		0, & 1\leq j\leq p-1.
	\end{cases}
	\]
	Thus $\varphi$ has unit mass and its even moments of orders
	$2,4,\ldots,2p-2$ vanish.
	
	Because $\varphi$ is a finite linear combination of dilates of $\eta$, it belongs to $W^{2,1}(0,\infty)$. The integrability assumption on $\eta$ also yields
	$\int_0^\infty|\varphi(t)|t^{2p}\,dt<\infty$. Thus $\varphi$
	satisfies all the requirements of
	Definition~\ref{def:moment-conditions}.
\end{proof}

%\begin{remark}
%	%	The minimum in Proposition~\ref{prop:jfg-minimum-energy-profile} concerns
%	%	the leading variance constant.  It need not minimize the consistency
%	%	constant, which also depends on the first nonvanishing moment and on an
%	%	absolute \(2p\)-th moment of the profile.  A profile should therefore be
%	%	compared through the resulting bias--variance curve, not through its
%	%	\(L_2\)-energy alone.
%	A concrete example is provided by the $C^2$ Wendland function on
%	$\mathbb{R}^2$,
%	\[
%	\eta(t)=(1-t)_+^4(4t+1),
%	\]
%	for which $\mu_0=1/3$. If $p=1$ and $\delta_1=1$, then
%	\eqref{eq:jfg-compact-profile} gives $\varphi=3\eta$, which satisfies
%	the moment conditions of order two. If $p=2$, $\delta_1=1$, and
%	$\delta_2=2$, then
%	\[
%	\varphi(t)=4\eta(t)-\frac12\eta(t/2),
%	\]
%	which satisfies the moment conditions of order four.
%\end{remark}

Throughout the paper, $c,C>0$ denote constants that may depend on the fixed kernel, smoothness parameters, superposition order and weight, but are independent of the dimension, sample size, and scale, unless indicated otherwise by a subscript. Their values may change from one occurrence to the next. For
nonnegative quantities $A$ and $B$, we write $A\lesssim_d B$ if
$A\leq C_dB$. We define $\gtrsim_d$ and $\asymp_d$
analogously. Additional subscripts indicate dependence on other fixed
parameters.

\section{Periodic convolution approximation}
\label{sec:periodic_convolution}

Restricting Euclidean radial basis functions to spheres provides a standard construction for spherical basis functions \cite{NarcowichSunWard2007}. In particular, under angular parametrization of the unit circle, a radial kernel on $\mathbb{R}^{2}$ becomes a univariate periodic kernel. Taking coordinatewise tensor products of such kernels then yields multivariate periodic kernels on the torus. In this section, we first study the moment and decay conditions below, which determine when these restricted kernels form an approximate identity, and then discuss the tensor-product structure that enables periodic convolution approximation.

\subsection{Restricted kernels}
\label{subsec:jfg-circle-restriction}

%The relationship between the Fourier transform of a radial function and the spherical harmonic expansion coefficients of its restriction was investigated in .

We identify the one-dimensional torus $\mathbb{T}$
with the unit circle through the parametrization
$\boldsymbol{e}(x):=(\cos(2\pi x),\sin(2\pi x))$. For $x,y\in\mathbb{T}$,
the corresponding chordal distance satisfies
$$\|\boldsymbol{e}(x)-\boldsymbol{e}(y)\|_2
=
2\bigl|\sin\bigl(\pi(x-y)\bigr)\bigr|.$$
Viewing $\varphi$ as a radial function on $\mathbb{R}^2$, consider the scaled radial kernel
\[
(\boldsymbol{u},\boldsymbol{v})
\longmapsto
\rho^{-2}
\varphi\left(
\rho^{-1}
\lVert\boldsymbol{u}-\boldsymbol{v}\rVert_2
\right).
\]
Its restriction to the unit circle depends only on $x-y$ and is
translation invariant on $\mathbb{T}$. For $0<\rho\leq\rho_0$, let
\begin{equation}
	A_\rho:=
	\int_0^{2/\rho}
	\frac{\varphi(t)}
	{\sqrt{1-\rho^2t^2/4}}\,dt.
	\label{eq:jfg-circle-normalizing-factor}
\end{equation}
Whenever $A_\rho\neq 0$, we define the normalized restriction by
\begin{equation}
	\psi_\rho(x):=
	\frac{\pi}{\rho A_\rho}
	\varphi\left(
	\frac{2|\sin(\pi x)|}{\rho}
	\right),
	\quad x\in\mathbb{T}.
	\label{eq:jfg-circle-kernel}
\end{equation}
The change of variables used below shows that
$\int_{\mathbb{T}}\psi_\rho(x)\,dx=1$. 
Its Fourier coefficients are
\begin{equation}
	\widehat\psi_\rho(\ell)
	=\frac1{A_\rho}
	\int_0^{2/\rho}
	\frac{\varphi(t)
		\cos\bigl(2\ell\arcsin(\rho t/2)\bigr)}
	{\sqrt{1-\rho^2t^2/4}}\,dt.
	\label{eq:jfg-circle-multiplier}
\end{equation}

The subsequent analysis relies on two properties of the Fourier coefficients $\widehat{\psi}_\rho(\ell)$: a low-frequency estimate quantifying their approximate identity behavior and a uniform decay bound. These requirements are formalized in the following assumptions.

\begin{assumption}
	\label{ass:jfg-kernel}
	Let $m\geq 1$ and $0<\rho_0\leq1$. For any
	$0<\rho\leq\rho_0$, let $\psi_\rho\in L_1(\T)$ be defined as in \eqref{eq:jfg-circle-kernel}. Suppose that there exists a constant $c_1>0$, independent of $\rho$ and $\ell$, such that $\widehat\psi_\rho(0)=1$ and
	\begin{equation}
		\label{eq:jfg-kernel-assumption}
		\bigl|1-\widehat\psi_\rho(\ell)\bigr|
		\leq c_1
		\min\bigl\{1,(\rho|\ell|)^m\bigr\},
		\quad \ell\in\Z.
		%\\
		%\bigl|\widehat\psi_\rho(\ell)\bigr|.
		%&\leq c_2(1+\rho|\ell|)^{-2\kappa},
		%\quad \ell\in\Z.
		%\label{eq:jfg-kernel-decay}
	\end{equation}
\end{assumption}

\begin{assumption}
	\label{ass:jfg-kernel1}
	Let $\kappa>0$ and $0<\rho_0\leq1$. Suppose that there exists a constant $c_2>0$, independent of $\rho$ and $\ell$, such that
	$$|\widehat\psi_\rho(\ell)|
	\leq c_2(1+\rho|\ell|)^{-2\kappa},
	\quad \ell\in\Z.$$
\end{assumption}

Assumption~\ref{ass:jfg-kernel} characterizes a fundamental kernel property required for the convergence analysis of convolution approximation.
Although this assumption was originally formulated in the spherical setting \cite{sun_2025IMAJNA_scaled,sun_2026SINUM_monte}, it has been verified only for some specific kernels. Here, we establish a general sufficient condition in the periodic setting. We prove that, whenever the underlying radial function $\varphi$ satisfies the
prescribed moment and decay conditions, its restriction $\psi_\rho$ defined in \eqref{eq:jfg-circle-kernel} satisfies Assumption~\ref{ass:jfg-kernel}. Then, the kernel $\psi_\rho$ can be employed to construct convolution operators.

%Let
%$\boldsymbol e(x):=(\cos(2\pi x),\sin(2\pi x))$.
%The chordal distance between two points on the unit circle satisfies
%	$\|\boldsymbol e(x)-\boldsymbol e(y)\|_2
%	=2\big|\sin\bigl(\pi(x-y)\bigr)\big|$.
%Consequently, restriction of the scaled radial kernel
%$\rho^{-2}
%\varphi\left(\rho^{-1}\|\boldsymbol u-\boldsymbol v\|_2\right)$
%to the circle yields a periodic translation-invariant kernel. For
%$0<\rho\leq\rho_0$, we define the one-dimensional periodic kernel
%\begin{equation}
%	\psi_\rho(x):=
%	\frac{\pi}{\rho A_\rho}
%	\varphi\left(
%	\frac{2|\sin(\pi x)|}{\rho}
%	\right),
%	\label{eq:jfg-circle-kernel}
%\end{equation}
%where
%\begin{equation}
%	A_\rho:=
%	\int_0^{2/\rho}
%	\frac{\varphi(t)}
%	{\sqrt{1-\rho^2t^2/4}}\,dt.
%	\label{eq:jfg-circle-normalizing-factor}
%\end{equation}
%The integral in \eqref{eq:jfg-circle-normalizing-factor} is understood as an
%improper integral at $2/\rho$. The unnormalized circle restriction has mass
%$A_\rho/(\pi\rho)$; hence the normalization in
%\eqref{eq:jfg-circle-kernel} replaces the two-dimensional factor
%$\rho^{-2}$ by the appropriate one-dimensional scaling factor.

\begin{theorem}
	\label{thm:jfg-normalized-circle-kernel} 
    Let $m=2p$, and let $\varphi$ satisfy the moment condition of order $m$. Assume further that there exist constants $\varepsilon>0$ and $C_\varphi>0$ such that
	\begin{equation}
		|\varphi(t)|
		\leq C_\varphi
		(1+t)^{-m-1-\varepsilon},
		\quad t\geq0.
		\label{eq:jfg-circle-decay-assumption}
	\end{equation}
    Then the kernel
$\psi_\rho$ belongs to $L_1(\mathbb{T})$ and satisfies
Assumption~\ref{ass:jfg-kernel} with order $m$. More precisely, we have $\widehat\psi_\rho(0)=1$ and
    $$\bigl|1-\widehat\psi_\rho(\ell)\bigr|
		\leq C
		\min\bigl\{1,(\rho|\ell|)^m\bigr\},
		\quad \ell\in\Z.$$
\end{theorem}

The proof relies on the following lower bound for the normalization constant $A_\rho$.
\begin{lemma}
\label{lem:Arho}
	Under the assumptions of Theorem \ref{thm:jfg-normalized-circle-kernel}, let $A_\rho$ be given by \eqref{eq:jfg-circle-normalizing-factor}. Then there exists $\rho_0\in(0,1]$ such that
	$$A_\rho\geq\frac12,\quad 0<\rho\leq\rho_0.$$
\end{lemma}
\begin{proof}
    The decay assumption \eqref{eq:jfg-circle-decay-assumption} implies that all integrals below are absolutely convergent. We first prove that
	$A_\rho\rightarrow1$ as $\rho\downarrow0$. Split the integral in \eqref{eq:jfg-circle-normalizing-factor} at $1/\rho$. On $[0,1/\rho]$, we have
    $\frac{1}
	{\sqrt{1-\rho^2t^2/4}}\leq \frac{2}{\sqrt{3}}$.
    Hence, after extending the integrand by zero to $(1/\rho,\infty)$, dominated convergence gives
    $$\int_0^{1/\rho}
	\frac{\varphi(t)}
	{\sqrt{1-\rho^2t^2/4}}\,dt\rightarrow\int_0^\infty\varphi(t)\,dt=1.$$
    For $[1/\rho,2/\rho]$, the substitution $u=\rho t$ and \eqref{eq:jfg-circle-decay-assumption} yield
    $$\int_{1/\rho}^{2/\rho}
	\frac{|\varphi(t)|}
	{\sqrt{1-\rho^2t^2/4}}\,dt\leq C\rho^{m+\varepsilon}\int_{1}^2\frac{u^{-m-1-\varepsilon}}{\sqrt{1-u^2/4}}\,du\leq C\rho^{m+\varepsilon}.$$
    Thus $A_\rho\rightarrow1$. Reducing $\rho_0$ if necessary, we obtain $A_\rho\geq\frac12$ for $0<\rho\leq \rho_0$.
\end{proof}

\begin{proof}[\textbf{Proof of Theorem \ref{thm:jfg-normalized-circle-kernel}}]
	By symmetry and the change of variables
	$t=2\sin(\pi x)/\rho$ on $[0,1/2]$, we obtain
	\[
	\int_\T\psi_\rho(x)\,dx
	=\frac1{A_\rho}
	\int_0^{2/\rho}
	\frac{\varphi(t)}
	{\sqrt{1-\rho^2t^2/4}}\,dt
	=1.
	\]
	Applying the estimate in Lemma~\ref{lem:Arho} gives
	$\|\psi_\rho\|_{L_1(\T)}\leq C$.
	
	For $\ell\neq0$, we define
	\[
	H_\ell(z):=
	\frac{
		1-\cos\bigl(2\ell\arcsin(z/2)\bigr)
	}{
		\sqrt{1-z^2/4}
	},
	\quad 0\leq z<2.
	\]
	The function $H_\ell$ is even and smooth in a neighborhood of the origin. On the interval $0\leq z\leq1$, the derivatives of $\arcsin(z/2)$ and $(1-z^2/4)^{-1/2}$ are uniformly bounded. Repeated differentiation of the cosine shows that the Taylor polynomial of $H_\ell$ at zero
	through degree $m-1$ has the form
	\[
	\sum_{j=1}^{p-1} b_{j,\ell}z^{2j},
	\quad \text{with}\quad
	|b_{j,\ell}|\leq C_m|\ell|^{2j}.
	\]
	For $|\ell|\geq1$, Taylor's theorem then gives
	$\big|
	H_\ell(z)-\sum_{j=1}^{p-1}b_{j,\ell}z^{2j}
	\big|
	\leq C_m|\ell|^m z^m$ for $0\leq z\leq1$.
    
	For $1<z<2$, we have $
	|H_\ell(z)|
	\leq \frac{2}{\sqrt{1-z^2/4}}$.
	Since $|\ell|\geq1$ and $z\geq1$,
	\[
	\Big|
	\sum_{j=1}^{p-1}b_{j,\ell}z^{2j}
	\Big|
	\leq C_m|\ell|^m z^m
	\leq
	\frac{C_m|\ell|^m z^m}{\sqrt{1-z^2/4}}.
	\]
	Hence, combining the two ranges yields
	\begin{equation}
		\Big|
		H_\ell(z)-\sum_{j=1}^{p-1}b_{j,\ell}z^{2j}
		\Big|
		\leq
		\frac{C_m|\ell|^m z^m}{\sqrt{1-z^2/4}},
		\quad 0\leq z<2.
		\label{eq:jfg-circle-taylor-remainder}
	\end{equation}

%	The function $H_\ell$ is even and smooth in a neighborhood of the
%	origin. The derivatives of \(\arcsin(z/2)\) and
%	\((1-z^2/4)^{-1/2}\) are bounded on this interval, and repeated
%	differentiation of the cosine contributes at most \(C_m|\ell|^m\). Thus, Taylor’s theorem at zero shows that its polynomial through
%	degree $m-1$ has the form
%	$$\sum_{j=1}^{p-1}b_{j,\ell}z^{2j},\quad 
%	\text{with}\quad
%	|b_{j,\ell}|\leq C_m|\ell|^{2j}.$$
%	Moreover, for $|\ell|\geq 1$, its remainder satisfies
%	\begin{equation}
%		\Big|
%		H_\ell(z)
%		-\sum_{j=1}^{p-1}b_{j,\ell}z^{2j}
%		\Big|
%		\leq
%		\frac{C_m|\ell|^m z^m}
%		{\sqrt{1-z^2/4}},
%		\quad 0\leq z<2.
%		\label{eq:jfg-circle-taylor-remainder}
%	\end{equation}
	Using \eqref{eq:jfg-circle-multiplier}, we write
	\[
	1-\widehat\psi_\rho(\ell)
	=\frac1{A_\rho}
	\int_0^{2/\rho}
	\varphi(t)H_\ell(\rho t)\,dt.
	\]
	The moment conditions hold on $(0,\infty)$, whereas the above integral is truncated at $2/\rho$. Hence, for
	$1\leq j\leq p-1$, we have
	$\int_0^{2/\rho}t^{2j}\varphi(t)\,dt=-\int_{2/\rho}^\infty t^{2j}\varphi(t)\,dt$.
	The decay assumption \eqref{eq:jfg-circle-decay-assumption} gives
	\begin{align*}
		\int_{2/\rho}^\infty t^{2j}|\varphi(t)|\,dt
		&\leq
		C\int_{2/\rho}^\infty
		t^{2j-m-1-\varepsilon}\,dt
		\leq
		C\rho^{m+\varepsilon-2j}.
	\end{align*}
	It follows that
	\begin{align*}
		\rho^{2j}|b_{j,\ell}|
		\Big|
		\int_{2/\rho}^\infty
		t^{2j}\varphi(t)\,dt
		\Big|
		&\leq
		C\rho^{m+\varepsilon}|\ell|^{2j}
		\leq
		C\rho^m|\ell|^m,
	\end{align*}
	where we used $\rho\leq1$, $|\ell|\geq1$, and $2j<m$.
	
	It remains to estimate the Taylor remainder. Equation
	\eqref{eq:jfg-circle-taylor-remainder} leads to
	\begin{align*}
		\int_0^{2/\rho}
		|\varphi(t)|
		\Big|
		H_\ell(\rho t)
		-\sum_{j=1}^{p-1}
		b_{j,\ell}(\rho t)^{2j}
		\Big|dt
		\leq
		C(\rho|\ell|)^m
		\int_0^{2/\rho}
		\frac{t^m|\varphi(t)|}
		{\sqrt{1-\rho^2t^2/4}}\,dt.
	\end{align*}
	The last integral is bounded uniformly in $\rho$.
	On $[0,1/\rho]$, the denominator bound gives
	\[
	\int_0^{1/\rho}
	\frac{t^m|\varphi(t)|}
	{\sqrt{1-\rho^2t^2/4}}\,dt
	\leq
	\frac2{\sqrt3}
	\int_0^\infty t^m|\varphi(t)|\,dt
	\leq C.
	\]
	On $[1/\rho,2/\rho]$, the decay assumption \eqref{eq:jfg-circle-decay-assumption} gives
	\begin{align*}
		\int_{1/\rho}^{2/\rho}
		\frac{t^m|\varphi(t)|}
		{\sqrt{1-\rho^2t^2/4}}\,dt
		&\leq
		C\int_{1/\rho}^{2/\rho}
		\frac{t^{-1-\varepsilon}}
		{\sqrt{1-\rho^2t^2/4}}\,dt
		\leq C\rho^\varepsilon
		\leq C,
	\end{align*}
	which implies that
	\[
	\int_0^{2/\rho}
	|\varphi(t)|
	\Big|
	H_\ell(\rho t)
	-\sum_{j=1}^{p-1}
	b_{j,\ell}(\rho t)^{2j}
	\Big|dt
	\leq C(\rho|\ell|)^m.
	\]
	
	Combining these estimates and using
	$A_\rho\geq1/2$, we obtain
	$\bigl|1-\widehat\psi_\rho(\ell)\bigr|
	\leq C(\rho|\ell|)^m$.
	Finally, the uniform $L_1$ bound yields
	$\bigl|1-\widehat\psi_\rho(\ell)\bigr|
	\leq 1+\|\psi_\rho\|_{L_1(\T)}
	\leq C$, which completes the proof.
\end{proof}

\begin{remark}
	The approach in Theorem~\ref{thm:jfg-normalized-circle-kernel} can be extended naturally from the circle $\mathbb{T}=\mathbb{S}^1$ to spheres $\mathbb{S}^d$, $d\geq1$. This extension would allow the corresponding
	spherical kernel assumption to be verified through suitable moment and decay conditions on the underlying radial function.
	%Once established, such a result would extend the framework of \cite{SunGaoSun2025} beyond the examples treated there.
\end{remark}

The next  theorem characterizes the $L_2$-energy of the scaled periodic kernel.
\begin{theorem}
	\label{cor:jfg-circle-kernel-energy}
	Under the assumptions of
	Theorem~\ref{thm:jfg-normalized-circle-kernel}, the kernel $\psi_\rho$ satisfies
	\begin{equation}
		\|\psi_\rho\|_{L_2(\T)}^2
		=
		\frac{\pi}{\rho A_\rho^2}
		\int_0^{2/\rho}
		\frac{|\varphi(t)|^2}
		{\sqrt{1-\rho^2t^2/4}}\,dt
		\asymp\rho^{-1}.
		\label{eq:jfg-circle-energy}
	\end{equation}
\end{theorem}

\begin{proof}
	The identity follows from symmetry and the change of variables used in
	the proof of Theorem~\ref{thm:jfg-normalized-circle-kernel}. To determine the asymptotic behavior of the integral as $\rho\downarrow 0$, first consider $[0,1/\rho]$. After extending the integrand by zero, dominated convergence applies because the denominator is bounded below by $\sqrt{3}/2$ and $\varphi\in L_2(0,\infty)$. On $[1/\rho,2/\rho]$, the decay assumption \eqref{eq:jfg-circle-decay-assumption} and the substitution $u=\rho t$ give
    $$\int_{1/\rho}^{2/\rho}
	\frac{|\varphi(t)|^2}
	{\sqrt{1-\rho^2t^2/4}}\,dt\leq C\rho^{2m+1+2\varepsilon}\int_{1}^2\frac{u^{-2m-2-2\varepsilon}}{\sqrt{1-u^2/4}}\,du\leq C\rho^{2m+1+2\varepsilon}.$$
    Consequently, $$\int_{0}^{2/\rho}
	\frac{|\varphi(t)|^2}
	{\sqrt{1-\rho^2t^2/4}}\,dt\longrightarrow\|\varphi\|_{L_2(0,\infty)}^2.$$ 
	The limit is positive because $\int_0^\infty \varphi(t)dt=1$. Together with $A_\rho\rightarrow 1$, this proves \eqref{eq:jfg-circle-energy}.
\end{proof}

%\begin{remark}
%	For a single normalized nonnegative compactly supported profile, put
%	\(\mu_2:=\int_0^R\varphi(t)t^2\,dt\).  If the profile is not concentrated at
%	zero, then \(\mu_2>0\), and expansion of the numerator and denominator in
%	\eqref{eq:jfg-circle-multiplier} gives, for each fixed \(\ell\neq0\),
%	\[
%	1-\widehat\psi_\rho(\ell)
%	=\frac{\ell^2\mu_2}{2}\rho^2+O_\ell(\rho^4).
%	\]
%	Thus normalization gives second-order consistency but cannot give an order
%	greater than two.  Higher even orders in
%	Proposition~\ref{prop:jfg-compact-moment-profile} therefore use signed
%	dilation combinations.  The theorem proves an upper Fourier-decay estimate;
%	the multiplier may have zeros because of cancellation.
%\end{remark}

\subsection{Tensor-product kernels and convolution}

We construct multivariate kernels by taking tensor products of the univariate kernels introduced above. Given a scale vector
$\brho=(\rho_1,\ldots,\rho_d)\in(0,\rho_0]^d$, we define the full tensor-product kernel by
\begin{equation}
	\Psi_{\brho}^{(d)}(\boldsymbol{x})
	:=\prod_{j=1}^d\psi_{\rho_j}(x_j),
	\quad \boldsymbol{x}\in\T^d.
	\label{eq:jfg-tensor-kernel}
\end{equation}
The corresponding convolution operator is
\begin{equation}
	\mathcal C_{\brho}^{(d)}f(\boldsymbol{x})
	:=(f*\Psi_{\brho}^{(d)})(\boldsymbol{x})
	=\int_{\T^d}f(\boldsymbol{y})
	\Psi_{\brho}^{(d)}(\boldsymbol{x}-\boldsymbol{y})
	\,d\boldsymbol{y}.
	\label{eq:jfg-tensor-convolution}
\end{equation}
When the dimension is fixed, we omit the superscript $d$.

%The full tensor product provides the natural multivariate extension of the
%univariate convolution model. Its squared $L_2$-norm, however, is the product
%of the squared $L_2$-norms of its univariate factors and may therefore grow
%rapidly with the dimension. To mitigate this tensor-product growth,
%Section~\ref{sec:journal-direct-anova} replaces the full tensor kernel with a
%weight-adapted, finite-order ANOVA kernel before discretization.

The following two estimates provide the principal bounds required for the discretization analysis: one for the continuous convolution approximation error and the other for the $L_2$-energy of the tensor-product kernel.

\begin{theorem}
	\label{thm:jfg-continuous-bias}
	Let $s>1/2$ and $f\in H^s_{\mathrm{mix}}(\T^d)$. Suppose that the kernel $\psi_\rho$ given by
	\eqref{eq:jfg-circle-kernel} satisfies Assumption~\ref{ass:jfg-kernel}. Let $a:=\min\{m,s\}$. Then, for any fixed
	$d\in\mathbb{N}$, there exists a constant $C_d>0$ independent of $\brho$ such that
	\begin{equation}
		\|f-\mathcal C_{\brho}^{(d)}f\|_{2}
		\leq
		C_d\Big(\sum_{j=1}^d\rho_j^a\Big)
		\|f\|_{H^s_\mathrm{mix}(\T^d)}.
		\label{eq:jfg-continuous-bias}
	\end{equation}
\end{theorem}

\begin{proof}
	The tensor-product structure gives
	$\widehat{\Psi}_{\brho}^{(d)}(\bell)
	=\prod_{j=1}^d\widehat{\psi}_{\rho_j}(\ell_j)$ for $\bell\in\Z^d$. Moreover, Assumption~\ref{ass:jfg-kernel} implies the uniform bound $|\widehat{\psi}_{\rho_j}(\ell_j)|\leq 1+c_1$. Applying the telescoping identity for products gives
	\[
	\Big|
	1-\prod_{j=1}^d\widehat{\psi}_{\rho_j}(\ell_j)
	\Big|
	\leq
	C_d\sum_{j=1}^d
	\Big|1-\widehat{\psi}_{\rho_j}(\ell_j)\Big|
	\leq
	C_d\sum_{j=1}^d
	\min\{1,(\rho_j|\ell_j|)^m\}.
	\]
	Since $a=\min\{m,s\}$ and $0<\rho_j\leq1$, each term satisfies
	\[
	\min\{1,(\rho_j|\ell_j|)^m\}
	\leq \rho_j^a\langle \ell_j\rangle^s.
	\]
	Hence, we can get
	\[
	\Big|
	1-\widehat{\Psi}_{\brho}^{(d)}(\bell)
	\Big|=\Big|
	1-\prod_{j=1}^d\widehat{\psi}_{\rho_j}(\ell_j)
	\Big|
	\leq
	C_d\Big(\sum_{j=1}^d\rho_j^a\Big)
	\prod_{j=1}^d\langle \ell_j\rangle^s.
	\]
	Parseval's identity now yields
	\begin{align*}
		\|f-\mathcal C_{\brho}^{(d)}f\|_{2}^2
		&=
		\sum_{\bell\in\Z^d}
		\Big|1-\widehat{\Psi}_{\brho}^{(d)}(\bell)\Big|^2
		|\widehat f(\bell)|^2 \\
		&\leq
		C_d^2\Big(\sum_{j=1}^d\rho_j^a\Big)^2
		\sum_{\bell\in\Z^d}
		\Big(\prod_{j=1}^d\langle \ell_j\rangle^{2s}\Big)
		|\widehat f(\bell)|^2 \\
		&=
		C_d^2\Big(\sum_{j=1}^d\rho_j^a\Big)^2
		\|f\|_{H^s_\mathrm{mix}(\T^d)}^2.
	\end{align*}
	This completes the proof.
\end{proof}

\begin{theorem}
	\label{thm:jfg-kernel-energy}
	Suppose that Assumption~\ref{ass:jfg-kernel1} holds with $\kappa>1/4$. Then, for
	every fixed $d\in\mathbb{N}$, there exists a constant $C_d>0$ independent of $\brho$ such that
	\begin{equation}
		\|\Psi_{\brho}^{(d)}\|_{2}^2
		\leq C_d\prod_{j=1}^d\rho_j^{-1}.
		\label{eq:jfg-kernel-l2-upper}
	\end{equation}
\end{theorem}

\begin{proof}
	Since $\kappa>1/4$, we have
	\[
	\sum_{\ell\in\Z}(1+\rho|\ell|)^{-4\kappa}
	\leq C_\kappa\rho^{-1}.
	\]
	Assumption~\ref{ass:jfg-kernel1} implies
	$\sum_{\ell\in\Z}
	|\widehat{\psi}_\rho(\ell)|^2
	\leq C_\kappa\rho^{-1}$.
	In particular, the Fourier coefficients of $\psi_\rho$ belong to $\ell_2(\Z)$. By the Riesz--Fischer theorem, there is a function in $L_2(\T)$ with these Fourier coefficients. Uniqueness of Fourier coefficients in $L_1(\T)$ then shows that $\psi_\rho\in L_2(\T)$.
	Parseval's identity yields
	\[
	\|\psi_\rho\|_{L_2(\T)}^2
	=\sum_{\ell\in\Z}|\widehat{\psi}_\rho(\ell)|^2
	\leq C_\kappa\rho^{-1}.
	\]
	Finally, the tensor-product structure and Fubini's theorem give
	\[
	\|\Psi_{\brho}^{(d)}\|_{2}^2
	=\prod_{j=1}^d\|\psi_{\rho_j}\|_{L_2(\T)}^2
	\leq C_d\prod_{j=1}^d\rho_j^{-1},
	\]
	which completes the proof. 
\end{proof}

\section{Deterministic and stochastic quasi-interpolation on \texorpdfstring{\(\T^d\)}{the d-dimensional torus}}

To obtain a computable approximation, we replace the convolution integral by a suitable quadrature rule. We consider two discretization strategies: a deterministic rule on a uniform full-grid and a stochastic rule based on random sampling. The uniform rule exploits the periodic tensor-product structure, whereas random sampling gives a Monte Carlo approximation of the convolution operator. Periodic full-grid and sparse grid quasi-interpolants based on generalized Gaussian kernels have been studied in \cite{GaoHuSunWang2026}. Here, we consider the more general class of kernels satisfying the moment and multiplier conditions given in Sections~\ref{sec:jfg-preliminaries} and~\ref{sec:periodic_convolution}, and derive Sobolev space error estimates for both deterministic and stochastic discretizations.

\subsection{Deterministic full-grid quasi-interpolation}
\label{subsec:jfg-full-grid}

We first discretize the convolution with the isotropic tensor-product kernel
\(\Psi_\rho^{(d)}:=\Psi_{(\rho,\ldots,\rho)}^{(d)}\)
on a uniform grid. For \(n\in\mathbb N\), let
$$G_n^{(d)}:=\Big\{0,\frac{1}{n},\ldots,\frac{n-1}{n}\Big\}^d,\quad
M:=|G_n^{(d)}|=n^d.$$
Then, the associated full-grid quasi-interpolant is
defined by
\begin{equation}
	\mathcal Q_{n,\rho}f(\boldsymbol x)
	:=
	\frac{1}{n^d}
	\sum_{\boldsymbol y\in G_n^{(d)}}
	f(\boldsymbol y)
	\Psi_\rho^{(d)}(\boldsymbol x-\boldsymbol y).
	\label{eq:jfg-full-grid-rule}
\end{equation}

%The theorem below establishes an error estimate for the full-grid quasi-interpolant \eqref{eq:jfg-full-grid-rule}. %The proof is given in Appendix~\ref{Appendix1}.

\begin{theorem}
	\label{thm:jfg-full-grid-benchmark}
	Suppose that Assumptions~\ref{ass:jfg-kernel}
	and~\ref{ass:jfg-kernel1} hold with \(\kappa>1/4\). Let \(s>1/2\), \(a:=\min\{m,s\}\), and $0<\rho\leq \rho_0$. Let $Q_{n,\rho}f$ be defined by \eqref{eq:jfg-full-grid-rule}. Then, for any $f\in H^s_{\mathrm{mix}}(\T^d)$, there exists a constant \(C_{d}>0\) such that
	\begin{equation}
		\|f-\mathcal Q_{n,\rho}f\|_{2}
		\leq
		C_d
		\left\{
		\rho^a+n^{-s}+(n\rho)^{-2\kappa}
		\right\}
		\|f\|_{H^s_{\mathrm{mix}}(\T^d)}.
		\label{eq:jfg-full-grid-benchmark}
	\end{equation}
\end{theorem}
\begin{proof}
	See Appendix~\ref{Appendix1}.
\end{proof}

The error bound consists of three terms: the convolution error, the uniform-grid sampling error, and the high-frequency aliasing associated with
the kernel. Balancing these terms yields the following convergence rate for the full-grid approximation.

\begin{corollary}
	\label{cor:jfg-full-grid-balanced-rate}
	Under the assumptions of
	Theorem~\ref{thm:jfg-full-grid-benchmark}, let
	$\rho\asymp n^{-2\kappa/(a+2\kappa)}$ and set
	$r:=2\kappa a/(a+2\kappa)$. Then, for any $f\in H^s_\mathrm{mix}(\T^d)$ and sufficiently large $n$, there exists a constant \(C_d>0\) independent of \(n\) such that
	\begin{equation}
		\|f-\mathcal Q_{n,\rho}f\|_2
		\leq C_d\, n^{-r}\|f\|_{H^s_\mathrm{mix}(\T^d)}.
		\label{eq:jfg-full-grid-balanced-rate}
	\end{equation}
\end{corollary}

%\begin{proof}
%	The choice of $\rho_n$ balances the convolution-bias term
%	$\rho_n^a$ with the quadrature term $(n\rho_n)^{-2\kappa}$. Indeed,
%	\[
%	\rho_n^a
%	=n^{-2\kappa a/(a+2\kappa)}
%	=n^{-r},
%	\]
%	and, since $n\rho_n=n^{a/(a+2\kappa)}$,
%	\[
%	(n\rho_n)^{-2\kappa}
%	=n^{-2\kappa a/(a+2\kappa)}
%	=n^{-r}.
%	\]
%	Moreover, $r<a\leq s$, so the remaining term in
%	Theorem~\ref{thm:jfg-full-grid-benchmark} satisfies $n^{-s}\leq n^{-r}$.
%	The relations $\rho_n\to0$ and $n\rho_n\to\infty$ further show that
%	$\rho_n\leq\rho_0$ and $n\rho_n\geq4$ for all sufficiently large $n$.
%	Applying Theorem~\ref{thm:jfg-full-grid-benchmark} therefore gives
%	\[
%	\|f-\mathcal Q_{n,\rho_n}^{\mathrm{fg}}f\|_2
%	\leq C_d n^{-r}\|f\|_{s,\mathrm{mix}}.
%	\]
%	Finally, because the full tensor grid contains $M=n^d$ points, we have
%	$n^{-r}=M^{-r/d}$, which proves
%	\eqref{eq:jfg-full-grid-balanced-rate}.
%\end{proof}

%\begin{remark}
%	The factor \(1/d\) in the exponent of
%	\eqref{eq:jfg-full-grid-balanced-rate} is a direct consequence of the
%	cardinality relation \(M=n^d\). Thus, although the full-grid rule is
%	straightforward to implement and remains useful in low fixed
%	dimensions, both its sampling complexity and its convergence rate
%	exhibit the curse of dimensionality associated with tensor-product
%	grids.
%\end{remark}

\subsection{Stochastic Monte Carlo quasi-interpolation}
\label{subsec:jlrn-monte-carlo}
We first formulate the Monte Carlo discretization for a general periodic kernel  \(\calK\in L_2(\T^d)\) satisfying \(\widehat{\calK}(\boldsymbol 0)=1\). For a continuous function \(f\), we define the convolution operator by
\begin{equation}\label{eq:ConvOpt_Sto}
    \mathcal C_{\calK}f(\boldsymbol x):=f*\calK(\boldsymbol x)=\int_{\T^d}f(\boldsymbol y)\calK(\boldsymbol x-\boldsymbol y)d\boldsymbol y.
\end{equation}
Let \(\boldsymbol X\) be uniformly distributed on \(\T^d\), and let \(\boldsymbol X_1,\ldots,\boldsymbol X_M\) be independent copies of \(\boldsymbol X\) and \(f(\boldsymbol X_1),\ldots,f(\boldsymbol X_M)\) be the sampled data. Approximating the convolution integral by random sampling defines the Monte Carlo quasi-interpolant (MCQI):
\begin{equation*}
    \mathcal Q_{M,\calK}f(\boldsymbol x)
    :=
    \frac{1}{M}
    \sum_{j=1}^M
    f(\boldsymbol X_j)
    \calK(\boldsymbol x-\boldsymbol X_j),
    \quad
    \text{for a.e. }\boldsymbol x\in\T^d.
    \label{eq:jlrn-general-iid-rule}
\end{equation*}
%The MCQI scheme gives the approximant directly, without solving a linear system.

%The MCQI mean-square error has an exact decomposition into the deterministic convolution bias and the sampling variance.

\begin{theorem}
\label{thm:jlrn-iid-risk}
Let $\mathcal C_{\calK}$ be defined by \eqref{eq:ConvOpt_Sto} with
\(\widehat{\calK}(\boldsymbol 0)=1\). Then, for any
\(f\in C(\T^d)\), $\E\mathcal Q_{M,\calK}f
    =\mathcal C_{\calK}f$
as a Bochner expectation in \(L_2(\T^d)\). Moreover, we have
\begin{align}
    \E
    \bigl\|
    f-\mathcal Q_{M,\calK}f
    \bigr\|_2^2
    &=
    \|f-\mathcal C_{\calK}f\|_2^2
    +
    \frac{1}{M}
    \big(
    \|f\|_2^2\|\calK\|_2^2
    -
    \|\mathcal C_{\calK}f\|_2^2
    \big).
    \label{eq:jlrn-iid-exact}
\end{align}
\end{theorem}

\begin{proof}
For \(j=1,\ldots,M\), define the \(L_2(\T^d)\)-valued random variable
\[
\boldsymbol Z_j:=f(\boldsymbol X_j)\calK(\,\cdot-\boldsymbol X_j).
\]
The map $\boldsymbol y\mapsto\calK(\,\cdot-\boldsymbol y)$
is continuous from \(\T^d\) into \(L_2(\T^d)\), so \(\boldsymbol Z_j\) is strongly measurable. Since $\|\boldsymbol Z_j\|_2\leq\|f\|_\infty\|\calK\|_2$ almost surely, \(\boldsymbol Z_j\) is Bochner integrable. Uniformity of
\(\boldsymbol X_j\) and Fubini's theorem then give, for a.e.
\(\boldsymbol x\in\T^d\),
\[
\E \boldsymbol Z_j(\boldsymbol x)
=
\int_{\T^d}
f(\boldsymbol y)
\calK(\boldsymbol x-\boldsymbol y)
\,d\boldsymbol y
=
\mathcal C_{\calK}f(\boldsymbol x).
\]
Consequently, $\E \boldsymbol Z_j=\mathcal C_{\calK}f$ in $L_2(\T^d)$.
Furthermore, a direct computation gives
\begin{align*}
    \E\|\boldsymbol Z_j\|_2^2
    &=
    \E\Big[
    |f(\boldsymbol X_j)|^2
    \|\calK(\,\cdot-\boldsymbol X_j)\|_2^2
    \Big] =
    \|f\|_2^2\|\calK\|_2^2.
\end{align*}
It follows that the variance of $\boldsymbol Z_j$ is
\begin{equation}
    \E\|\boldsymbol Z_j-\E \boldsymbol Z_j\|_2^2=
\|f\|_2^2\|\calK\|_2^2
-\|\mathcal C_{\calK}f\|_2^2=:\sigma_{f,\calK}^2.
    \label{eq:jlrn-single-sample-variance}
\end{equation}

The centered random variables \(\boldsymbol Z_j-\E \boldsymbol Z_j\) are independent, so their
off-diagonal expected inner products vanish. Therefore, we get
\begin{align*}
    \E
    \Big\|
    \frac{1}{M}
    \sum_{j=1}^M
    (\boldsymbol Z_j-\E \boldsymbol Z_j)
    \Big\|_2^2
    =
    \frac{1}{M^2}
    \sum_{j=1}^M
    \E \|\boldsymbol Z_j-\E \boldsymbol Z_j\|_2^2 =
    \frac{1}{M}\sigma_{f,\calK}^2.
\end{align*}

Finally, we decompose the error as
\[
f-\mathcal Q_{M,\calK}f
=
f-\mathcal C_{\calK}f
-
\frac{1}{M}
\sum_{j=1}^M
(\boldsymbol Z_j-\E \boldsymbol Z_j).
\]
Expanding the squared norm gives the desired result because the cross term vanishes. 
\end{proof}

To derive a concentration bound, we apply McDiarmid's bounded differences inequality \cite{McDiarmid1989}.

\begin{lemma}[Bounded differences inequality]
\label{lem:jlrn-bounded-differences}
Let \(\Omega\) be a measurable space, let
\(F:\Omega^M\to\mathbb R\) be measurable, and suppose that there exist
constants \(c_1,\ldots,c_M\geq0\), not all zero, such that
\[
	\bigl|
	F(x_1,\ldots,x_j,\ldots,x_M)
	-
	F(x_1,\ldots,x_j',\ldots,x_M)
	\bigr|
	\leq c_j
\]
for every \(j\in\{1,\ldots,M\}\) and all
\(x_1,\ldots,x_M,x_j'\in\Omega\). If \(\boldsymbol X_1,\ldots,\boldsymbol X_M\) are
independent \(\Omega\)-valued random variables, then, for every
\(\epsilon>0\),
\begin{equation}
	\mathbb P\Big\{
		F(\boldsymbol X_1,\ldots,\boldsymbol X_M)-\mathbb EF(\boldsymbol X_1,\ldots,\boldsymbol X_M)
		\geq\epsilon
	\Big\}
	\leq
	\exp\Big(
		-\frac{2\epsilon^2}{\sum_{j=1}^M c_j^2}
	\Big).
	\label{eq:jlrn-mcdiarmid-one-sided}
\end{equation}
Moreover, we have
\begin{equation}
	\mathbb P\Big\{
		\big|
		F(\boldsymbol X_1,\ldots,\boldsymbol X_M)-\mathbb EF(\boldsymbol X_1,\ldots,\boldsymbol X_M)
		\big|
		\geq\epsilon
	\Big\}
	\leq
	2\exp\Big(
		-\frac{2\epsilon^2}{\sum_{j=1}^M c_j^2}
	\Big).
	\label{eq:jlrn-mcdiarmid-two-sided}
\end{equation}
\end{lemma}

\begin{theorem}
\label{thm:jlrn-iid-concentration}
Let the assumptions of Theorem~\ref{thm:jlrn-iid-risk} hold, and define
\(B_{f,\mathcal K}:=\|f\|_\infty\|\mathcal K\|_2\). If
\(B_{f,\mathcal K}>0\), then, for every \(\epsilon>0\),
\begin{equation}
	\mathbb P\Big\{
		\bigl\|f-\mathcal Q_{M,\mathcal K}f\bigr\|_2
		\geq
		\|f-\mathcal C_{\mathcal K}f\|_2
		+
		\frac{\sigma_{f,\mathcal K}}{\sqrt M}
		+
		\epsilon
	\Big\}\leq\exp\Big(-\frac{M\epsilon^2}{2B_{f,\mathcal K}^2}\Big).
	\label{eq:jlrn-iid-tail}
\end{equation}
If \(B_{f,\mathcal K}=0\), then \(f=0\), and the approximation error
vanishes almost surely.
\end{theorem}

\begin{proof}
Let \(\boldsymbol Z_1,\ldots,\boldsymbol Z_M\) be the random \(L_2\)-valued functions
introduced in the proof of Theorem~\ref{thm:jlrn-iid-risk}, and define
\[
	S_M
	:=
	\frac{1}{M}\sum_{j=1}^M\bigl(\boldsymbol Z_j-\mathbb E \boldsymbol Z_j\bigr).
\]
Translation invariance of the \(L_2\)-norm gives almost surely,
\[
	\|\boldsymbol Z_j\|_2
	=
	|f(\boldsymbol X_j)|\cdot
	\bigl\|\mathcal K(\,\cdot-\boldsymbol X_j)\bigr\|_2
	\leq
	\|f\|_\infty\|\mathcal K\|_2
	=
	B_{f,\mathcal K}.
\]

Regard \(\|S_M\|_2\) as a function of the independent sampling sites
\(\boldsymbol X_1,\ldots,\boldsymbol X_M\). Replacing
\(\boldsymbol X_j\) by an arbitrary \(\boldsymbol X_j'\) changes
\(S_M\) by
	$\frac{1}{M}
	\big[
		f(\boldsymbol X_j)
		\mathcal K(\,\cdot-\boldsymbol X_j)
		-
		f(\boldsymbol X_j')
		\mathcal K(\,\cdot-\boldsymbol X_j')
	\big]$.
The \(L_2\)-norm of this difference is bounded by
\(2B_{f,\mathcal K}/M\), so we obtain
\[
	\Bigl|\|S_M\|_2-\|S_M'\|_2\Bigr|
	\leq
	\|S_M-S_M'\|_2
	\leq
	\frac{2B_{f,\mathcal K}}{M}.
\]
Applying Lemma~\ref{lem:jlrn-bounded-differences} with
\(c_j=2B_{f,\mathcal K}/M\) yields
\begin{equation}
	\mathbb P\Big\{
		\|S_M\|_2\geq\mathbb E\|S_M\|_2+\epsilon
	\Big\}\leq\exp\Big(-\frac{M\epsilon^2}{2B_{f,\mathcal K}^2}\Big).
	\label{eq:jlrn-bounded-differences}
\end{equation}
Furthermore, Jensen's inequality and the variance identity in Theorem~\ref{thm:jlrn-iid-risk} give
\[
	\mathbb E\|S_M\|_2
	\leq
	\Bigl(\mathbb E\|S_M\|_2^2\Bigr)^{1/2}
	=
	\frac{\sigma_{f,\mathcal K}}{\sqrt M}.
\]
The triangle inequality
	$\bigl\|f-\mathcal Q_{M,\mathcal K}^{\mathrm{iid}}f\bigr\|_2
	\leq
	\|f-\mathcal C_{\mathcal K}f\|_2+\|S_M\|_2$ shows that
the event on the left hand side of \eqref{eq:jlrn-iid-tail} is
contained in
\[
	\left\{
		\|S_M\|_2
		\geq
		\frac{\sigma_{f,\mathcal K}}{\sqrt M}+\epsilon
	\right\}
	\subseteq
	\Big\{
		\|S_M\|_2
		\geq
		\mathbb E\|S_M\|_2+\epsilon
	\Big\}.
\]
Estimate \eqref{eq:jlrn-iid-tail} then follows from
\eqref{eq:jlrn-bounded-differences}.

If \(B_{f,\mathcal K}=0\), then
\(\widehat{\mathcal K}(\boldsymbol 0)=1\) implies
\(\|\mathcal K\|_2>0\), so \(\|f\|_\infty=0\). Hence \(f=0\),
and both the approximation and its error vanish almost surely.
\end{proof}

The following corollary presents a concentration inequality in a general framework.

\begin{corollary}
\label{cor:jlrn-kernel-concentration}
Let \(\mathcal H_d\) be a normed function space continuously embedded in
\(C(\T^d)\), and suppose that there exist constants
\(c_{2,d},c_{\infty,d}>0\) such that
\(\|f\|_2\leq c_{2,d}\|f\|_{\mathcal H_d}\) and
\(\|f\|_\infty\leq c_{\infty,d}\|f\|_{\mathcal H_d}\) for any
\(f\in\mathcal H_d\). Let \(\mathcal K\in L_2(\T^d)\) satisfy
\(\widehat{\mathcal K}(\boldsymbol0)=1\) and suppose that
\[
	\|f-\mathcal C_{\mathcal K}f\|_2
	\leq A_{\mathcal K}\|f\|_{\mathcal H_d},
	\quad
	\|\mathcal K\|_2\leq E_{\mathcal K}.
\]
Let \(f\in\mathcal H_d\) satisfy \(\|f\|_{\mathcal H_d}\leq1\). If
\(M\in\mathbb N\) and \(\epsilon>0\) satisfy
\begin{equation}
	A_{\mathcal K}
	+
	\frac{c_{2,d}E_{\mathcal K}}{\sqrt M}
	\leq
	\frac{\epsilon}{2},
	\label{eq:jlrn-kernel-epsilon-condition}
\end{equation}
then we have 
\begin{equation}
	\mathbb P\Big\{
		\bigl\|
		f-\mathcal Q_{M,\mathcal K}f
		\bigr\|_2
		\geq\epsilon
	\Big\}
	\leq
	\exp\Big(
		-\frac{M\epsilon^2}
		{8c_{\infty,d}^2E_{\mathcal K}^2}
	\Big).
	\label{eq:jlrn-kernel-concentration}
\end{equation}
%Both the condition and the tail estimate are dimension-independent whenever \(c_{2,d}\), \(c_{\infty,d}\), \(A_\theta\), and \(E_\theta\) are bounded uniformly in \(d\).
\end{corollary}

\begin{proof}
For \(\|f\|_{\mathcal H_d}\leq1\), the variance defined by \eqref{eq:jlrn-single-sample-variance} satisfies
\[
	\sigma_{f,\mathcal K}
	\leq
	\|f\|_2\|\mathcal K\|_2
	\leq
	c_{2,d}E_{\mathcal K},
\]
and
\(B_{f,\mathcal K}\leq
c_{\infty,d}E_{\mathcal K}\). The convolution error estimate and
\eqref{eq:jlrn-kernel-epsilon-condition} lead to
\[
	\|f-\mathcal C_{\mathcal K}f\|_2
	+
	\frac{\sigma_{f,\mathcal K}}{\sqrt M}
	\leq
	\frac{\epsilon}{2}.
\]
It follows that
\[
	\left\{
		\bigl\|
		f-\mathcal Q_{M,\mathcal K}f
		\bigr\|_2
		\geq\epsilon
	\right\}
	\subseteq
	\left\{
		\bigl\|
		f-\mathcal Q_{M,\mathcal K}f
		\bigr\|_2
		\geq
		\|f-\mathcal C_{\mathcal K}f\|_2
		+
		\frac{\sigma_{f,\mathcal K}}{\sqrt M}
		+
		\frac{\epsilon}{2}
	\right\}.
\]
If \(B_{f,\mathcal K}>0\), we apply Theorem
\ref{thm:jlrn-iid-concentration} with deviation
\(\epsilon/2\) to derive
\begin{align*}
	\mathbb P\Big\{
		\bigl\|
		f-\mathcal Q_{M,\mathcal K}f
		\bigr\|_2
		\geq\epsilon
	\Big\}
	\leq
	\exp\Big(
		-\frac{M\epsilon^2}
		{8B_{f,\mathcal K}^2}
	\Big)
	\leq
	\exp\left(
		-\frac{M\epsilon^2}
		{8c_{\infty,d}^2E_{\mathcal K}^2}
	\right).
\end{align*}
If \(B_{f,\mathcal K}=0\), the approximation error vanishes almost surely by Theorem~\ref{thm:jlrn-iid-concentration}.
\end{proof}

Specializing these estimates to the isotropic tensor-product
kernel \(\Psi_\rho^{(d)}\) from \eqref{eq:jfg-tensor-kernel} gives the Monte Carlo counterpart of the deterministic
full-grid construction in Section~\ref{subsec:jfg-full-grid}.

\begin{corollary}
\label{cor:jlrn-iid-tensor-baseline}
Under Assumptions~\ref{ass:jfg-kernel}
and~\ref{ass:jfg-kernel1} with \(\kappa>1/4\), let \(s>1/2\), set \(a:=\min\{m,s\}\), and fix \(d\in\mathbb N\). Then there exist constants \(C_d,c_d>0\), independent of \(f\), \(M\), \(\rho\), and \(\epsilon\), such that for any \(f\in H^s_{\mathrm{mix}}(\T^d)\) with
\(\|f\|_{s,\mathrm{mix}}\leq1\), the following concentration estimate holds:

For any \(M\in\mathbb N\), \(0<\rho\leq\rho_0\), and \(\epsilon>0\) satisfying
\begin{equation}
	C_d\Bigl(
		\rho^a+M^{-1/2}\rho^{-d/2}
	\Bigr)
	\leq
	\frac{\epsilon}{2},
	\label{eq:jlrn-iid-tensor-epsilon-condition}
\end{equation}
we have
\begin{equation}
	\mathbb P\Bigl\{
		\bigl\|
		f-\mathcal Q_{M,\Psi_\rho^{(d)}}f
		\bigr\|_2
		\geq\epsilon
	\Bigr\}
	\leq
	\exp\bigl(-c_dM\rho^d\epsilon^2\bigr).
	\label{eq:jlrn-iid-tensor-epsilon-tail}
\end{equation}

In particular, choosing \(\rho\asymp M^{-1/(2a+d)}\) gives the following estimate:
\begin{equation}
	\Bigl(
		\mathbb E
		\bigl\|
		f-\mathcal Q_{M,\Psi_{\rho}^{(d)}}f
		\bigr\|_2^2
	\Bigr)^{1/2}
	\leq
	C_dM^{-\frac{a}{2a+d}}.
	\label{eq:jlrn-iid-tensor-rms}
\end{equation}

% Finally, let \(0<\delta<1\), define
% \(L_\delta:=1+\sqrt{2\log(1/\delta)}\), and set
% \[
% 	\rho_{M,\delta}
% 	:=
% 	\Bigl(\frac{L_\delta^2}{M}\Bigr)^{1/(2a+d)}.
% \]
% If \(M\geq L_\delta^2\rho_0^{-(2a+d)}\), then, with probability at
% least \(1-\delta\),
% \begin{equation}
% 	\bigl\|
% 	f-\mathcal Q_{M,\Psi_{\rho_{M,\delta}}^{(d)}}^{\mathrm{iid}}f
% 	\bigr\|_2
% 	\leq
% 	C_d
% 	\Bigl(
% 		\frac{L_\delta^2}{M}
% 	\Bigr)^{a/(2a+d)}.
% 	\label{eq:jlrn-iid-tensor-concentration}
% \end{equation}
\end{corollary}

\begin{proof}
By Theorem \ref{thm:jfg-continuous-bias} and Theorem \ref{thm:jfg-kernel-energy}, we obtain
\begin{equation}
	\|f-\mathcal C_{\Psi_\rho^{(d)}}f\|_2
	\leq
	C_d\rho^a\|f\|_{s,\mathrm{mix}},
	\quad
	\|\Psi_\rho^{(d)}\|_2
	\leq
	C_d\rho^{-d/2}.
	\label{eq:jlrn-tensor-bias-energy}
\end{equation}
Since \(s>1/2\), the mixed
Sobolev embedding gives
\[
	\|f\|_2+\|f\|_\infty
	\leq
	C_d\|f\|_{s,\mathrm{mix}}.
\]
For \(\|f\|_{s,\mathrm{mix}}\leq1\), these estimates imply
\begin{equation}
	\sigma_{f,\Psi_\rho^{(d)}}
	\leq
	\|f\|_2\|\Psi_\rho^{(d)}\|_2
	\leq
	C_d\rho^{-d/2},
	\quad
	B_{f,\Psi_\rho^{(d)}}^2
	\leq
	D_d\rho^{-d}
	\label{eq:jlrn-tensor-variance-bounds}
\end{equation}
for some \(D_d>0\).

Enlarging \(C_d\) if necessary,
we obtain from \eqref{eq:jlrn-iid-tensor-epsilon-condition} that
\[
	\|f-\mathcal C_{\Psi_\rho^{(d)}}f\|_2
	+
	\frac{\sigma_{f,\Psi_\rho^{(d)}}}{\sqrt M}
	\leq
	\frac{\epsilon}{2}.
\]
Thus, we get 
\[
	\Bigl\{
		\bigl\|
		f-\mathcal Q_{M,\Psi_\rho^{(d)}}f
		\bigr\|_2
		\geq\epsilon
	\Bigr\}
	\subseteq
	\Bigl\{
		\bigl\|
		f-\mathcal Q_{M,\Psi_\rho^{(d)}}f
		\bigr\|_2
		\geq
		\|f-\mathcal C_{\Psi_\rho^{(d)}}f\|_2
		+
		\frac{\sigma_{f,\Psi_\rho^{(d)}}}{\sqrt M}
		+
		\frac{\epsilon}{2}
	\Bigr\}.
\]
Applying Theorem~\ref{thm:jlrn-iid-concentration} and
\eqref{eq:jlrn-tensor-variance-bounds}, we arrive at
\[
	\mathbb P\Bigl\{
		\bigl\|
		f-\mathcal Q_{M,\Psi_\rho^{(d)}}f
		\bigr\|_2
		\geq\epsilon
	\Bigr\}
	\leq
	\exp\Bigl(
		-\frac{M\rho^d\epsilon^2}{8D_d}
	\Bigr),
\]
which proves \eqref{eq:jlrn-iid-tensor-epsilon-tail} with
\(c_d=(8D_d)^{-1}\).

For the mean-square estimate, Theorem~\ref{thm:jlrn-iid-risk}
together with \eqref{eq:jlrn-tensor-bias-energy} yields
\begin{equation}
	\Bigl(
		\mathbb E
		\bigl\|
		f-\mathcal Q_{M,\Psi_\rho^{(d)}}f
		\bigr\|_2^2
	\Bigr)^{1/2}
	\leq
	C_d\Bigl(
		\rho^a+M^{-1/2}\rho^{-d/2}
	\Bigr).
	\label{eq:jlrn-tensor-rms-general}
\end{equation}
For \(\rho\asymp M^{-1/(2a+d)}\), we have
	$\rho^a
	\asymp
	M^{-1/2}\rho^{-d/2}
	\asymp
	M^{-a/(2a+d)}$.
This completes the proof.
\end{proof}

%The convergence exponent $a/(2a+d)$ in
%\eqref{eq:jlrn-iid-tensor-rms} and
%\eqref{eq:jlrn-iid-tensor-concentration} deteriorates as the dimension
%increases. Independent sampling removes the cardinality constraint $M=n^d$
%associated with a tensor-product grid, but it does not remove the
%dimension dependence generated by the complete-tensor kernel. Indeed, its
%energy estimate contains the factor $\rho^{-d}$:
%\[
%\|\Psi_\rho^{(d)}\|_2^2
%\leq
%C_d\rho^{-d}.
%\]
%Consequently, decreasing $\rho$ reduces the convolution bias while
%simultaneously increasing the sampling variance at a rate that depends
%exponentially on the dimension. In the next section, we will develop the ANOVA constructions, which alleviate this effect by restricting
%the kernel to interactions of bounded order. 

\section{Weight-adapted ANOVA quasi-interpolation}
\label{sec:journal-direct-anova}

For a full tensor-product kernel, the squared \(L_2\)-norm is the product of the univariate kernel energies and can therefore grow rapidly with the ambient dimension. We control this growth by
retaining only those coordinate interactions permitted by the product weights and a prescribed maximal interaction order. The resulting kernel incorporates both coordinate importance and the
interaction structure of the function class. The associated stochastic quasi-interpolant remains an average of kernel translates. Under suitable weight assumptions, however, its sampling variance
admits a bound that depends on the maximal interaction order but is uniform in the ambient dimension.

Throughout this section, suppose that Assumption \ref{ass:jfg-kernel} and Assumption \ref{ass:jfg-kernel1} hold, and
\begin{equation}\label{eq:G_a}
    s>1/2,~a=\min\{m,s\},~G_a:=\sum_{j=1}^\infty \gamma_j^{1/(2a)}<\infty.
\end{equation}

\subsection{Construction of the weight-adapted ANOVA kernel}

Fix $q\in\mathbb N_0$. Let \(0<\rho\leq\rho_0\), and set the threshold
\(\tau_\rho:=(\rho/\rho_0)^{2a}\). We define the active coordinate set and
the family of admissible interactions by
\begin{equation}
	\begin{aligned}
		J_\rho^{(d)}
		&:=
		\{j\in[d]:\gamma_j\geq\tau_\rho\},\\
		\mathfrak U_{q,\rho}^{(d)}
		&:=
		\left\{
		\boldsymbol u\subseteq[d]:
		1\leq|\boldsymbol u|\leq q,\ 
		\gamma_{\boldsymbol u}\geq\tau_\rho
		\right\}.
	\end{aligned}
	\label{eq:jadqi-active-set}
\end{equation}
%Every admissible support is contained in the active coordinate set because \(\gamma_{\boldsymbol u}\leq\gamma_j\) if \(j\in\boldsymbol u\).
For
\(j\in J_\rho^{(d)}\), we have
\(\gamma_j^{1/(2a)}\geq\rho/\rho_0\) and
	$|J_\rho^{(d)}|
	\leq
	\frac{\rho_0G_a}{\rho}$.
We define the anisotropic scale
\begin{equation}
	\rho_j
	:=
	\rho\gamma_j^{-1/(2a)},\quad j\in J_\rho^{(d)}.
	\label{eq:jadqi-anisotropic-scale}
\end{equation}
The definition of the active coordinate set ensures that
\(\rho\leq\rho_j\leq\rho_0\). Thus, less influential coordinates with smaller weights are approximated using coarser scales.

Let \(\psi_{\rho_j}^{\circ}:=\psi_{\rho_j}-1\). The weight-adapted ANOVA kernel is defined by
\begin{equation}
	\Phi_{q,\rho}^{(d)}(\boldsymbol x)
	:=
	1+
	\sum_{\boldsymbol u\in\mathfrak U_{q,\rho}^{(d)}}
	\prod_{j\in\boldsymbol u}
	\psi_{\rho_j}^{\circ}(x_j).
	\label{eq:jadqi-kernel}
\end{equation}
The corresponding convolution operator is
\begin{equation}
	\mathcal C_{q,\rho}^{(d)}f(\boldsymbol x)
	:=
	f*\Phi_{q,\rho}^{(d)}(\boldsymbol x).
	\label{eq:jadqi-convolution-operator}
\end{equation}

Since \(\int_{\T}\psi_{\rho_j}^{\circ}(x_j)dx_j=0\), terms with distinct coordinate supports lie in mutually orthogonal Fourier--ANOVA subspaces. A related centering principle yields identifiable ANOVA decompositions in reproducing kernel Hilbert spaces \cite{DurrandeEtAl2013}. Here, however, the univariate factors act as convolution filters and are not assumed to be positive definite.

\begin{proposition}
	\label{prop:jdaq-kernel-structure}
	The kernel \(\Phi_{q,\rho}^{(d)}\) defined by \eqref{eq:jadqi-kernel} satisfies \(\widehat\Phi_{q,\rho}^{(d)}(\boldsymbol0)=1\), and its Fourier multiplier is given by
	\begin{equation}
		\widehat\Phi_{q,\rho}^{(d)}(\bell)
		=
		\begin{cases}
			1,
			& \bell=\boldsymbol0,\\[1mm]
			\displaystyle
			\prod_{j\in\supp\bell}
			\widehat\psi_{\rho_j}(\ell_j),
			&
			\supp\bell\in
			\mathfrak U_{q,\rho}^{(d)},\\[3mm]
			0,
			& \text{otherwise}.
		\end{cases}
		\label{eq:jdaq-kernel-multiplier}
	\end{equation}
	Moreover, there exists a constant \(C>0\) independent of \(d\) and \(\rho\) such that
	\begin{align}
		\|\Phi_{q,\rho}^{(d)}\|_2^2
		=
		1+
		\sum_{\boldsymbol u\in\mathfrak U_{q,\rho}^{(d)}}
		\prod_{j\in\boldsymbol u}
		\bigl(\|\psi_{\rho_j}\|_2^2-1\bigr)
		\leq
		C\rho^{-q}.
		\label{eq:jdaq-exact-kernel-energy}
	\end{align}
\end{proposition}

\begin{proof}
   For each \(j\in J_\rho^{(d)}\), the centered factor gives \(\widehat{\psi_{\rho_j}^{\circ}}(0)=0\), whereas \(\widehat{\psi_{\rho_j}^{\circ}}(k)=\widehat{\psi_{\rho_j}}(k)\) for \(k\neq0\). Thus the term indexed by \(\boldsymbol u\) has nonzero Fourier coefficients only at frequencies \(\boldsymbol\ell\) with \(\operatorname{supp}\boldsymbol\ell=\boldsymbol u\). Terms indexed by distinct sets therefore have disjoint Fourier supports and are orthogonal in \(L_2(\mathbb T^d)\). This gives \eqref{eq:jdaq-kernel-multiplier}.

	The same orthogonality also yields
	\[
	\|\Phi_{q,\rho}^{(d)}\|_2^2
	=
	1+
	\sum_{\boldsymbol u\in\mathfrak U_{q,\rho}^{(d)}}
	\prod_{j\in\boldsymbol u}
	\|\psi_{\rho_j}^{\circ}\|_2^2.
	\]
	Since each \(\psi_{\rho_j}\) is normalized, we obtain
	\[
	\|\psi_{\rho_j}^{\circ}\|_2^2
	=
	\|\psi_{\rho_j}-1\|_2^2
	=
	\|\psi_{\rho_j}\|_2^2-1,
	\]
	which proves the identity in
	\eqref{eq:jdaq-exact-kernel-energy}.
	
	By Theorem \ref{thm:jfg-kernel-energy} with $d=1$, we have
	\[
	\|\psi_{\rho_j}\|_2^2-1
	\leq
	C\rho_j^{-1}
	=
	C\rho^{-1}\gamma_j^{1/(2a)}.
	\]
	Enlarging the sum to all subsets of \([d]\) of cardinality at most
	\(q\), we can estimate
	\begin{align*}
		\|\Phi_{q,\rho}^{(d)}\|_2^2
		&\leq
		\sum_{r=0}^q
		\Big(\frac{C}{\rho}\Big)^r
		\sum_{\substack{
				\boldsymbol u\subseteq[d]\\
				|\boldsymbol u|=r}}
		\prod_{j\in\boldsymbol u}
		\gamma_j^{1/(2a)}
		\\
		&\leq
		\sum_{r=0}^q
		\frac{1}{r!}
		\Big(
		\frac{C}{\rho}
		\sum_{j=1}^{\infty}
		\gamma_j^{1/(2a)}
		\Big)^r
		\\
		&=
		\sum_{r=0}^q
		\frac{1}{r!}
		\Big(\frac{CG_a}{\rho}\Big)^r
		\leq
		C\rho^{-q},
	\end{align*}
	where in the second inequality we used the standard bound for elementary symmetric polynomials. 
\end{proof}

The anisotropic scales in \eqref{eq:jadqi-anisotropic-scale} balance the coordinatewise approximation errors against the weights, which yields the following convolution estimate. 

\begin{theorem}
	\label{thm:jdaq-convolution-bias}
	Let
	\(f\in H^s_{\mathrm{Kor},\boldsymbol\gamma,q}(\T^d)\) and
	\(0<\rho\leq\rho_0\). Then there exists a constant \(C>0\) independent of \(d\) and \(\rho\) such that
	\begin{equation}
		\|f-\mathcal C_{q,\rho}^{(d)}f\|_2
		\leq
		C\rho^a\|f\|_{s,\boldsymbol\gamma}.
		\label{eq:jdaq-convolution-bias}
	\end{equation}
\end{theorem}

\begin{proof}
	Fix \(\bell\neq\boldsymbol0\), and write
	\(\boldsymbol u:=\supp\bell\). Since \(f\) has superposition
	order at most \(q\), it suffices to consider
	\(1\leq|\boldsymbol u|\leq q\).
	
	First suppose that
	\(\boldsymbol u\in\mathfrak U_{q,\rho}^{(d)}\). A telescoping
	decomposition of the product in
	\eqref{eq:jdaq-kernel-multiplier} together with
	Assumption~\ref{ass:jfg-kernel} gives
	\begin{equation}
		\Big|
		1-
		\prod_{j\in\boldsymbol u}
		\widehat\psi_{\rho_j}(\ell_j)
		\Big|
		\leq
		C\sum_{j\in\boldsymbol u}
		\rho_j^a|\ell_j|^a.
		\label{eq:jdaq-product-multiplier-error}
	\end{equation}
	Because \(a\leq s\), for each \(j\in\boldsymbol u\),
	\[
	|\ell_j|^a
	\prod_{i\in\boldsymbol u}
	\langle \ell_i\rangle^{-s}
	\leq1.
	\]
	It follows that
	\[
	\prod_{j\in\boldsymbol u}
	\langle \ell_j\rangle^{-s}
	\Big|
	1-
	\widehat\Phi_{q,\rho}^{(d)}(\bell)
	\Big|
	\leq
	C\sum_{j\in\boldsymbol u}\rho_j^a.
	\]
	The anisotropic scale choice now gives
	\begin{align}
		\gamma_{\boldsymbol u}^{1/2}
		\sum_{j\in\boldsymbol u}\rho_j^a
		=
		\rho^a
		\sum_{j\in\boldsymbol u}
		\gamma_{\boldsymbol u}^{1/2}\gamma_j^{-1/2}
		=
		\rho^a
		\sum_{j\in\boldsymbol u}
		\prod_{i\in\boldsymbol u\setminus\{j\}}
		\gamma_i^{1/2}
		\leq
		q\rho^a.
		\label{eq:jdaq-scale-weight-cancellation}
	\end{align}
	Hence, we derive
	\begin{equation}
		\gamma_{\boldsymbol u}^{1/2}
		\prod_{j\in\boldsymbol u}
		\langle \ell_j\rangle^{-s}
		\Big|
		1-
		\widehat\Phi_{q,\rho}^{(d)}(\bell)
		\Big|
		\leq
		C\rho^a.
		\label{eq:jdaq-retained-multiplier-error}
	\end{equation}
	
	Next suppose that \(1\leq|\boldsymbol u|\leq q\), but
	\(\boldsymbol u\notin\mathfrak U_{q,\rho}^{(d)}\). In this case,
	\(\gamma_{\boldsymbol u}<\tau_\rho\) and
	\(\widehat\Phi_{q,\rho}^{(d)}(\bell)=0\). We get
	\[
	\gamma_{\boldsymbol u}^{1/2}
	<
	\tau_\rho^{1/2}
	=
	\Big(\frac{\rho}{\rho_0}\Big)^a.
	\]
	Since
	\(\prod_{j\in\boldsymbol u}\langle \ell_j\rangle^{-s}\leq1\), we again obtain
	\[
	\gamma_{\boldsymbol u}^{1/2}
	\prod_{j\in\boldsymbol u}
	\langle \ell_j\rangle^{-s}
	\Big|
	1-
	\widehat\Phi_{q,\rho}^{(d)}(\bell)
	\Big|
	\leq
	C\rho^a.
	\]
	
	Thus, for every nonzero frequency $\bell$ with
	\(|\supp\bell|\leq q\), we get
	\begin{equation}
		\Big|
		1-
		\widehat\Phi_{q,\rho}^{(d)}(\bell)
		\Big|
		\leq
		C\rho^a
		\gamma_{\supp\bell}^{-1/2}
		\prod_{j\in\supp\bell}
		\langle \ell_j\rangle^s.
		\label{eq:jdaq-uniform-multiplier-error}
	\end{equation}
	Parseval's identity and the superposition assumption lead to
	\begin{align*}
		\|f-\mathcal C_{q,\rho}^{(d)}f\|_2^2
		&=
		\sum_{\bell\in\mathbb Z^d}
		\Big|
		1-
		\widehat\Phi_{q,\rho}^{(d)}(\bell)
		\Big|^2
		|\widehat f(\bell)|^2
		\\
		&\leq
		C\rho^{2a}
		\sum_{\bell\in\mathbb Z^d}
		\gamma_{\supp\bell}^{-1}
		\prod_{j\in\supp\bell}
		\langle \ell_j\rangle^{2s}
		|\widehat f(\bell)|^2
		\\
		&=
		C\rho^{2a}
		\|f\|_{s,\boldsymbol\gamma}^2,
	\end{align*}
	which completes the proof.
\end{proof}

\subsection{Stochastic quasi-interpolation with ANOVA kernels}

Discretizing the convolution integral \eqref{eq:jadqi-convolution-operator} with random samples \(\boldsymbol X_1,\ldots,\boldsymbol X_M\), we obtain the stochastic ANOVA quasi-interpolant
\begin{equation}
	\mathcal Q_{M,q,\rho}^{(d)}f
	:=
	\frac{1}{M}
	\sum_{j=1}^M
	f(\boldsymbol X_j)
	\Phi_{q,\rho}^{(d)}(\,\cdot-\boldsymbol X_j).
	\label{eq:jdaq-noiseless-rule}
\end{equation}
Since \(\Phi_{q,\rho}^{(d)}\) is normalized, the general unbiasedness identity gives
$\E\mathcal Q_{M,q,\rho}^{(d)}f=\mathcal C_{q,\rho}^{(d)}f$.

The following result shows the mean-square error and a dimension-uniform upper bound.

\begin{theorem}
\label{thm:jadqi-dimension-uniform}
Assume that \eqref{eq:G_a} holds. Let \(d,M\in\mathbb N\),
\(f\in H^s_{\mathrm{Kor},\boldsymbol\gamma,q}(\T^d)\), and
\(0<\rho\leq\rho_0\). Then there exists a constant \(C>0\) independent of \(d\), \(M\), \(\rho\), such that
\begin{align}
\E
\bigl\|f-\mathcal Q_{M,q,\rho}^{(d)}f\bigr\|_2^2
&=
\bigl\|f-\mathcal C_{q,\rho}^{(d)}f\bigr\|_2^2
+
\frac{1}{M}
\left(
\|f\|_2^2\|\Phi_{q,\rho}^{(d)}\|_2^2
-
\|\mathcal C_{q,\rho}^{(d)}f\|_2^2
\right)
\notag\\
&\leq
C\left(
\rho^{2a}
+
\frac{\rho^{-q}}{M}
\right)
\|f\|_{s,\boldsymbol\gamma}^2.
\label{eq:jdaq-exact-iid-risk}
\end{align}
\end{theorem}

\begin{proof}
Lemma~\ref{lem:jdaq-uniform-embedding} ensures that \(f\) has a continuous representative. Proposition~\ref{prop:jdaq-kernel-structure} shows that
\(\Phi_{q,\rho}^{(d)}\in L_2(\T^d)\) and
\(\widehat\Phi_{q,\rho}^{(d)}(\boldsymbol0)=1\).
Thus Theorem~\ref{thm:jlrn-iid-risk} applies with
\(\calK=\Phi_{q,\rho}^{(d)}\).

For the upper bound, omit the nonpositive term
\(-M^{-1}\|\mathcal C_{q,\rho}^{(d)}f\|_2^2\).
Theorem~\ref{thm:jdaq-convolution-bias} gives
\[
\|f-\mathcal C_{q,\rho}^{(d)}f\|_2^2
\leq
C\rho^{2a}\|f\|_{s,\boldsymbol\gamma}^2.
\]
Moreover, \(\|f\|_2\leq\|f\|_{s,\boldsymbol\gamma}\) by definition of the weighted Korobov norm, so Proposition~\ref{prop:jdaq-kernel-structure} yields
\[
\frac{1}{M}\|f\|_2^2\|\Phi_{q,\rho}^{(d)}\|_2^2
\leq
\frac{C\rho^{-q}}{M}\|f\|_{s,\boldsymbol\gamma}^2.
\]
Together, these bounds give the inequality in
\eqref{eq:jdaq-exact-iid-risk}.	
\end{proof}

%Balancing the bias and variance gives the following root-mean-square convergence rate.

\begin{corollary}
	\label{cor:jadqi-concentration}
	Assume that \eqref{eq:G_a} holds. Let
	\(d,M\in\mathbb N\), \(0<\rho\leq\rho_0\), \(\epsilon>0\), and
	\(f\in H^s_{\mathrm{Kor},\boldsymbol\gamma,q}(\T^d)\) with
	\(\|f\|_{s,\boldsymbol\gamma}\leq1\). There exist
	constants \(C,c>0\), independent of \(d\), \(M\), \(\rho\),
	\(\epsilon\), and \(f\), with the following property: if
	\begin{equation}
		C\left(
		\rho^a+
		\frac{\rho^{-q/2}}{\sqrt M}
		\right)
		\leq
		\frac{\epsilon}{2},
		\label{eq:jadqi-concentration-condition}
	\end{equation}
	then we have
	\begin{equation}
		\mathbb P\left\{
			\bigl\|
			f-\mathcal Q_{M,q,\rho}^{(d)}f
			\bigr\|_2
			\geq\epsilon
		\right\}
		\leq
		\exp\bigl(-cM\rho^q\epsilon^2\bigr).
		\label{eq:jadqi-concentration}
	\end{equation}
\end{corollary}

\begin{proof}
	According to Corollary~\ref{cor:jlrn-kernel-concentration}, we can take $\mathcal H_d=H^s_{\mathrm{Kor},\boldsymbol\gamma,q}(\T^d)$ and $\calK=\Phi_{q,\rho}^{(d)}$.
	Then, the definition of the weighted Korobov norm and
	Lemma~\ref{lem:jdaq-uniform-embedding} give
	\[
		\|f\|_2\leq\|f\|_{s,\boldsymbol\gamma},
		\quad
		\|f\|_\infty
		\leq
		\exp\Big(\frac{A_sG_a}{2}\Big)
		\|f\|_{s,\boldsymbol\gamma}.
	\]
	Theorem~\ref{thm:jdaq-convolution-bias} and
	Proposition~\ref{prop:jdaq-kernel-structure} supply constants \(C_{\mathrm b},C_\Phi>0\) for which
	\[
		\|f-\mathcal C_{q,\rho}^{(d)}f\|_2
		\leq
		C_{\mathrm b}\rho^a\|f\|_{s,\boldsymbol\gamma},
		\quad
		\|\Phi_{q,\rho}^{(d)}\|_2
		\leq
		C_\Phi\rho^{-q/2}.
	\]
	Furthermore, we employ the following constants in Corollary~\ref{cor:jlrn-kernel-concentration}:
	\[
		c_{2,d}=1,
		\quad
		c_{\infty,d}=\exp\Big(\frac{A_sG_a}{2}\Big),
		\quad
		A_{\calK}=C_{\mathrm b}\rho^a,
		\quad
		E_{\calK}=C_\Phi\rho^{-q/2}.
	\]
	Taking
	\(C\geq\max\{C_{\mathrm b},C_\Phi\}\) makes
	\eqref{eq:jadqi-concentration-condition} sufficient for its threshold condition. The conclusion \eqref{eq:jadqi-concentration} then holds with $c=(8C_\Phi^2\exp(A_sG_a))^{-1}$.
\end{proof}

For the mean-square error, the choice
\(\rho\asymp M^{-1/(2a+q)}\) balances the two terms in
\eqref{eq:jdaq-exact-iid-risk}.

\begin{corollary}
	\label{cor:jadqi-rate}
	Assume that \eqref{eq:G_a} holds. Let
	\(d,M\in\mathbb N\) and set
	\(\rho \asymp M^{-1/(2a+q)}\). Then there exists a constant \(C>0\) independent of \(d\) and \(M\) such that
	\begin{equation}
		\sup_{\substack{
				f\in
				H^s_{\mathrm{Kor},\boldsymbol\gamma,q}(\T^d)\\
				\|f\|_{s,\boldsymbol\gamma}\leq1}}
		\Big(
		\E
		\bigl\|
		f-\mathcal Q_{M,q,\rho}^{(d)}f
		\bigr\|_2^2
		\Big)^{1/2}
		\leq
		CM^{-\frac{a}{2a+q}}.
		\label{eq:jadqi-rate}
	\end{equation}
\end{corollary}

\subsection{Recursive evaluation of the ANOVA kernel}

Since \(0<\gamma_j\leq1\), every nonempty subset of an admissible support is admissible. We use this downward closure to evaluate the ANOVA kernel by traversing a tree with node set
\(\mathfrak U_{q,\rho}^{(d)}\cup\{\emptyset\}\) and root \(\emptyset\).
With \(\gamma_{\emptyset}:=1\) and \(\max\emptyset:=0\), the children of a node \(\boldsymbol u\) are
\[
\boldsymbol u\cup\{j\},
\qquad
j\in J_\rho^{(d)},\quad
j>\max\boldsymbol u,\quad
|\boldsymbol u|<q,\quad
\gamma_{\boldsymbol u}\gamma_j\geq\tau_\rho.
\]
Thus each nonempty support has the unique parent
\(\boldsymbol u\setminus\{\max\boldsymbol u\}\).

For fixed \(\boldsymbol x,\boldsymbol y\in\T^d\), set
\[
g_j:=\psi_{\rho_j}(x_j-y_j)-1,
\qquad j\in J_\rho^{(d)},
\]
and initialize \(P(\emptyset)=1\). Along an edge from \(\boldsymbol u\) to \(\boldsymbol u\cup\{j\}\), update
\[
P(\boldsymbol u\cup\{j\})=P(\boldsymbol u)g_j,
\qquad
\gamma_{\boldsymbol u\cup\{j\}}=\gamma_{\boldsymbol u}\gamma_j.
\]
The traversal accumulates \(P(\boldsymbol u)\) at each visited node.

\begin{proposition}
	\label{prop:jdaq-support-tree-traversal}
	Suppose that the active indices are processed in increasing order and that \((\gamma_j)\) is nonincreasing. For every \(\boldsymbol x,\boldsymbol y\in\T^d\), the depth-first traversal visits each admissible support exactly once and yields
	\[
	\Phi_{q,\rho}^{(d)}(\boldsymbol x-\boldsymbol y)
	=
	1+\sum_{\boldsymbol u\in\mathfrak U_{q,\rho}^{(d)}}P(\boldsymbol u).
	\]
\end{proposition}

\begin{proof}
	Let \(\boldsymbol u=\{j_1<\cdots<j_r\}\in\mathfrak U_{q,\rho}^{(d)}\).
	Because \(0<\gamma_j\leq1\), each \(j_\ell\) is active and every nonempty prefix
	\(\{j_1,\ldots,j_\ell\}\) is admissible. The traversal therefore reaches
	\(\boldsymbol u\) by adjoining \(j_1,\ldots,j_r\) in order. No other path can reach it, since indices are appended in increasing order. Conversely, the cardinality and weight tests ensure that every generated nonroot node is admissible. Hence the nonroot nodes are precisely the members of \(\mathfrak U_{q,\rho}^{(d)}\).
	
	The update rule gives
	\[
	P(\boldsymbol u)
	=
	\prod_{j\in\boldsymbol u}
	\bigl(\psi_{\rho_j}(x_j-y_j)-1\bigr).
	\]
	Summing over the visited nodes and using \eqref{eq:jadqi-kernel}, we obtain
	\[
	P(\emptyset)+
	\sum_{\boldsymbol u\in\mathfrak U_{q,\rho}^{(d)}}P(\boldsymbol u)
	=
	1+
	\sum_{\boldsymbol u\in\mathfrak U_{q,\rho}^{(d)}}
	\prod_{j\in\boldsymbol u}
	\bigl(\psi_{\rho_j}(x_j-y_j)-1\bigr)
	=
	\Phi_{q,\rho}^{(d)}(\boldsymbol x-\boldsymbol y).
	\]
	This completes the proof.
\end{proof}

\begin{remark}
	Once the active indices are available in increasing order, computing the factors \(g_j\) requires at most \(|J_\rho^{(d)}|\) univariate kernel evaluations and \(O(|J_\rho^{(d)}|)\) elementary operations. Each successful extension takes \(O(1)\) further operations. For a fixed node, the nonincreasing weights make the admissible candidate indices an initial segment: after the first failed weight test, no larger index can be admitted. There is one successful extension per admissible support and at most one failed test per visited node. The traversal therefore takes
	\[
	O\bigl(1+|J_\rho^{(d)}|+
	|\mathfrak U_{q,\rho}^{(d)}|\bigr)
	\]
	elementary operations. Storing the factors and a depth-first stack of at most \(q\) constant-size frames requires
	\(O\bigl(1+|J_\rho^{(d)}|+q\bigr)\) auxiliary storage.
\end{remark}

\section{Multilevel weight-adapted ANOVA quasi-interpolation}
\label{sec:mlaqi}

Multilevel methods represent an approximation as a sum of tensor-product differences between successive coordinatewise resolutions. A hyperbolic constraint on the level indices allows one coordinate to be refined without imposing comparable refinement on the others. Related constructions appear in hyperbolic ANOVA approximation \cite{LippertPottsUllrich2023}. We use this uneven allocation of resolution to control the \(L_2\)-energy of the weight-adapted ANOVA kernel. 

In the single-level construction of Section~\ref{sec:journal-direct-anova}, all coordinatewise scales are determined by one global parameter. Here, we assign each coordinate a weight-dependent entry level and retain only tensor-product differences satisfying the hyperbolic constraint. The resulting algebraic convergence is independent of the superposition order \(q\).

%\subsection{Dyadic differences}

Let \((\psi_\rho)_{\rho>0}\) be the normalized kernel family, and fix
\(\alpha\) such that
\begin{equation}
	\frac12<\alpha\leq s,
	\quad
	\alpha<m,
	\quad
	G_\alpha<\infty.
	\label{eq:mlaqi-alpha-assumption}
\end{equation}
At the dyadic scales \(\rho_0 2^{-n}\), define the difference kernels
\begin{equation}
	\vartheta_0:=\psi_{\rho_0}-1,
	\quad
	\vartheta_n
	:=
	\psi_{\rho_0 2^{-n}}
	-
	\psi_{\rho_0 2^{-(n-1)}},
	\quad n\geq1.
	\label{eq:mlaqi-difference-zero}
\end{equation}
For \(n\geq0\) and \(\ell\in\mathbb Z\), write
\(d_n(\ell):=\widehat{\vartheta_n}(\ell)\).
Normalization gives \(d_n(0)=0\) for every \(n\geq0\).
For each \(\ell\neq0\), the multiplier differences telescope, and the
small-scale limit yields
\begin{equation}
	\sum_{n=0}^{N}d_n(\ell)
	=
	\widehat{\psi}_{\rho_0 2^{-N}}(\ell)
	\longrightarrow 1
	\quad\text{as }N\to\infty.
	\label{eq:mlaqi-telescoping-multiplier}
\end{equation}
Thus, \((d_n)_{n\geq0}\) gives a dyadic resolution of the identity on
the nonzero Fourier modes.

We incorporate the coordinate weights into the level selection through
the entry levels
\begin{equation}
	\lambda_j
	:=
	\left\lfloor
	\frac{1}{2\alpha}\log_2\gamma_j^{-1}
	\right\rfloor,
	\quad j\geq1.
	\label{eq:mlaqi-entry-level}
\end{equation}
For \(\boldsymbol u\subseteq[d]\), set
\(\lambda_{\boldsymbol u}:=\sum_{j\in\boldsymbol u}\lambda_j\), with
\(\lambda_{\emptyset}:=0\). The definition of \(\lambda_j\) gives
\begin{equation}
	2^{-\lambda_j}
	\leq
	2\gamma_j^{1/(2\alpha)},
	\quad
	\gamma_j^{1/2}
	\leq
	2^{-\alpha\lambda_j}.
	\label{eq:mlaqi-entry-level-comparison}
\end{equation}
Consequently, coordinates with smaller weights have larger entry levels and enter the construction only at higher global levels.

\subsection{The multilevel ANOVA kernel and convolution}

For \(L\in\mathbb N_0\), define the weight-adapted multilevel ANOVA
kernel by
\begin{equation}
	\Phi_{q,L}^{\mathrm{ml},(d)}(\boldsymbol x)
	:=
	1+
	\sum_{r=1}^{\min\{q,d\}}
	\sum_{\substack{
			\boldsymbol u\subseteq[d]\\
			|\boldsymbol u|=r,\ \lambda_{\boldsymbol u}\leq L}}
	\;
	\sum_{\substack{
			\boldsymbol n\in\mathbb N_0^{\boldsymbol u}\\
			|\boldsymbol n|_1\leq L-\lambda_{\boldsymbol u}}}
	\prod_{j\in\boldsymbol u}
	\vartheta_{n_j}(x_j),
	\label{eq:mlaqi-kernel}
\end{equation}
where \(\boldsymbol n=(n_j)_{j\in\boldsymbol u}\) and
$|\boldsymbol n|_1:=\sum_{j\in\boldsymbol u}n_j$.

The pair \((\boldsymbol u,\boldsymbol n)\) occurs in the sum exactly when $\lambda_{\boldsymbol u}+|\boldsymbol n|_1\leq L$. Here \(\lambda_{\boldsymbol u}\) records the contribution of the coordinate weights, while \(|\boldsymbol n|_1\) is the total local dyadic level. For a fixed interaction \(\boldsymbol u\), the available
level \(L-\lambda_{\boldsymbol u}\) is distributed among its coordinates through the hyperbolic constraint
\(|\boldsymbol n|_1\leq L-\lambda_{\boldsymbol u}\).

The approximation and \(L_2\)-energy estimates use the following univariate bounds. 

\begin{lemma}
	\label{lem:mlaqi-bandpass}
	Let \((\psi_\rho)_{\rho>0}\) satisfy
	Assumption~\ref{ass:jfg-kernel} and Assumption~\ref{ass:jfg-kernel1} with \(\kappa>1/2\). 
	For any \(0<\beta<m\), we have
	\begin{equation}
		\sum_{n=0}^{\infty}
		2^{\beta n}|d_n(\ell)|
		\leq
		C_\beta\langle \ell\rangle^\beta,
		\quad \ell\neq0.
		\label{eq:mlaqi-weighted-bandpass-sum}
	\end{equation}
	If
	\(0<\varepsilon<\min\{m,2\kappa-1\}\), then
	\begin{equation}
		\big|
		\langle\vartheta_n,\vartheta_v\rangle_{L_2(\mathbb T)}
		\big|
		\leq
		C_\varepsilon
		2^{\min\{n,v\}}
		2^{-\varepsilon|n-v|},
		\quad n,v\in\mathbb N_0.
		\label{eq:mlaqi-almost-orthogonality}
	\end{equation}
\end{lemma}
\begin{proof}
	See Appendix \ref{Appendix2}.
\end{proof}

Since each \(\vartheta_n\) has zero mean, it follows that $\widehat{\Phi}_{q,L}^{\mathrm{ml},(d)}(\boldsymbol0)=1$. The associated convolution operator is given by
\begin{equation}
	\mathcal C_{q,L}^{\mathrm{ml},(d)}f
	:=
	f*\Phi_{q,L}^{\mathrm{ml},(d)}.
	\label{eq:mlaqi-convolution}
\end{equation}

To state the multiplier formula, let
\(\emptyset\neq\boldsymbol u\subseteq[d]\), \(B\in\mathbb N_0\), and
$\bell_{\boldsymbol u}=(\ell_j)_{j\in\boldsymbol u}$ for $\ell_j\neq0$. Define the truncated component multiplier by
\begin{equation}
	a_{\boldsymbol u,B}(\bell_{\boldsymbol u})
	:=
	\sum_{\substack{
			\boldsymbol n\in\mathbb N_0^{\boldsymbol u}\\
			|\boldsymbol n|_1\leq B}}
	\prod_{j\in\boldsymbol u}d_{n_j}(\ell_j).
	\label{eq:mlaqi-component-multiplier}
\end{equation}

\begin{proposition}
	\label{prop:mlaqi-multiplier-energy}
	Suppose that \eqref{eq:mlaqi-alpha-assumption} holds, and let
	\(q\in\mathbb N_0\) be fixed. Then we have
	\begin{equation}
		\widehat{\Phi}_{q,L}^{\mathrm{ml},(d)}(\bell)
		=
		\begin{cases}
			1,
			& \bell=\boldsymbol0,\\[1mm]
			a_{\boldsymbol u,L-\lambda_{\boldsymbol u}}
			(\bell_{\boldsymbol u}),
			&
			\begin{array}{l}
				\boldsymbol u=\supp\bell,\\[-1mm]
				1\leq|\boldsymbol u|\leq q,\quad
				\lambda_{\boldsymbol u}\leq L,
			\end{array}\\[4mm]
			0,
			& \text{otherwise}.
		\end{cases}
		\label{eq:mlaqi-multiplier}
	\end{equation}
	Moreover, there exists a constant \(C>0\) independent of \(d\) and \(L\) such that
	\begin{equation}
		\|\Phi_{q,L}^{\mathrm{ml},(d)}\|_2^2
		\leq
		C2^L(L+1)^{q-1}.
		\label{eq:mlaqi-energy}
	\end{equation}
\end{proposition}

\begin{proof}
	See Appendix \ref{Appendix3}.
\end{proof}

We next establish an error estimate for the multilevel convolution operator \eqref{eq:mlaqi-convolution}.

\begin{theorem}
\label{thm:mlaqi-bias}
Suppose that \eqref{eq:mlaqi-alpha-assumption} holds, and let
\(q\in\mathbb N_0\) be fixed. Let 
\(f\in H^s_{\mathrm{Kor},\boldsymbol\gamma,q}(\mathbb T^d)\). There exists a constant \(C>0\),
independent of \(d\), \(L\), and \(f\), such that
\begin{equation}
\|f-\mathcal C_{q,L}^{\mathrm{ml},(d)}f\|_2
\leq
C2^{-\alpha L}
\|f\|_{s,\boldsymbol\gamma}.
\label{eq:mlaqi-bias}
\end{equation}
\end{theorem}

\begin{proof}
If \(q=0\), the estimate follows easily. Assume henceforth that \(q\geq1\). Fix
\(\boldsymbol\ell\neq\boldsymbol0\) and set
\(\boldsymbol u:=\supp\boldsymbol\ell\). Since \(f\) has superposition
order at most \(q\), it suffices to consider
\(1\leq|\boldsymbol u|\leq q\).

First suppose that \(\lambda_{\boldsymbol u}\leq L\).
For each \(j\in\boldsymbol u\), the telescoping identity
\eqref{eq:mlaqi-telescoping-multiplier} gives
\(\sum_{n=0}^{\infty}d_n(\ell_j)=1\).
The bound \eqref{eq:mlaqi-weighted-bandpass-sum}, with
\(\beta=\alpha<m\), implies absolute convergence. We may therefore
expand and rearrange the finite product of these series. By the
definition of \(a_{\boldsymbol u,L-\lambda_{\boldsymbol u}}\),
\[
1-
a_{\boldsymbol u,L-\lambda_{\boldsymbol u}}
(\boldsymbol\ell_{\boldsymbol u})
=
\sum_{\substack{
		\boldsymbol n\in\mathbb N_0^{\boldsymbol u}\\
		|\boldsymbol n|_1>L-\lambda_{\boldsymbol u}}}
\prod_{j\in\boldsymbol u}d_{n_j}(\ell_j).
\]
Every multi-index in this sum satisfies
\[
2^{\alpha|\boldsymbol n|_1}
>
2^{\alpha(L-\lambda_{\boldsymbol u})}.
\] It then follows that
\begin{align}
	\big|
	1-
	a_{\boldsymbol u,L-\lambda_{\boldsymbol u}}
	(\boldsymbol \ell_{\boldsymbol u})
	\big|
	&\leq
	2^{-\alpha(L-\lambda_{\boldsymbol u})}
	\sum_{\boldsymbol n\in\mathbb N_0^{\boldsymbol u}}
	2^{\alpha|\boldsymbol n|_1}
	\prod_{j\in\boldsymbol u}|d_{n_j}(\ell_j)|
	\notag\\
	&=
	2^{-\alpha(L-\lambda_{\boldsymbol u})}
	\prod_{j\in\boldsymbol u}
	\sum_{n=0}^{\infty}
	2^{\alpha n}|d_n(\ell_j)|
	\notag\\
	&\leq
	C^q
	2^{-\alpha(L-\lambda_{\boldsymbol u})}
	\prod_{j\in\boldsymbol u}
	\langle \ell_j\rangle^\alpha.
	\label{eq:mlaqi-retained-bias-multiplier}
\end{align}

Let $x_j:=\frac{1}{2\alpha}\log_2\gamma_j^{-1}$,
so that \(\lambda_j=\lfloor x_j\rfloor\). Then
$\gamma_{\boldsymbol u}^{1/2}=2^{-\alpha\sum_{j\in\boldsymbol u}x_j}$ and $\lambda_{\boldsymbol u}\leq\sum_{j\in\boldsymbol u}x_j$. Hence $\gamma_{\boldsymbol u}^{1/2}
2^{\alpha\lambda_{\boldsymbol u}}\leq1$.
Using this inequality and \(\alpha\leq s\) in
\eqref{eq:mlaqi-retained-bias-multiplier} yields
\begin{equation}
	\gamma_{\boldsymbol u}^{1/2}
	\big|
	1-
	a_{\boldsymbol u,L-\lambda_{\boldsymbol u}}
	(\boldsymbol \ell_{\boldsymbol u})
	\big|
	\leq
	C^q2^{-\alpha L}
	\prod_{j\in\boldsymbol u}
	\langle \ell_j\rangle^s.
	\label{eq:mlaqi-weighted-retained-bias}
\end{equation}

Next, consider the case
\(\lambda_{\boldsymbol u}>L\). Since
\(\lambda_{\boldsymbol u}\) is an integer,
\(\lambda_{\boldsymbol u}\geq L+1\). The relation
\eqref{eq:mlaqi-entry-level-comparison} gives
$$\gamma_{\boldsymbol u}^{1/2}\leq2^{-\alpha\lambda_{\boldsymbol u}}\leq2^{-\alpha(L+1)}.$$ In this case,
\(\widehat{\Phi}_{q,L}^{\mathrm{ml},(d)}(\boldsymbol \ell)=0\).
Since
\(\prod_{j\in\boldsymbol u}\langle \ell_j\rangle^s\geq1\), we get
\[
\gamma_{\boldsymbol u}^{1/2}
\big|
1-
\widehat{\Phi}_{q,L}^{\mathrm{ml},(d)}(\boldsymbol \ell)
\big|
\leq
2^{-\alpha L}
\prod_{j\in\boldsymbol u}
\langle \ell_j\rangle^s.
\]

Both cases lead to the multiplier estimate
\[
\big|
1-
\widehat{\Phi}_{q,L}^{\mathrm{ml},(d)}(\boldsymbol \ell)
\big|
\leq
C2^{-\alpha L}
\gamma_{\supp\boldsymbol \ell}^{-1/2}
\prod_{j\in\supp\boldsymbol \ell}
\langle \ell_j\rangle^s
\]
whenever
\(\boldsymbol \ell\neq\boldsymbol0\) and
\(|\supp\boldsymbol \ell|\leq q\).

 Finally, we apply Parseval's identity to obtain
\begin{align*}
	\|f-\mathcal C_{q,L}^{\mathrm{ml},(d)}f\|_2^2
	=
	\sum_{\boldsymbol \ell\in\mathbb Z^d}
	\big|
	1-
	\widehat{\Phi}_{q,L}^{\mathrm{ml},(d)}(\boldsymbol \ell)
	\big|^2
	|\widehat f(\boldsymbol \ell)|^2
	\leq
	C2^{-2\alpha L}
	\|f\|_{s,\boldsymbol\gamma}^2,
\end{align*}
which completes the proof.
\end{proof}

\subsection{Random sampling}

 Let
\(\boldsymbol X_1,\ldots,\boldsymbol X_M\) be independent random variables uniformly distributed on \(\mathbb T^d\). The direct multilevel quasi-interpolant is defined by
\begin{equation}
	\mathcal Q_{M,q,L}^{\mathrm{ml},(d)}f
	:=
	\frac{1}{M}
	\sum_{j=1}^{M}
	f(\boldsymbol X_j)
	\Phi_{q,L}^{\mathrm{ml},(d)}
	(\,\cdot-\boldsymbol X_j).
	\label{eq:mlaqi-noiseless-rule}
\end{equation}
Uniform sampling and Fubini's theorem give
$\mathbb E\mathcal Q_{M,q,L}^{\mathrm{ml},(d)}f=\mathcal C_{q,L}^{\mathrm{ml},(d)}f$. 

\begin{theorem}
\label{thm:mlaqi-risk}
Let \(d,M\in\mathbb N\) and \(L\in\mathbb N_0\), and suppose that \eqref{eq:mlaqi-alpha-assumption} holds. Then, for any
\(f\in H^s_{\mathrm{Kor},\boldsymbol\gamma,q}(\mathbb T^d)\), we have
\begin{align}
	\mathbb E
	\bigl\|
	f-\mathcal Q_{M,q,L}^{\mathrm{ml},(d)}f
	\bigr\|_2^2
	&=
	\bigl\|
	f-\mathcal C_{q,L}^{\mathrm{ml},(d)}f
	\bigr\|_2^2
	+
	\frac{1}{M}
	\left(
	\|f\|_2^2
	\|\Phi_{q,L}^{\mathrm{ml},(d)}\|_2^2
	-
	\|\mathcal C_{q,L}^{\mathrm{ml},(d)}f\|_2^2
	\right)
	\notag\\
	&\leq
	C
	\left(
	2^{-2\alpha L}
	+
	\frac{2^L(L+1)^{q-1}}{M}
	\right)
	\|f\|_{s,\boldsymbol\gamma}^2.
	\label{eq:mlaqi-exact-risk}
\end{align}
%For fixed \(q\), the constant \(C>0\) is independent of \(d\), \(M\), \(L\), and \(f\).
\end{theorem}

The level \(L_M\) below balances the two terms in the above mean-square estimate.

\begin{corollary}
\label{cor:mlaqi-rate}
Suppose that \eqref{eq:mlaqi-alpha-assumption} holds. Let
\(d\in\mathbb N\) and \(M\geq2\), and define
\begin{equation}
	L_M
	:=
	\max\left\{
	L\in\mathbb N_0:
	2^{(2\alpha+1)L}(L+1)^{q-1}
	\leq M
	\right\}.
	\label{eq:mlaqi-level-choice}
\end{equation}
Then there exists a constant \(C>0\) independent of \(d\) and \(M\) such that 
\begin{equation}
	\sup_{\substack{
		f\in
		H^s_{\mathrm{Kor},\boldsymbol\gamma,q}(\mathbb T^d)\\
		\|f\|_{s,\boldsymbol\gamma}\leq1}}
	\left(
	\mathbb E
	\bigl\|
		f-\mathcal Q_{M,q,L_M}^{\mathrm{ml},(d)}f
		\bigr\|_2^2
	\right)^{1/2} \leq C M^{-\frac{\alpha}{2\alpha+1}}
	\bigl[\log(e+M)\bigr]^{\frac{\alpha(q-1)}{2\alpha+1}}.
	\label{eq:mlaqi-rate}
\end{equation}
\end{corollary}

\section{Numerical experiments}

We present several numerical experiments to illustrate the theory developed in the preceding sections. We first test the convergence of full-grid and Monte Carlo schemes using kernels with moment conditions of different orders. We then compare the convolution errors and \(L_2\)-norms of single-level and multilevel ANOVA kernels. A dimension study contrasts weighted and unweighted multilevel schemes. Finally, we compare Monte Carlo sampling with two QMC discretizations of the convolution integral and explore their effect on the total \(L_2\)-error.

\subsection{Convergence test}
\label{sec:random-wendland-experiment}
To verify the convergence results of deterministic and stochastic quasi-interpolants established in Section~4, we consider the finitely smooth periodic function \cite{GriebelHamaekers2014}:
\begin{equation}
	f_1(\boldsymbol{x})
	=
	\prod_{j=1}^{d}
	\zeta_\mu
	\Big[
	2+
	\operatorname{sgn}\!\big(2\pi(x_j-\tfrac12)\big)
	\sin^\mu\!\big(2\pi(x_j-\tfrac12)\big)
	\Big],
	\label{eq:random-we33-target}
\end{equation}
where $\zeta_\mu$ is chosen such that \(\|f_1\|_{L_2(\mathbb T^d)}=1\). The function satisfies \(f_1\in H^{\mu+1/2-\varepsilon}(\mathbb T^d)\) for every \(\varepsilon>0\).  In the numerical experiments, we set \(\mu=8\), which provides sufficient regularity for all approximation orders considered below. 

The quasi-interpolants use a periodic kernel constructed from the compactly supported Wendland function \(\phi_{3,3}\). More precisely, we define
\begin{equation}
	\phi_{3,3}(r)
	=
	(1-r)_+^8(1+8r+25r^2+32r^3),
	\quad
	\psi_\rho(t)
	=
	\frac{\pi}{\rho A_\rho}
	\phi_{3,3}\!\left(
	\frac{2|\sin(\pi t)|}{\rho}
	\right),
	\label{eq:we33-periodic-kernel}
\end{equation}
where \((1-r)_+=\max\{1-r,0\}\). For \(m=2p\), we construct a higher-order kernel satisfying the moment conditions of order \(m\) by
\begin{equation}
	\psi_{\rho,m}(t)
	=
	\sum_{\ell=1}^{p}
	(-1)^{\ell+1}
	\binom{p}{\ell}
	\psi_{\rho\sqrt{\ell/p}}(t).
	\label{eq:we33-moment-kernel}
\end{equation}
The corresponding \(d\)-variate kernel is \(\Psi_{\rho,m}^{(d)}(\boldsymbol{x})
=\prod_{j=1}^{d}\psi_{\rho,m}(x_j)\), as defined in Section~4, and is used in both quasi-interpolation operators.

We first consider the full-grid periodic quasi-interpolation. The scheme uses the full tensor grid
\(\mathcal X_n=\{\boldsymbol{k}/n:
\boldsymbol{k}\in\{0,\ldots,n-1\}^d\}\), with \(h=1/n\) and \(n^d\) sampling points. Let \(\mathcal G\subset\mathbb T^d\) be a fixed evaluation grid. The discrete \(L_2\)-error of the deterministic quasi-interpolant is defined by
\begin{equation}
	\mathcal{E}_n
	=
	\left[
	\frac{1}{|\mathcal G|}
	\sum_{\boldsymbol{y}\in\mathcal G}
	\left|
	Q_{n,\rho}f_1(\boldsymbol{y})
	-f_1(\boldsymbol{y})
	\right|^2
	\right]^{1/2}.
	\label{eq:deterministic-L2-error}
\end{equation}
 The convergence rates are computed from discrete errors according to \(\operatorname{rate}_n=\log_2(\mathcal E_{n/2}/\mathcal E_n)\).  For targets with smoothness at least \(m\), the estimate in Section~4 gives the optimal scaled parameter \(\rho\asymp n^{-2\kappa/(2\kappa+m)}\) and the theoretical error rate \(O(n^{-r})\) with \(r=2\kappa m/(2\kappa+m)\). As \(\kappa\) increases, the convergence order \(r\) approaches \(m\). In the experiments, we set \(\rho=c_mn^{-1}\), with \(c_m=30,60,120\) for \(m=2,4,6\), respectively. These coefficients are not claimed to be optimal. Table~\ref{tab:wendland-fixed-kernel} shows the discrete errors and convergence rates in dimension \(d=2\), from which we observe that the measured orders are higher than the theoretical estimates and approach the moment condition order \(m\).

 We next study stochastic quasi-interpolation based on \(M\)
 independent, uniformly distributed samples on \(\mathbb T^d\).
 For each parameter set, we perform \(R=20\) independent trials and
define the root-mean-square error (RMSE) by
 \begin{equation}
 	\operatorname{RMSE}_M
 	=
 	\left[
 	\frac{1}{R}
 	\sum_{r=1}^{R}
 	\frac{1}{|\mathcal G|}
 	\sum_{\boldsymbol{y}\in\mathcal G}
 	\left|
 	Q^{(r)}_{M,\Psi_{\rho,m}^{(d)}}f_1(\boldsymbol{y})
 	-f_1(\boldsymbol{y})
 	\right|^2
 	\right]^{1/2},
 	\label{eq:stochastic-RMSE}
 \end{equation}
 where \(Q^{(r)}_{M,\Psi_{\rho,m}^{(d)}}f_1\) is the quasi-interpolant obtained in trial \(r\). We choose the shape parameter as $\rho=c_m M^{-1/(2m+d)}$, for which the estimate in Section~4 predicts a convergence rate of \(\mathcal O(M^{-m/(2m+d)})\).  Figure~\ref{fig:random-wendland-corrected-order-dimension}
 illustrates the dependence of the error on the moment condition order \(m\) and the dimension \(d\). The left panel fixes \(d=1\) and uses \(c_m=1,1.5,3\) for \(m=2,4,6\), respectively. The fitted convergence orders are \(0.40\), \(0.44\), and \(0.46\), in close agreement with the corresponding theoretical orders. It demonstrates that increasing \(m\) can improves the convergence rate and bring it closer to the optimal rate $1/2$. In addition, the right panel fixes \(m=6\) and \(c_6=3\) while varying \(d=1,2,3\). The corresponding fitted orders are \(0.459\), \(0.430\), and \(0.401\). For this tensor-product test problem, the RMSE increases with dimension, whereas the corresponding convergence rates decrease. Overall, these results support the predicted dependence of the convergence rate on both \(m\) and \(d\).

% BEGIN original tables/wendland_comparison_table
% Requires \usepackage{booktabs}. EOC uses unrounded source errors.
\begin{table*}[t]
\centering
\caption{Discrete errors and convergence rates of the full-grid periodic quasi-interpolation based on $\mathrm{WE}_{3,3}$ for approximating $f_1$ in dimension $d=2$, with moment condition orders $m=2,4,6$.}
\label{tab:wendland-fixed-kernel}
\small
\setlength{\tabcolsep}{4pt}
\renewcommand{\arraystretch}{1.12}
\begin{tabular}{@{}lcrcrcr@{}}
\toprule
& \multicolumn{2}{c}{$m=2$}
& \multicolumn{2}{c}{$m=4$}
& \multicolumn{2}{c}{$m=6$}\\
\cmidrule(lr){2-3}\cmidrule(lr){4-5}\cmidrule(l){6-7}
$n$ & error & rate & error & rate & error & rate\\
\midrule
128
& $1.0231\times10^{-2}$ & -
& $3.2510\times10^{-3}$ & -
& $7.5736\times10^{-3}$ & - \\
256
& $2.6318\times10^{-3}$ & 1.96
& $2.4055\times10^{-4}$ & 3.78
& $2.4447\times10^{-4}$ & 4.95 \\
512
& $6.6276\times10^{-4}$ & 1.99
& $1.5709\times10^{-5}$ & 3.94
& $4.6648\times10^{-6}$ & 5.71 \\
1024
& $1.6599\times10^{-4}$ & 2.00
& $9.9320\times10^{-7}$ & 3.98
& $7.6698\times10^{-8}$ & 5.93 \\
2048
& $4.1518\times10^{-5}$ & 2.00
& $6.8452\times10^{-8}$ & 3.86
& $1.4699\times10^{-9}$ & 5.71 \\
\bottomrule
\end{tabular}
\end{table*}
% END original tables/wendland_comparison_table

\subsection{Weight-adapted ANOVA quasi-interpolation}
We first compare the continuous convolution errors and kernel $L_2$-norms of the single-level and multilevel constructions of weight-adapted ANOVA kernels. To this end, we consider the weighted product function
\begin{equation}
	f_{2}(\boldsymbol x)
	=
	\zeta_0\prod_{j=1}^{d}
	\left[
	1+\gamma_j\left(x_j-\tfrac12\right)^2\sin(2\pi x_j-\pi)
	\right],
	\quad
	\boldsymbol x\in\mathbb T^{d},
	\label{eq:weighted-periodic-sine-target}
\end{equation}
where \(\gamma_j\) is the \(j\)-th coordinate weight, and the constant \(\zeta_0\) is chosen so that
\(\|f_{2}\|_{L_2(\mathbb T^{d})}=1\). 
The periodic extension of \(f_2\) belongs to \(H^s(\mathbb T^d)\) for every
\(s<5/2\).

We take \(d=q=10\), set \(\gamma_j=j^{-2}\), and use the periodic \(\mathrm{WE}_{3,3}\) kernel of order \(m=4\), with \(\alpha=1\). At each level \(L=0,\ldots,10\), both constructions use the scale $\rho_L=2^{-L}$. For the single-level kernel, the coordinatewise scales are \(\rho_{L,j}=\rho_L\gamma_j^{-1/2}\), and an interaction \(\boldsymbol u\) is retained if
\(\gamma_{\boldsymbol u}\geq\rho_L^2\). For the multilevel kernel, the entry level of coordinate \(j\) is $\lambda_j=\lfloor\log_2j\rfloor$. A tensor-product difference indexed by
\((\boldsymbol{u},\boldsymbol{n})\) is retained whenever \(\lambda_{\boldsymbol u}+|\boldsymbol n|_1\leq L\), where \(\lambda_{\boldsymbol{u}}\) denotes the corresponding interaction entry level. 

Figure~\ref{fig:Conv_err_comparison} presents the convolution errors and \(L_2\)-norms of the ANOVA kernels, defined by
\begin{equation}
	\begin{aligned}
		\mathcal{E}_L^{\mathrm{sl}}
		&=
		\left\|
		f_2-\mathcal{C}_{q,\rho_L}^{(d)}f_2
		\right\|_{L_2(\mathbb{T}^{d})},
		&
		\mathcal{E}_L^{\mathrm{ml}}
		&=
		\left\|
		f_2-\mathcal{C}_{q,L}^{\mathrm{ml},(d)}f_2
		\right\|_{L_2(\mathbb{T}^{d})},
		\\
		K_L^{\mathrm{sl}}
		&=
		\left\|
		\Phi_{q,\rho_L}^{(d)}
		\right\|_{L_2(\mathbb{T}^{d})},
		&
		K_L^{\mathrm{ml}}
		&=
		\left\|
		\Phi_{q,L}^{\mathrm{ml},(d)}
		\right\|_{L_2(\mathbb{T}^{d})}.
	\end{aligned}
	\label{eq:weighted-periodic-convolution-quantities}
\end{equation}
As shown in the figure, for \(L\geq1\), the multilevel method yields a substantially smaller convolution error than the single-level method. At the final level \(L=10\), the convolution errors are
$\mathcal{E}_{10}^{\mathrm{sl}}= 5.587\times10^{-10}$ and $\mathcal{E}_{10}^{\mathrm{ml}}=7.880\times10^{-11}$, respectively. The reduction in kernel norm is more pronounced. At
\(L=10\), the corresponding kernel norms are $K_{10}^{\mathrm{sl}} = 3.650\times10^{11}$,
$K_{10}^{\mathrm{ml}}=1.358\times10^{5}$, respectively. The hyperbolic level constraint avoids simultaneous finest-scale refinement of all coordinates in an active interaction and thereby limits the growth of the kernel norm. In this example, the multilevel construction gives a smaller convolution error together with a much smaller kernel norm.

The second experiment explores how coordinate weights affect the dependence of the multilevel quasi-interpolation error on the ambient dimension. We use weighted versions of the pyramid function in \cite[Eq.~(5.4)]{LippertPottsUllrich2023}:
\begin{equation}
	f_{3}(\bx)=\sqrt{72}\sum_{i=1}^{d/2}\alpha_{i,d}
	\left(\frac13-\max\{|x_{2i-1}|,|x_{2i}|\}\right),
	\quad 
	\alpha_{i,d}
	=
	\frac{i^{-2}}
	{\bigl(\sum_{j=1}^{d/2}j^{-4}\bigr)^{1/2}},
	\label{eq:dimexp-concise-target}
\end{equation}
where \(\bx\in[-1/2,1/2)^d\) and \(d\in\{2,4,8,16,32\}\). Each target function has zero mean and unit \(L_2(\mathbb T^d)\) norm. Its nonzero ANOVA components are supported on the \(d\) singletons and the pairs \(\{2i-1,2i\}\), \(i=1,\ldots,d/2\). Both approximants use exactly these supports.

The weighted and unweighted schemes use the periodic kernel \eqref{eq:we33-moment-kernel} with \(m=4\), \(q=2\), and \(\rho_\ell=6\cdot2^{-\ell}\). The weight-adapted scheme sets
$\gamma_{2i-1} = \gamma_{2i} = i^{-1.4}$.
The unweighted control sets all coordinate entry levels to zero. Figure~\ref{fig:anova-we33-dimension-study} plots the RMSE against the sample size \(N\). We see that the errors of the weight-adapted scheme vary less with $d$. At \(N=2^{20}\), both methods have RMSE \(2.605\times10^{-2}\) for \(d=2\). For \(d=32\), the weighted and unweighted errors are \(3.983\times10^{-2}\) and \(8.082\times10^{-2}\), respectively. These results indicate that the weight-adapted scheme is less sensitive to increasing dimension.

% BEGIN original figures/random_points_comparison.tex
% Self-contained figure. Requires pgfplots and its groupplots library.
\begin{figure*}[t]
\centering
\begingroup
\definecolor{comparisonLow}{RGB}{0,114,178}
\definecolor{comparisonMid}{RGB}{230,159,0}
\definecolor{comparisonHigh}{RGB}{0,158,115}
\begin{tikzpicture}
\begin{groupplot}[
 group style={group size=2 by 1,horizontal sep=1.2cm},
 width=0.42\textwidth,height=0.48\textwidth,
 xmode=log,ymode=log,log basis x=10,log basis y=10,
 grid=both,
 major grid style={line width=0.6pt, draw=gray!45, dashed},
 minor grid style={line width=0.4pt, draw=gray!25, dotted},
 axis line style={line width=0.55pt},tick style={line width=0.45pt},
 tick align=inside,xlabel={$M$},
 xlabel style={at={(axis description cs:0.5,-0.12)},anchor=north},
 title style={yshift=-1mm},label style={font=\small},
 tick label style={font=\footnotesize},
 every axis plot/.append style={line width=0.95pt},
 legend style={at={(0.02,0.02)},anchor=south west,font=\scriptsize,
  draw=none,fill=white,fill opacity=0.86,text opacity=1,
  inner sep=2pt,outer sep=1pt,legend columns=1,cells={anchor=west}}
]
% LEFT AXIS: edit these limits, tick positions and labels independently.
\nextgroupplot[
 title={$d=1$},
 ylabel={RMSE},
 xmin=300,xmax=4000000,ymin=0.003,ymax=2,
 ,legend style={at={(0.98,0.98)},anchor=north east}
]
\addplot[color=comparisonLow,only marks,mark=*,mark size=1.45pt] coordinates {
 (1024,1.7962e-1)
 (2048,1.3289e-1)
 (4096,1.0013e-1)
 (8192,8.0072e-2)
 (16384,6.1939e-2)
 (32768,4.5573e-2)
 (65536,3.5036e-2)
 (131072,2.6073e-2)
 (262144,1.9469e-2)
 (524288,1.4789e-2)
 (1048576,1.1170e-2)
};
\addlegendentry{$m=2$, rate=0.40}
\addplot[black,densely dashed,mark=none,line width=0.65pt,forget plot] coordinates {
 (1024,1.7976e-1)
 (1048576,1.1338e-2)
};
\addplot[color=comparisonMid,only marks,mark=square*,mark size=1.35pt] coordinates {
 (1024,1.5473e-1)
 (2048,1.0810e-1)
 (4096,7.9921e-2)
 (8192,6.2617e-2)
 (16384,4.6682e-2)
 (32768,3.3577e-2)
 (65536,2.5128e-2)
 (131072,1.7884e-2)
 (262144,1.2826e-2)
 (524288,9.5430e-3)
 (1048576,7.0458e-3)
};
\addlegendentry{$m=4$, rate=0.44}
\addplot[black,densely dashed,mark=none,line width=0.65pt,forget plot] coordinates {
 (1024,1.5324e-1)
 (1048576,7.1142e-3)
};
\addplot[color=comparisonHigh,only marks,mark=triangle*,mark size=1.7pt] coordinates {
 (1024,1.1535e-1)
 (2048,7.7667e-2)
 (4096,5.8102e-2)
 (8192,4.5128e-2)
 (16384,3.3023e-2)
 (32768,2.4073e-2)
 (65536,1.7442e-2)
 (131072,1.2294e-2)
 (262144,8.8085e-3)
 (524288,6.3696e-3)
 (1048576,4.5968e-3)
};
\addlegendentry{$m=6$, rate=0.46}
\addplot[black,densely dashed,mark=none,line width=0.65pt,forget plot] coordinates {
 (1024,1.13799185103e-1)
 (1048576,4.71720038292e-3)
};
% RIGHT AXIS: edit these limits, tick positions and labels independently.
\nextgroupplot[
 title={$m=6$},
 xmin=300,xmax=4000000,ymin=0.003,ymax=2,
]
\addplot[color=comparisonLow,only marks,mark=*,mark size=1.45pt] coordinates {
 (1024,1.1535e-1)
 (2048,7.7667e-2)
 (4096,5.8102e-2)
 (8192,4.5128e-2)
 (16384,3.3023e-2)
 (32768,2.4073e-2)
 (65536,1.7443e-2)
 (131072,1.2294e-2)
 (262144,8.8085e-3)
 (524288,6.3696e-3)
 (1048576,4.5968e-3)
};
\addlegendentry{$d=1$, rate=0.46}
\addplot[black,densely dashed,mark=none,line width=0.65pt,forget plot] coordinates {
 (1024,1.13799185103e-1)
 (1048576,4.71720038292e-3)
};
\addplot[color=comparisonMid,only marks,mark=square*,mark size=1.35pt] coordinates {
 (1024,4.1769e-1)
 (2048,3.1465e-1)
 (4096,2.3431e-1)
 (8192,1.7443e-1)
 (16384,1.2919e-1)
 (32768,9.5365e-2)
 (65536,6.9925e-2)
 (131072,5.2148e-2)
 (262144,3.8985e-2)
 (524288,2.8829e-2)
 (1048576,2.1676e-2)
};
\addlegendentry{$d=2$, rate=0.43}
\addplot[black,densely dashed,mark=none,line width=0.65pt,forget plot] coordinates {
 (1024,4.22449868062e-1)
 (1048576,2.14922223922e-2)
};
\addplot[color=comparisonHigh,only marks,mark=triangle*,mark size=1.7pt] coordinates {
 (1024,1.4752e+0)
 (2048,1.1185e+0)
 (4096,8.4850e-1)
 (8192,6.4287e-1)
 (16384,4.8436e-1)
 (32768,3.6894e-1)
 (65536,2.7982e-1)
 (131072,2.1108e-1)
 (262144,1.5983e-1)
 (524288,1.2121e-1)
 (1048576,9.2050e-2)
};
\addlegendentry{$d=3$, rate=0.40}
\addplot[black,densely dashed,mark=none,line width=0.65pt,forget plot] coordinates {
 (1024,1.47651435038e+0)
 (1048576,9.18992859544e-2)
};
\end{groupplot}
\end{tikzpicture}
\endgroup
\caption{RMSE of stochastic quasi-interpolation for $f_1$ with $\mu=8$. Each point is computed over $R=20$ independent trials at $M=2^k$, $k=10,\ldots,20$, using $\rho=c_mM^{-1/(2m+d)}$. The left panel compares $(m,c_m)=(2,1),(4,1.5),(6,3)$ for $d=1$; the right compares $d=1,2,3$ for $m=6$ and $c_6=3$.}
\label{fig:random-wendland-corrected-order-dimension}
\end{figure*}
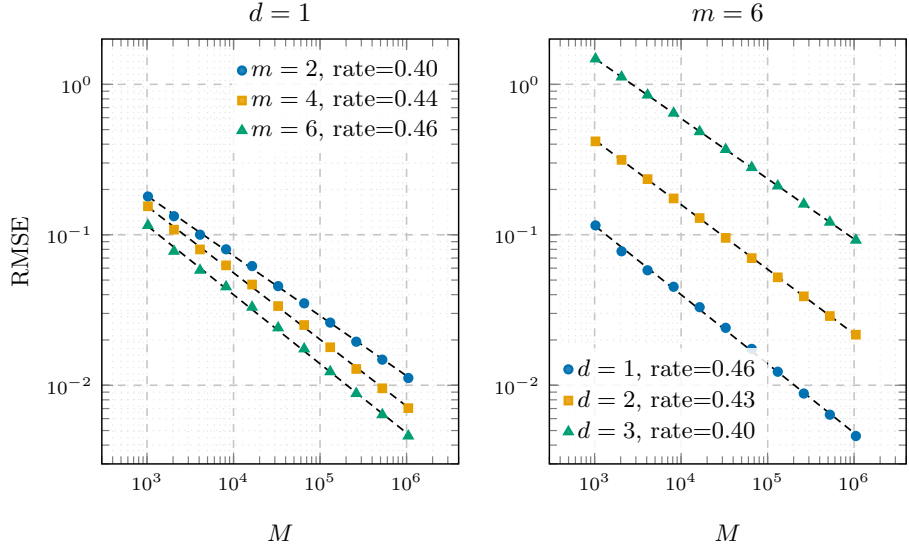
% END original figures/random_points_comparison.tex

\subsection{QMC discretization}
A key advantage of the quasi-interpolation framework is its flexibility in discretizing the convolution integral. Different quadrature rules give rise to different quasi-interpolation schemes and may lead to substantially different approximation properties. In this subsection, we compare standard Monte Carlo
(MC) sampling with two deterministic QMC rules: an unscrambled Sobol' sequence and a rank-1 lattice generated by a component-by-component (CBC) construction. 

 Let \(B_3\) denote the centered cubic cardinal B-spline. We define $$v(t)
=\frac{1}{\sqrt{3/40}}\Big[B_3(4t-\pi^{-1})-\frac{1}{4}\Big],\quad t\in[-1/2,1/2)$$ and extend it periodically to \(\mathbb T\). The test function is
\begin{equation}
	f_4(\bx)
	=
	\sum_{i=1}^{d/2}
	\alpha_{i,d} v(x_{2i-1})v(x_{2i}).
	\label{eq:b3-pair-target}
\end{equation}

We use the weight-adapted ANOVA kernel with maximal interaction order \(q=2\), univariate kernel order \(m=4\), and adaptation parameter \(a=0.6\). The coordinate weights are \(\gamma_{2i-1}=\gamma_{2i}=i^{-1.4}\). The deterministic QMC discretizations use the same seven cardinalities
\[
N\in
\{10141,32077,99793,316801,995329,3161089,9966529\}.
\]
These values are prime, as required by the CBC construction described below. The Monte Carlo method uses the RMS of the total \(L_2\)-errors over $R$ independent samples. The Sobol rule uses prefixes of the unscrambled Sobol sequence \cite{JoeKuo2003,Sobol1967}, while the CBC discretization is based on the rank-1 lattice
\[
\mathcal X_N(\boldsymbol z)
=
\left\{
\left\{\frac{n\boldsymbol z}{N}\right\}
:
n=0,\ldots,N-1
\right\},
\quad
\boldsymbol z\in\{1,\ldots,N-1\}^{d},
\]
where the inner braces denote the componentwise fractional part. For each \(N\), the generating vector is constructed component-by-component \cite{Kuo2003} and implemented by a fast algorithm \cite{NuyensCools2006}.

Figure~\ref{fig:b3-anova-three-rule-d4-d6} compares the $L_2$-errors of three discretization methods. For \(d=4\), MC and Sobol use the shape parameter \(\rho=4^{-7/6}\), while CBC uses \(\rho=6^{-7/6}\). For \(d=6\), all three methods employ the same shape parameter \(\rho=9^{-7/6}\). Least-squares fits over all sample sizes give decay rates \(0.50\), \(0.89\), and \(1.59\) for MC, Sobol, and CBC in dimension \(d=4\). The corresponding values for
\(d=6\) are \(0.50\), \(0.77\), and \(0.97\). The MC results agree with the expected \(N^{-1/2}\) decay. Both QMC rules yield faster error decay and more accurate approximations than MC sampling.

% BEGIN original figures/multiscale_experiment.tex

% END original figures/multiscale_experiment.tex

% BEGIN original figures/multilevel_dimension_convergence.tex
% Insertable PGFPlots snippet generated from level_results_4sig.csv.
% Required in the manuscript preamble:
%   \usepackage{pgfplots}
%   \usepgfplotslibrary{groupplots}
%   \pgfplotsset{compat=1.18}
%
% All ordinate data use four significant digits. Distinct colors and line
% styles identify the two methods; no fitted or theoretical-order line is shown.
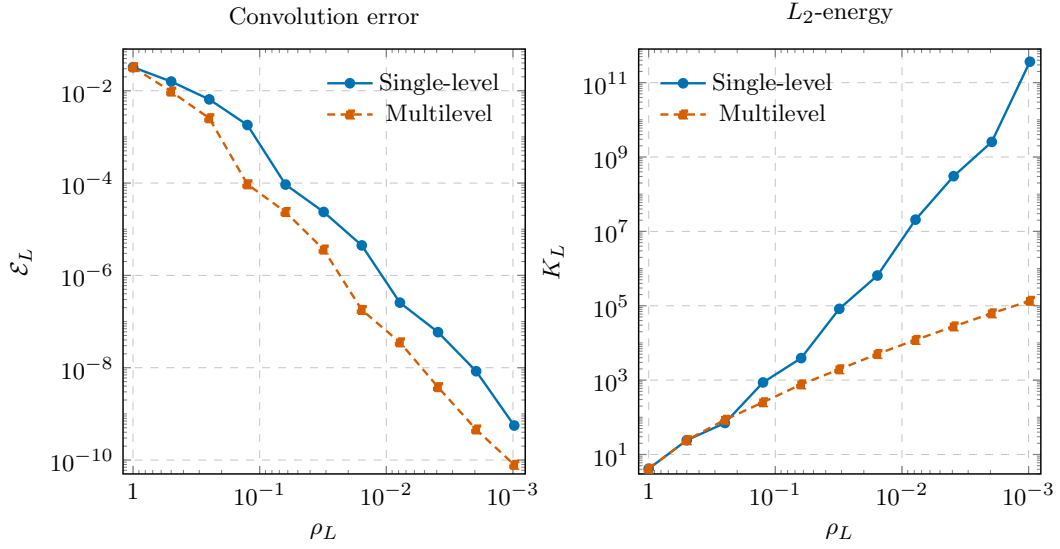
\begin{figure*}[t]
\centering
\begingroup
\definecolor{dimBlue}{RGB}{0,114,178}
\definecolor{vermillion}{RGB}{213,94,0}
\begin{tikzpicture}
	\begin{groupplot}[
		group style={group size=2 by 1,horizontal sep=1.5cm},
		width=0.46\textwidth,
		height=0.48\textwidth,
		xmode=log,
		log basis x=10,
		x dir=reverse,
		xmin=8e-4,
		xmax=1.2,
		xtick={1,0.1,0.01,0.001},
		xticklabels={$1$,$10^{-1}$,$10^{-2}$,$10^{-3}$},
		minor xtick={0.002,0.003,0.004,0.005,0.006,0.007,0.008,0.009,
			0.02,0.03,0.04,0.05,0.06,0.07,0.08,0.09,
			0.2,0.3,0.4,0.5,0.6,0.7,0.8,0.9},
		ymode=log,
		log basis y=10,
		grid=major,
		grid style={gray!25,line width=0.25pt},
		major grid style={gray!45,solid,line width=0.25pt,dashed},
		axis line style={line width=0.55pt},
		tick style={line width=0.45pt},
		minor tick style={line width=0.30pt},
		xtick pos=both,
		ytick pos=both,
		tick align=inside,
		major tick length=3pt,
		minor tick length=1.5pt,
		scaled ticks=false,
		minor x tick num=0,
		minor y tick num=0,
		label style={font=\small},
		tick label style={font=\scriptsize},
		title style={font=\small},
		every axis plot/.append style={solid,line width=0.9pt},
		legend style={draw=none,fill=none,font=\scriptsize}
		]
		
		\nextgroupplot[
		title={Convolution error},
		xlabel={$\rho_L$},
		ylabel={$\mathcal{E}_L$},
		ymin=5e-11,
		ymax=8e-2,
		ytick={1e-10,1e-8,1e-6,1e-4,1e-2},
		yticklabels={$10^{-10}$,$10^{-8}$,$10^{-6}$,$10^{-4}$,$10^{-2}$},
		minor ytick={6e-11,7e-11,8e-11,9e-11,2e-10,3e-10,4e-10,5e-10,6e-10,7e-10,8e-10,9e-10,1e-9,2e-9,3e-9,4e-9,5e-9,6e-9,7e-9,8e-9,9e-9,2e-8,3e-8,4e-8,5e-8,6e-8,7e-8,8e-8,9e-8,1e-7,2e-7,3e-7,4e-7,5e-7,6e-7,7e-7,8e-7,9e-7,2e-6,3e-6,4e-6,5e-6,6e-6,7e-6,8e-6,9e-6,1e-5,2e-5,3e-5,4e-5,5e-5,6e-5,7e-5,8e-5,9e-5,2e-4,3e-4,4e-4,5e-4,6e-4,7e-4,8e-4,9e-4,1e-3,2e-3,3e-3,4e-3,5e-3,6e-3,7e-3,8e-3,9e-3,2e-2,3e-2,4e-2,5e-2,6e-2,7e-2,8e-2},
		legend pos=north east
		]
		\addplot[dimBlue,solid,mark=*,mark size=1.55pt] coordinates {
			(1,0.03241)
			(0.5,0.01582)
			(0.25,0.0065)
			(0.125,0.001811)
			(0.0625,9.279e-05)
			(0.03125,2.369e-05)
			(0.015625,4.462e-06)
			(0.0078125,2.569e-07)
			(0.00390625,5.891e-08)
			(0.001953125,8.397e-09)
			(0.0009765625,5.587e-10)
		};
		\addlegendentry{Single-level}
		\addplot[vermillion,densely dashed,mark=square*,mark size=1.45pt] coordinates {
			(1,0.03241)
			(0.5,0.009605)
			(0.25,0.002539)
			(0.125,9.408e-05)
			(0.0625,2.363e-05)
			(0.03125,3.627e-06)
			(0.015625,1.788e-07)
			(0.0078125,3.542e-08)
			(0.00390625,3.797e-09)
			(0.001953125,4.597e-10)
			(0.0009765625,7.88e-11)
		};
		\addlegendentry{Multilevel}
		
		\nextgroupplot[
		title={$L_2$-energy},
		xlabel={$\rho_L$},
		ylabel={$K_L$},
		ymin=3,
		ymax=8e11,
		ytick={1e1,1e3,1e5,1e7,1e9,1e11},
		yticklabels={$10^1$,$10^3$,$10^5$,$10^7$,$10^9$,$10^{11}$},
		minor ytick={3e0,4e0,5e0,6e0,7e0,8e0,9e0,2e1,3e1,4e1,5e1,6e1,7e1,8e1,9e1,1e2,2e2,3e2,4e2,5e2,6e2,7e2,8e2,9e2,2e3,3e3,4e3,5e3,6e3,7e3,8e3,9e3,1e4,2e4,3e4,4e4,5e4,6e4,7e4,8e4,9e4,2e5,3e5,4e5,5e5,6e5,7e5,8e5,9e5,1e6,2e6,3e6,4e6,5e6,6e6,7e6,8e6,9e6,2e7,3e7,4e7,5e7,6e7,7e7,8e7,9e7,1e8,2e8,3e8,4e8,5e8,6e8,7e8,8e8,9e8,2e9,3e9,4e9,5e9,6e9,7e9,8e9,9e9,1e10,2e10,3e10,4e10,5e10,6e10,7e10,8e10,9e10,2e11,3e11,4e11,5e11,6e11,7e11,8e11},
		legend pos=north west
		]
		\addplot[dimBlue,solid,mark=*,mark size=1.55pt] coordinates {
			(1,4.138)
			(0.5,24.2)
			(0.25,70.26)
			(0.125,865.5)
			(0.0625,3911)
			(0.03125,82570)
			(0.015625,649000)
			(0.0078125,20690000)
			(0.00390625,304900000)
			(0.001953125,2550000000)
			(0.0009765625,365000000000)
		};
		\addlegendentry{Single-level}
		\addplot[vermillion,densely dashed,mark=square*,mark size=1.45pt] coordinates {
			(1,4.138)
			(0.5,24.21)
			(0.25,85.82)
			(0.125,253.9)
			(0.0625,767.6)
			(0.03125,1955)
			(0.015625,5031)
			(0.0078125,12050)
			(0.00390625,27820)
			(0.001953125,62560)
			(0.0009765625,135800)
		};
		\addlegendentry{Multilevel}
		
	\end{groupplot}
\end{tikzpicture}
\endgroup
\caption{Convolution errors (left) and corresponding kernel $L_2$-norms (right) for single-level and multilevel ANOVA approximation of $f_2(\boldsymbol{x})$. }
	\label{fig:Conv_err_comparison}
\end{figure*}

\begin{figure*}[t]
	\centering
	\begingroup
	\definecolor{dimBlue}{RGB}{0,114,178}
	\definecolor{dimOrange}{RGB}{230,159,0}
	\definecolor{dimGreen}{RGB}{0,158,115}
	\definecolor{dimPurple}{RGB}{204,121,167}
	\definecolor{dimRed}{RGB}{213,94,0}
	\begin{tikzpicture}
		\begin{groupplot}[
			group style={group size=2 by 1,horizontal sep=1.15cm},
			width=0.46\textwidth,height=0.48\textwidth,
			xmode=log,ymode=log,log basis x=10,log basis y=10,
			xlabel={$N$},
			grid=both,
			major grid style={line width=0.6pt, draw=gray!45, dashed},
			minor grid style={line width=0.4pt, draw=gray!25, dotted},
			axis line style={line width=0.55pt},tick style={line width=0.45pt},tick align=inside,
			label style={font=\small},tick label style={font=\scriptsize},title style={font=\small},
			every axis plot/.append style={line width=0.9pt},
			legend style={draw=none,fill=white,fill opacity=0.94,text opacity=1,font=\scriptsize,
				cells={anchor=west},inner sep=2pt,at={(0.02,0.02)},anchor=south west}
			]
			% LEFT axis: independent limits and tick labels.
			\nextgroupplot[title={Unweighted multilevel},ylabel={RMSE},
			xmin=500,xmax=2000000,ymin=0.008,ymax=1.2,
			]
			\addplot[dimBlue,solid,mark=*,mark size=1.7pt] coordinates {
				(1024,1.860e-1)
				(2048,1.549e-1)
				(4096,1.313e-1)
				(8192,1.042e-1)
				(16384,8.334e-2)
				(32768,7.055e-2)
				(65536,5.953e-2)
				(131072,4.684e-2)
				(262144,3.943e-2)
				(524288,3.260e-2)
				(1048576,2.605e-2)
			};
			\addlegendentry{$d=2$}
			\addplot[dimOrange,solid,mark=square*,mark size=1.7pt] coordinates {
				(1024,2.378e-1)
				(2048,1.867e-1)
				(4096,1.519e-1)
				(8192,1.382e-1)
				(16384,1.059e-1)
				(32768,8.352e-2)
				(65536,7.850e-2)
				(131072,5.906e-2)
				(262144,4.683e-2)
				(524288,4.255e-2)
				(1048576,3.254e-2)
			};
			\addlegendentry{$d=4$}
			\addplot[dimGreen,solid,mark=triangle*,mark size=1.7pt] coordinates {
				(1024,3.187e-1)
				(2048,2.409e-1)
				(4096,1.845e-1)
				(8192,1.861e-1)
				(16384,1.379e-1)
				(32768,1.052e-1)
				(65536,1.066e-1)
				(131072,7.820e-2)
				(262144,5.938e-2)
				(524288,5.808e-2)
				(1048576,4.280e-2)
			};
			\addlegendentry{$d=8$}
			\addplot[dimPurple,solid,mark=diamond*,mark size=1.7pt] coordinates {
				(1024,4.387e-1)
				(2048,3.203e-1)
				(4096,2.394e-1)
				(8192,2.578e-1)
				(16384,1.860e-1)
				(32768,1.384e-1)
				(65536,1.470e-1)
				(131072,1.062e-1)
				(262144,7.808e-2)
				(524288,8.067e-2)
				(1048576,5.835e-2)
			};
			\addlegendentry{$d=16$}
			\addplot[dimRed,solid,mark=pentagon*,mark size=1.7pt] coordinates {
				(1024,6.126e-1)
				(2048,4.398e-1)
				(4096,3.206e-1)
				(8192,3.614e-1)
				(16384,2.577e-1)
				(32768,1.877e-1)
				(65536,2.052e-1)
				(131072,1.467e-1)
				(262144,1.061e-1)
				(524288,1.129e-1)
				(1048576,8.082e-2)
			};
			\addlegendentry{$d=32$}
			% RIGHT axis: independent limits and tick labels.
			\nextgroupplot[title={Weight-adapted multilevel},
			xmin=500,xmax=2000000,ymin=0.008,ymax=1.2,
			]
			\addplot[dimBlue,solid,mark=*,mark size=1.7pt] coordinates {
				(1024,1.860e-1)
				(2048,1.549e-1)
				(4096,1.313e-1)
				(8192,1.042e-1)
				(16384,8.334e-2)
				(32768,7.055e-2)
				(65536,5.953e-2)
				(131072,4.684e-2)
				(262144,3.943e-2)
				(524288,3.260e-2)
				(1048576,2.605e-2)
			};
			\addlegendentry{$d=2$}
			\addplot[dimOrange,solid,mark=square*,mark size=1.7pt] coordinates {
				(1024,2.264e-1)
				(2048,1.988e-1)
				(4096,1.791e-1)
				(8192,1.178e-1)
				(16384,9.629e-2)
				(32768,8.380e-2)
				(65536,6.697e-2)
				(131072,5.364e-2)
				(262144,4.566e-2)
				(524288,3.652e-2)
				(1048576,2.947e-2)
			};
			\addlegendentry{$d=4$}
			\addplot[dimGreen,solid,mark=triangle*,mark size=1.7pt] coordinates {
				(1024,2.468e-1)
				(2048,2.191e-1)
				(4096,1.979e-1)
				(8192,1.299e-1)
				(16384,1.069e-1)
				(32768,9.309e-2)
				(65536,7.741e-2)
				(131072,6.348e-2)
				(262144,5.538e-2)
				(524288,4.074e-2)
				(1048576,3.299e-2)
			};
			\addlegendentry{$d=8$}
			\addplot[dimPurple,solid,mark=diamond*,mark size=1.7pt] coordinates {
				(1024,2.520e-1)
				(2048,2.249e-1)
				(4096,2.045e-1)
				(8192,1.377e-1)
				(16384,1.153e-1)
				(32768,1.021e-1)
				(65536,8.299e-2)
				(131072,6.940e-2)
				(262144,6.144e-2)
				(524288,4.549e-2)
				(1048576,3.807e-2)
			};
			\addlegendentry{$d=16$}
			\addplot[dimRed,solid,mark=pentagon*,mark size=1.7pt] coordinates {
				(1024,2.527e-1)
				(2048,2.258e-1)
				(4096,2.054e-1)
				(8192,1.391e-1)
				(16384,1.171e-1)
				(32768,1.041e-1)
				(65536,8.521e-2)
				(131072,7.156e-2)
				(262144,6.368e-2)
				(524288,4.733e-2)
				(1048576,3.983e-2)
			};
			\addlegendentry{$d=32$}
		\end{groupplot}
	\end{tikzpicture}
	\endgroup
	\caption{RMSE of multilevel ANOVA quasi-interpolation for $f_3(\boldsymbol{x})$ as the ambient dimension increases. The unweighted (left) and weight-adapted (right) schemes use centered $\mathrm{WE}_{3,3}$ kernels of order $m=4$ with maximal interaction order $q=2$.}
	\label{fig:anova-we33-dimension-study}
\end{figure*}
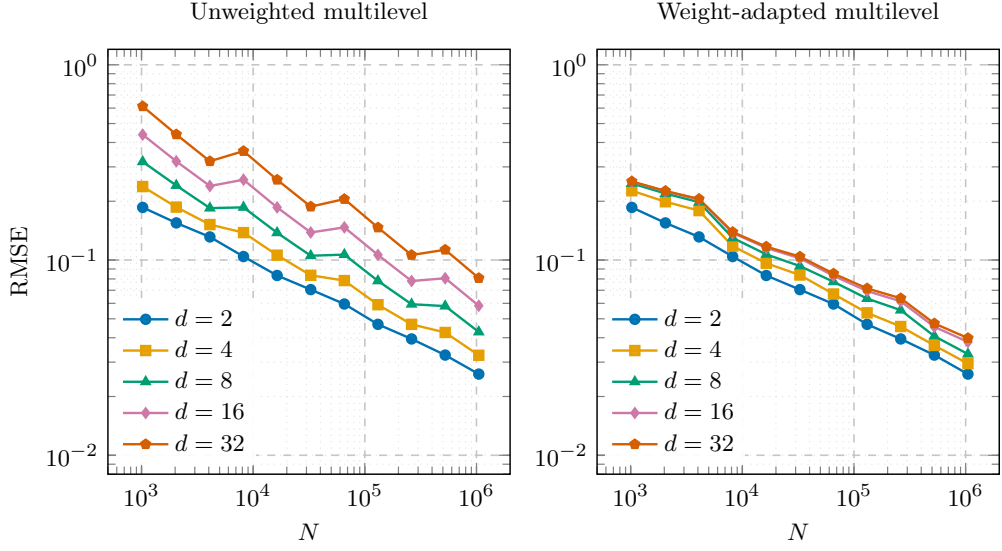

\begin{figure*}[t]
	\centering
	\begingroup
	\definecolor{dimBlue}{RGB}{0,114,178}
	\definecolor{dimOrange}{RGB}{230,159,0}
	\definecolor{dimGreen}{RGB}{0,158,115}
	\begin{tikzpicture}
		\begin{groupplot}[
			group style={group size=2 by 1,horizontal sep=1.35cm},
			width=0.46\textwidth,
			height=0.48\textwidth,
			xmode=log,
			ymode=log,
			log basis x=10,
			log basis y=10,
			xlabel={$N$},
			grid=both,
			major grid style={line width=0.6pt, draw=gray!45, dashed},
			minor grid style={line width=0.4pt, draw=gray!25, dotted},
			%minor x tick num=3, minor y tick num=3,
			%grid= both,
			%grid style={gray!25,line width=0.4pt,densely dashed},
%			axis line style={line width=0.55pt},
%			tick style={line width=0.45pt},
%			tick align=inside,
%			label style={font=\small},
			tick label style={font=\scriptsize},
			title style={font=\small},
			every axis plot/.append style={solid,line width=0.9pt},
			legend style={
				draw=none,
				fill=white,
				fill opacity=0.94,
				text opacity=1,
				font=\scriptsize,
				cells={anchor=west},
				inner sep=2pt,
				at={(0.02,0.02)},
				anchor=south west
			}
			]
			
			% Left panel: updated d=4 CBC bandwidth.
			\nextgroupplot[
			title={$d=4$},
			ylabel={$L_2$-error},
			xmin=5e3,
			xmax=2e7,
			ymin=2e-5,
			ymax=10,
			]
			\addplot[dimBlue,mark=*,mark size=1.7pt] coordinates {
				(10141,1.454)
				(32077,8.194e-1)
				(99793,4.670e-1)
				(316801,2.622e-1)
				(995329,1.492e-1)
				(3161089,8.337e-2)
				(9966529,4.680e-2)
			};
			\addlegendentry{MC}
			
			\addplot[dimOrange,mark=square*,mark size=1.7pt] coordinates {
				(10141,1.207)
				(32077,4.213e-1)
				(99793,1.843e-1)
				(316801,7.626e-2)
				(995329,2.548e-2)
				(3161089,7.723e-3)
				(9966529,2.552e-3)
			};
			\addlegendentry{Sobol}
			
			\addplot[dimGreen,mark=triangle*,mark size=1.7pt] coordinates {
				(10141,1.989)
				(32077,7.246e-1)
				(99793,1.628e-1)
				(316801,6.319e-2)
				(995329,4.432e-3)
				(3161089,4.439e-4)
				(9966529,3.679e-5)
			};
			\addlegendentry{CBC}
			
			% Right panel: independent vertical limits and visible tick labels.
			\nextgroupplot[
			title={$d=6$},
			xmin=5e3,
			xmax=2e7,
			ymin=2e-5,
			ymax=10,
			]
			\addplot[dimBlue,mark=*,mark size=1.7pt] coordinates {
				(10141,4.734)
				(32077,2.659)
				(99793,1.509)
				(316801,8.464e-1)
				(995329,4.783e-1)
				(3161089,2.680e-1)
				(9966529,1.510e-1)
			};
			\addlegendentry{MC}
			
			\addplot[dimOrange,mark=square*,mark size=1.7pt] coordinates {
				(10141,4.653)
				(32077,2.363)
				(99793,1.155)
				(316801,4.870e-1)
				(995329,1.661e-1)
				(3161089,6.970e-2)
				(9966529,2.379e-2)
			};
			\addlegendentry{Sobol}
			
			\addplot[dimGreen,mark=triangle*,mark size=1.7pt] coordinates {
				(10141,4.530)
				(32077,2.147)
				(99793,8.645e-1)
				(316801,5.149e-1)
				(995329,9.636e-2)
				(3161089,2.641e-2)
				(9966529,5.280e-3)
			};
			\addlegendentry{CBC lattice}
			
		\end{groupplot}
	\end{tikzpicture}
	\endgroup
	\caption{$L_2$ approximation errors for $f_4(\boldsymbol{x})$ using weighted ANOVA quasi-interpolation with MC, Sobol, and CBC discretization.}
	\label{fig:b3-anova-three-rule-d4-d6}
\end{figure*}
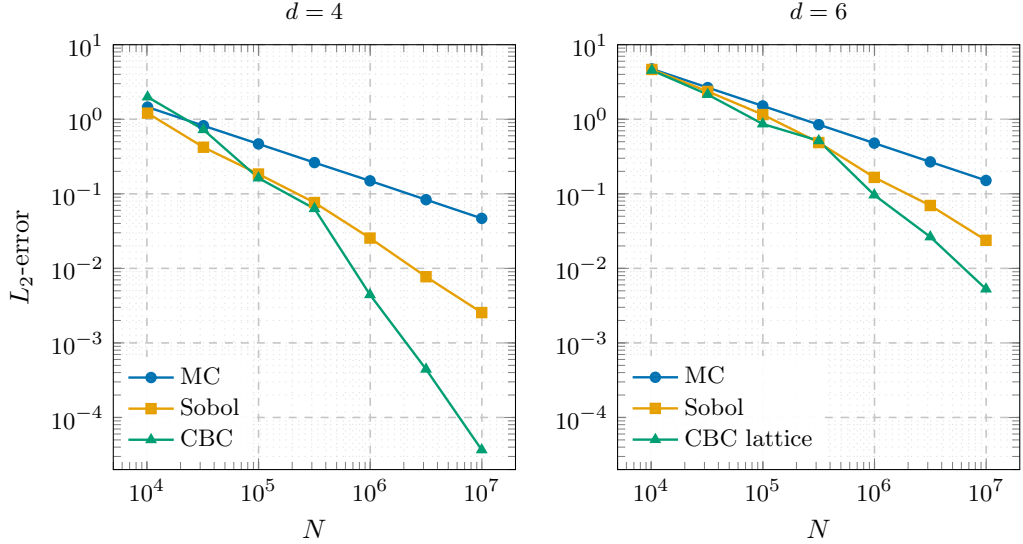
% END original figures/multilevel_dimension_convergence.tex

\section{Conclusion}

Moment and decay assumptions on the univariate function yield high-order approximation by normalized periodic convolution kernels. Once the convolution integral is discretized, however, convolution accuracy alone does not determine the error of the resulting quasi-interpolant. Our full-grid and independent-sampling estimates quantify this additional error. In the latter case, an exact mean-square identity isolates the contribution of kernel \(L_2\)-energy to the sampling variance, explaining the unfavorable dimension dependence of tensor-product kernels at fine scales.
This observation motivates the weight-adapted ANOVA constructions, which allocate resolution according to coordinate importance and restrict the interaction order. Under summable coordinate weights and fixed superposition order, both the single-level and multilevel quasi-interpolants admit dimension-uniform error bounds. For the multilevel scheme, the algebraic RMSE exponent \(\alpha/(2\alpha+1)\) is independent of \(q\), although the logarithmic factor and constants may depend on it.
The numerical results are consistent with the theoretical convergence behavior. Multilevel refinement reduces the convolution error and substantially lowers the kernel norm. Sobol and CBC discretizations give smaller total errors than Monte Carlo sampling. We will investigate other QMC rules for the convolution integral and seek optimal bounds that balance quadrature error with convolution bias in high-dimensional periodic approximation.

\section*{Acknowledgment}

The work was partially supported by National Natural Science Foundation of China (No. 12571407),
Basic Research Program of Jiangsu (No. BK20252037), the Fundamental Research Funds for the Central Universities (No.30926WZCT202601009), the Stability Support Program for Basic Research (No. WZJC202601003) and a Jiangsu Shuangchuang Team program (No.
JSSCTD202449).

\backmatter

\begin{appendices}

\section{\textbf{Proof of Theorem \ref{thm:jfg-full-grid-benchmark}}}
\label{Appendix1}
\begin{proof}
	 We divide the proof into two steps.
	 
	 \textbf{Step I: univariate estimate.} Write $H^s(\T):=H^s_{\text{mix}}(\T)$, and let
	$$J_n:=\{-\lfloor n/2\rfloor,\ldots,\lceil n/2\rceil-1\},$$ and define the discrete Fourier coefficients by
	$$\widetilde f_n(\ell):=\frac{1}{n}\sum_{j=0}^{n-1}
	f\big(\frac{j}{n}\big)e^{-2\pi\mathrm i \ell j/n}.$$
	Since \(s>1/2\), the Fourier series of \(f\in H^s(\T)\) is absolutely convergent, and discrete orthogonality yields the aliasing identity
	\[
	\widetilde f_n(\ell)
	=
	\sum_{k\in\Z}\widehat f(\ell+kn),
	\quad \ell\in J_n.
	\]
	Let \(P_nf:=\sum_{\ell\in J_n}\widetilde f_n(\ell)\chi_\ell\). Then the Fourier representation of the univariate full-grid rule is
	\[
	\mathcal Q_{n,\rho}f
	=
	\mathcal C_\rho P_nf+\mathcal R_{n,\rho}f,
	\]
	where
	\[
	\mathcal R_{n,\rho}f
	:=
	\sum_{\ell\in J_n}\widetilde f_n(\ell)
	\sum_{q\neq0}
	\widehat\psi_\rho(\ell+qn)\chi_{\ell+qn}.
	\]
	The error can be decomposed into three parts:
	\begin{equation}\label{eq:err_decomposition_1d}
			f-\mathcal Q_{n,\rho}f
		=
		(f-\mathcal C_\rho f)
		+\mathcal C_\rho(f-P_nf)
		-\mathcal R_{n,\rho}f.
	\end{equation}

	The first term on the right-hand side is the convolution approximation error, which, by Theorem~\ref{thm:jfg-continuous-bias} with $d=1$, satisfies
	$\|f-\mathcal C_\rho f\|_{L_2(\T)}
	\leq
	C\rho^a\|f\|_{H^s(\T)}$. In addition, Assumption~\ref{ass:jfg-kernel1} implies
	$\|\mathcal C_\rho\|_{L_2(\T)\to L_2(\T)}
	\leq c_2$.
	
	Next, we bound the second term on the right-hand side of \eqref{eq:err_decomposition_1d}. For \(\ell\in J_n\) and \(k\neq0\), we have
	\(|\ell+k n|\geq n|k|/2\), and hence
	$\sum_{k\neq0}\langle \ell+k n\rangle^{-2s}
	\leq C_s n^{-2s}$.
	The weighted Cauchy--Schwarz inequality gives
	\[
	\Big|
	\sum_{k\neq0}\widehat f(\ell+kn)
	\Big|^2
	\leq
	C_s n^{-2s}
	\sum_{k\neq0}
	\langle \ell+kn\rangle^{2s}
	|\widehat f(\ell+kn)|^2.
	\]
	The sets \(\{\ell+kn:k\in\Z\}\), \(\ell\in J_n\), are the disjoint residue classes modulo \(n\). Summing over \(\ell\in J_n\) and using
	$\sum_{j\notin J_n}|\widehat f(j)|^2
	\leq
	C_s n^{-2s}\|f\|_{H^s(\T)}^2$,
	we obtain
	\[
	\big\|
	f-P_nf
	\big\|_{L_2(\T)}
	\leq
	C_s n^{-s}\|f\|_{H^s(\T)}.
	\]
	Furthermore, an application of the weighted Cauchy--Schwarz inequality again yields
	\[
	|\widetilde f_n(\ell)|^2
	\leq
	\Big(
	\sum_{k\in\Z}\langle \ell+kn\rangle^{-2s}
	\Big)
	\Big(
	\sum_{k\in\Z}
	\langle \ell+kn\rangle^{2s}
	|\widehat f(\ell+kn)|^2
	\Big).
	\]
	Since
	$\sup_{\ell\in J_n}
	\sum_{k\in\Z}\langle \ell+kn\rangle^{-2s}
	\leq C_s$,
	summing over the residue classes gives
	\begin{equation}
		\sum_{\ell\in J_n}|\widetilde f_n(\ell)|^2
		\leq
		C_s\|f\|_{H^s(\T)}^2.
		\label{eq:jfg-full-grid-aliasing}
	\end{equation}

	Then, we estimate $\mathcal R_{n,\rho}f$. By the decay condition in
	Assumption~\ref{ass:jfg-kernel1}, for \(\ell\in J_n\) and \(q\neq0\), we have
	\[
	|\widehat\psi_\rho(\ell+qn)|^2
	\leq
	C(n\rho)^{-4\kappa}|q|^{-4\kappa}.
	\]
	Since \(4\kappa>1\), it follows that \(\psi_\rho\in L_2(\T)\) and
	\[
	\sum_{q\neq0}
	|\widehat\psi_\rho(\ell+qn)|^2
	\leq
	C(n\rho)^{-4\kappa},
	\quad \ell\in J_n.
	\]
	%The frequencies \(k+qn\), with \(k\in J_n\) and \(q\neq0\), are pairwise distinct.
	 Parseval's identity and
	\eqref{eq:jfg-full-grid-aliasing} imply
	\begin{align*}
		\|\mathcal R_{n,\rho}f\|_{L_2(\T)}^2
		&=
		\sum_{\ell\in J_n}
		|\widetilde f_n(\ell)|^2
		\sum_{q\neq0}
		|\widehat\psi_\rho(\ell+qn)|^2
		\\
		&\leq
		C(n\rho)^{-4\kappa}
		\sum_{\ell\in J_n}|\widetilde f_n(\ell)|^2
		\\
		&\leq
		C(n\rho)^{-4\kappa}\|f\|_{H^s(\T)}^2.
	\end{align*}
	
    Combining the above results, we conclude that
	\begin{equation}
		\|f-\mathcal Q_{n,\rho}f\|_{L_2(\T)}
		\leq
		C\bigl(
		\rho^a+n^{-s}+(n\rho)^{-2\kappa}
		\bigr)
		\|f\|_{H^s(\T)}.
		\label{eq:jfg-full-grid-one-dimensional}
	\end{equation}
	
	\smallskip
	\textbf{Step II: multivariate estimate.} Note that the following Hilbert tensor-product representations hold:
	\[
	H^s_{\mathrm{mix}}(\T^d)
	=
	\widehat{\bigotimes}_{j=1}^d H^s(\T),
	\quad
	L_2(\T^d)
	=
	\widehat{\bigotimes}_{j=1}^d L_2(\T).
	\]
	Because \(s>1/2\), point evaluation is bounded on \(H^s(\T)\). Together with \(\psi_\rho\in L_2(\T)\), this shows that the univariate rule defines a bounded operator
	\(Q:H^s(\T)\to L_2(\T)\).
	
	The tensor-product structure of the grid and kernel gives
	\(\mathcal Q_{n,\rho}=Q^{\otimes d}\). Similarly, the canonical embedding from \(H^s_{\mathrm{mix}}(\T^d)\) into \(L_2(\T^d)\) is \(I^{\otimes d}\), where
	\(I:H^s(\T)\to L_2(\T)\) denotes the univariate embedding. Using the telescoping identity, we get
	\[
	I^{\otimes d}-Q^{\otimes d}
	=
	\sum_{j=1}^d
	Q^{\otimes(j-1)}
	\otimes(I-Q)
	\otimes I^{\otimes(d-j)},
	\]
	and
	\[
	\|I^{\otimes d}-Q^{\otimes d}\|
	\leq
	\sum_{j=1}^d
	\|Q\|^{j-1}\|I-Q\|\|I\|^{d-j}.
	\]
	By \eqref{eq:jfg-full-grid-one-dimensional}, we obtain
	\[
	\|I-Q\|
	\leq
	C\bigl(
	\rho^a+n^{-s}+(n\rho)^{-2\kappa}
	\bigr).
	\]
	Moreover, \(\|Q\|\leq\|I\|+\|I-Q\|\). Because \(\|I\|\) and \(\|Q\|\) are uniformly bounded, we have
	\[
	\|I^{\otimes d}-Q^{\otimes d}\|
	\leq
	C_d\bigl(
	\rho^a+n^{-s}+(n\rho)^{-2\kappa}
	\bigr),
	\]
	which completes the proof.
\end{proof}

\section{\textbf{Proof of Lemma \ref{lem:mlaqi-bandpass}}}
\label{Appendix2}
\begin{proof}
%	For \(n\geq1\), according to
%	Assumption~\ref{ass:jfg-kernel}, we can get
%	\begin{align*}
%		|d_n(\ell)|
%		&\leq
%		\bigl|1-\widehat\psi_{\rho_0 2^{-n}}(\ell)\bigr|
%		+
%		\bigl|1-\widehat\psi_{\rho_0 2^{-(n-1)}}(\ell)\bigr|\\
%		&\leq
%		C\min\left\{
%		1,\,
%		(\rho_0 2^{-n}|\ell|)^m
%		\right\}.
%	\end{align*}
%	Applying the decay estimate to the same two multipliers yields
%	\[
%	|d_n(\ell)|
%	\leq
%	C(1+\rho_0 2^{-n}|\ell|)^{-2\kappa}.
%	\]
%	This proves \eqref{eq:mlaqi-bandpass-bound}, while
%	the case $n=0$ follows directly from the definition of
%	\(\vartheta_0\).
	
	Let
	$N(\ell)
	:=
	\max\left\{
	0,\,
	\left\lceil\log_2(\rho_0|\ell|)\right\rceil
	\right\}$.
	For \(1\leq n\leq N(\ell)\), the decay estimate gives
	\[
	2^{\beta n}|d_n(\ell)|
	\leq
	C2^{\beta N(\ell)}
	2^{-(\beta+2\kappa)(N(\ell)-n)}.
	\]
	For \(n>N(\ell)\), we have
	\[
	2^{\beta n}|d_n(\ell)|
	\leq
	C2^{\beta N(\ell)}
	2^{-(m-\beta)(n-N(\ell))}.
	\]
	The first geometric series converges because $\beta+2\kappa>0$, and the second because \(\beta<m\). The term \(n=0\) is uniformly bounded, and \(2^{N(\ell)}\leq C\langle \ell\rangle\), proving \eqref{eq:mlaqi-weighted-bandpass-sum}.
	
	It remains to prove almost orthogonality. Parseval's identity yields
	$\|\vartheta_n\|_2^2 = \sum_{\ell\in\Z}|d_n(\ell)|^2
		\leq C2^n$.
	This proves \eqref{eq:mlaqi-almost-orthogonality} when \(n=v\). Now suppose that \(0\leq n<v\), and write \(h:=v-n\) and
	\(a_n:=\rho_0 2^{-n}\). Using decay for \(d_n\), consistency for \(d_v\), and Parseval's identity, we obtain
	\begin{equation}
		\bigl|
		\langle\vartheta_n,\vartheta_v\rangle
		\bigr|
		\leq
		C\sum_{0\neq \ell\in\Z}
		(1+a_n|\ell|)^{-2\kappa}
		\min\left\{
		1,\,
		(2^{-h}a_n|\ell|)^m
		\right\}.
		\label{eq:mlaqi-cross-sum}
	\end{equation}
	Split the sum into the ranges \(0<a_n|\ell|\leq1\),
	\(1<a_n|\ell|\leq2^h\), and \(a_n|\ell|>2^h\). Comparison with the
	corresponding integrals gives
	\begin{align*}
		\sum_{0<a_n|\ell|\leq1}
		(1+a_n|\ell|)^{-2\kappa}
		(2^{-h}a_n|\ell|)^m
		&\leq
		Ca_n^{-1}2^{-mh},
		\\
		\sum_{1<a_n|\ell|\leq2^h}
		(a_n|\ell|)^{-2\kappa}
		(2^{-h}a_n|\ell|)^m
		&\leq
		Ca_n^{-1}2^{-mh}
		\int_1^{2^h}x^{m-2\kappa}\,dx,
        \\
		\sum_{a_n|\ell|>2^h}
		(a_n|\ell|)^{-2\kappa}
		&\leq
		Ca_n^{-1}
		\int_{2^h}^{\infty}x^{-2\kappa}\,dx.
	\end{align*}
	The last integral is finite because \(2\kappa>1\) and is of order \(2^{-(2\kappa-1)h}\). For the middle integral,
	\begin{equation*}
		2^{-mh}\int_1^{2^h}x^{m-2\kappa}\,dx
		\leq
		C
		\begin{cases}
			2^{-mh},
			&m<2\kappa-1,\\
			(h+1)2^{-mh},
			&m=2\kappa-1,\\
			2^{-(2\kappa-1)h},
			&m>2\kappa-1.
		\end{cases}
	\end{equation*}
	Since \(a_n^{-1}=\rho_0^{-1}2^n\), the right-hand side of
	\eqref{eq:mlaqi-cross-sum} is bounded by
	\[
	C2^n
	\left[
	2^{-mh}
	+
	2^{-(2\kappa-1)h}
	+
	(h+1)2^{-mh}
	\boldsymbol 1_{\{m=2\kappa-1\}}
	\right].
	\]
	For every \(0<\varepsilon<\min\{m,2\kappa-1\}\), this is at most \(C_\varepsilon2^n2^{-\varepsilon h}\). Symmetry in \(n\) and \(v\) completes the proof.
\end{proof}

\section{\textbf{Proof of Proposition \ref{prop:mlaqi-multiplier-energy}}}
\label{Appendix3}
\begin{proof}
	Because \(\widehat{\vartheta}_n(0)=0\), a tensor-product summand
	indexed by \(\boldsymbol u\) can contribute only at frequencies whose
	coordinate support is \(\boldsymbol u\). Summing its multiplier over
	the admissible local levels proves \eqref{eq:mlaqi-multiplier}.
	
	If \(q=0\), then
	\(\Phi_{0,L}^{\mathrm{ml},(d)}\equiv1\), and
	\eqref{eq:mlaqi-energy} is immediate. Assume henceforth that
	\(q\geq1\).
	
	For the energy estimate, fix
	\(\boldsymbol u\subseteq[d]\) with
	\(|\boldsymbol u|=r\), and define
	\[
	\Theta_{\boldsymbol u,B}(\boldsymbol x)
	:=
	\sum_{\substack{
			\boldsymbol n\in\mathbb N_0^{\boldsymbol u}\\
			|\boldsymbol n|_1\leq B}}
	\prod_{j\in\boldsymbol u}
	\vartheta_{n_j}(x_j).
	\]
	Choose
	\(0<\varepsilon<\min\{m,2\kappa-1\}\). Tensor-product factorization
	and \eqref{eq:mlaqi-almost-orthogonality} lead to
	\begin{align*}
		\|\Theta_{\boldsymbol u,B}\|_2^2
		&\leq
		C^r
		\sum_{\substack{
				\boldsymbol n,\boldsymbol v
				\in\mathbb N_0^{\boldsymbol u}\\
				|\boldsymbol n|_1\leq B,\,
				|\boldsymbol v|_1\leq B}}
		2^{\sum_{j\in\boldsymbol u}\min\{n_j,v_j\}}
		2^{-\varepsilon|\boldsymbol n-\boldsymbol v|_1}
		\\
		&=
		C^r
		\sum_{\substack{
				\boldsymbol n,\boldsymbol v
				\in\mathbb N_0^{\boldsymbol u}\\
				|\boldsymbol n|_1\leq B,\,
				|\boldsymbol v|_1\leq B}}
		2^{\frac12(|\boldsymbol n|_1+|\boldsymbol v|_1)}
		2^{-(\varepsilon+1/2)
			|\boldsymbol n-\boldsymbol v|_1}.
	\end{align*}
	For fixed \(\boldsymbol n\), the inequality
	$2^{|\boldsymbol v|_1/2}
	\leq
	2^{|\boldsymbol n|_1/2}
	2^{|\boldsymbol n-\boldsymbol v|_1/2}$
	gives the bound
	\[
	2^{\frac12(|\boldsymbol n|_1+|\boldsymbol v|_1)}
	2^{-(\varepsilon+1/2)|\boldsymbol n-\boldsymbol v|_1}
	\leq
	2^{|\boldsymbol n|_1}
	2^{-\varepsilon|\boldsymbol n-\boldsymbol v|_1}.
	\]
	The series $\sum_{\boldsymbol h\in\mathbb Z^r}2^{-\varepsilon|\boldsymbol h|_1}=(\sum_{h\in\mathbb Z}2^{-\varepsilon|h|})^r$ converges, and its dependence on \(r\) can be included in \(C^r\).
	Summing first over \(\boldsymbol v\) yields
	\begin{align}
		\|\Theta_{\boldsymbol u,B}\|_2^2
		&\leq
		C^r
		\sum_{\substack{
				\boldsymbol n\in\mathbb N_0^{\boldsymbol u}\\
				|\boldsymbol n|_1\leq B}}
		2^{|\boldsymbol n|_1}
		=
		C^r
		\sum_{t=0}^{B}
		\binom{t+r-1}{r-1}2^t
		\leq
		C^r2^B(B+1)^{r-1}.
		\label{eq:mlaqi-component-energy}
	\end{align}
	
	Applying \eqref{eq:mlaqi-component-energy} with
	\(B=L-\lambda_{\boldsymbol u}\) gives
	\begin{align}
		\|\Phi_{q,L}^{\mathrm{ml},(d)}\|_2^2
		&\leq
		1+
		\sum_{r=1}^{\min\{q,d\}}
		\sum_{\substack{
				\boldsymbol u\subseteq[d]\\
				|\boldsymbol u|=r,\ \lambda_{\boldsymbol u}\leq L}}
		C^r2^{L-\lambda_{\boldsymbol u}}
		(L-\lambda_{\boldsymbol u}+1)^{r-1}
		\notag\\
		&\leq
		1+
		2^L
		\sum_{r=1}^{\min\{q,d\}}
		C^r(L+1)^{r-1}
		\sum_{\substack{
				\boldsymbol u\subseteq[d]\\
				|\boldsymbol u|=r}}
		2^{-\lambda_{\boldsymbol u}}.
		\label{eq:mlaqi-global-energy-transition}
	\end{align}
	Here \(\lambda_j\geq0\), since \(0<\gamma_j\leq1\).
	
	The comparison \eqref{eq:mlaqi-entry-level-comparison} and the elementary symmetric sum estimate imply
	\begin{align}
		\sum_{\substack{
				\boldsymbol u\subseteq[d]\\
				|\boldsymbol u|=r}}
		2^{-\lambda_{\boldsymbol u}}
		&\leq
		\frac{1}{r!}
		\left(
		\sum_{j=1}^d2^{-\lambda_j}
		\right)^r
		\leq
		\frac{(2G_\alpha)^r}{r!}.
		\label{eq:mlaqi-entry-level-symmetric-sum}
	\end{align}
	Combining this estimate with
	\eqref{eq:mlaqi-global-energy-transition} gives
	\[
	\|\Phi_{q,L}^{\mathrm{ml},(d)}\|_2^2
	\leq
	1+
	2^L
	\sum_{r=1}^{q}
	C^r(L+1)^{r-1}
	\frac{(2G_\alpha)^r}{r!}.
	\]
	Since \(q\) is fixed and \(1\leq r\leq q\), we have 
	$(L+1)^{r-1}
	\leq
	(L+1)^{q-1}$.
	The finite sum over \(r\), together with
	\(1\leq2^L(L+1)^{q-1}\), now leads to
	\[
	\|\Phi_{q,L}^{\mathrm{ml},(d)}\|_2^2
	\leq
	C2^L(L+1)^{q-1}.
	\]
	This completes the proof.
\end{proof}

\end{appendices}
% Alphabetical reference order is configured in the local sn-mathphys-num.bst.
\bibliography{references}


\begin{thebibliography}{60}
\ifx \bisbn   \undefined \def \bisbn  #1{ISBN #1}\fi
\ifx \binits  \undefined \def \binits#1{#1}\fi
\ifx \bauthor  \undefined \def \bauthor#1{#1}\fi
\ifx \batitle  \undefined \def \batitle#1{#1}\fi
\ifx \bjtitle  \undefined \def \bjtitle#1{#1}\fi
\ifx \bvolume  \undefined \def \bvolume#1{\textbf{#1}}\fi
\ifx \byear  \undefined \def \byear#1{#1}\fi
\ifx \bissue  \undefined \def \bissue#1{#1}\fi
\ifx \bfpage  \undefined \def \bfpage#1{#1}\fi
\ifx \blpage  \undefined \def \blpage #1{#1}\fi
\ifx \burl  \undefined \def \burl#1{\textsf{#1}}\fi
\ifx \doiurl  \undefined \def \doiurl#1{\url{https://doi.org/#1}}\fi
\ifx \betal  \undefined \def \betal{\textit{et al.}}\fi
\ifx \binstitute  \undefined \def \binstitute#1{#1}\fi
\ifx \binstitutionaled  \undefined \def \binstitutionaled#1{#1}\fi
\ifx \bctitle  \undefined \def \bctitle#1{#1}\fi
\ifx \beditor  \undefined \def \beditor#1{#1}\fi
\ifx \bpublisher  \undefined \def \bpublisher#1{#1}\fi
\ifx \bbtitle  \undefined \def \bbtitle#1{#1}\fi
\ifx \bedition  \undefined \def \bedition#1{#1}\fi
\ifx \bseriesno  \undefined \def \bseriesno#1{#1}\fi
\ifx \blocation  \undefined \def \blocation#1{#1}\fi
\ifx \bsertitle  \undefined \def \bsertitle#1{#1}\fi
\ifx \bsnm \undefined \def \bsnm#1{#1}\fi
\ifx \bsuffix \undefined \def \bsuffix#1{#1}\fi
\ifx \bparticle \undefined \def \bparticle#1{#1}\fi
\ifx \barticle \undefined \def \barticle#1{#1}\fi
\bibcommenthead
\ifx \bconfdate \undefined \def \bconfdate #1{#1}\fi
\ifx \botherref \undefined \def \botherref #1{#1}\fi
\ifx \url \undefined \def \url#1{\textsf{#1}}\fi
\ifx \bchapter \undefined \def \bchapter#1{#1}\fi
\ifx \bbook \undefined \def \bbook#1{#1}\fi
\ifx \bcomment \undefined \def \bcomment#1{#1}\fi
\ifx \oauthor \undefined \def \oauthor#1{#1}\fi
\ifx \citeauthoryear \undefined \def \citeauthoryear#1{#1}\fi
\ifx \endbibitem  \undefined \def \endbibitem {}\fi
\ifx \bconflocation  \undefined \def \bconflocation#1{#1}\fi
\ifx \arxivurl  \undefined \def \arxivurl#1{\textsf{#1}}\fi
\csname PreBibitemsHook\endcsname

%%% 1
\bibitem[\protect\citeauthoryear{Addy et~al.}{2026}]{AddyLatzTeckentrup2026}
\begin{barticle}
\bauthor{\bsnm{Addy}, \binits{E.J.}},
\bauthor{\bsnm{Latz}, \binits{J.}},
\bauthor{\bsnm{Teckentrup}, \binits{A.L.}}:
\batitle{Lengthscale-informed sparse grids for kernel methods in high
  dimensions}.
\bjtitle{Numer. Math.}
\bvolume{158},
\bfpage{799}--\blpage{846}
(\byear{2026})
\doiurl{10.1007/s00211-026-01527-5}
\end{barticle}
\endbibitem

%%% 2
\bibitem[\protect\citeauthoryear{Bungartz and
  Griebel}{2004}]{bungartz_2004Acta_sparse}
\begin{barticle}
\bauthor{\bsnm{Bungartz}, \binits{H.-J.}},
\bauthor{\bsnm{Griebel}, \binits{M.}}:
\batitle{Sparse grids}.
\bjtitle{Acta Numer.}
\bvolume{13},
\bfpage{147}--\blpage{269}
(\byear{2004})
\doiurl{10.1017/S0962492904000182}
\end{barticle}
\endbibitem

%%% 3
\bibitem[\protect\citeauthoryear{Byrenheid
  et~al.}{2017}]{byrenheid_2017NM_tight}
\begin{barticle}
\bauthor{\bsnm{Byrenheid}, \binits{G.}},
\bauthor{\bsnm{K{\"a}mmerer}, \binits{L.}},
\bauthor{\bsnm{Ullrich}, \binits{T.}},
\bauthor{\bsnm{Volkmer}, \binits{T.}}:
\batitle{Tight error bounds for rank-1 lattice sampling in spaces of hybrid
  mixed smoothness}.
\bjtitle{Numer. Math.}
\bvolume{136}(\bissue{4}),
\bfpage{993}--\blpage{1034}
(\byear{2017})
\doiurl{10.1007/s00211-016-0861-7}
\end{barticle}
\endbibitem

%%% 4
\bibitem[\protect\citeauthoryear{Cai and Goda}{2026}]{CaiGoda2026}
\begin{barticle}
\bauthor{\bsnm{Cai}, \binits{M.}},
\bauthor{\bsnm{Goda}, \binits{T.}}:
\batitle{A note on approximation in weighted {K}orobov spaces via multiple
  rank-1 lattices}.
\bjtitle{J. Complexity}
\bvolume{96},
\bfpage{102054}
(\byear{2026})
\doiurl{10.1016/j.jco.2026.102054}
\end{barticle}
\endbibitem

%%% 5
\bibitem[\protect\citeauthoryear{Cai et~al.}{2026}]{CaiGodaKazashi2026}
\begin{barticle}
\bauthor{\bsnm{Cai}, \binits{M.}},
\bauthor{\bsnm{Goda}, \binits{T.}},
\bauthor{\bsnm{Kazashi}, \binits{Y.}}:
\batitle{{$L_2$}-approximation using randomized lattice algorithms}.
\bjtitle{J. Approx. Theory}
\bvolume{318},
\bfpage{106301}
(\byear{2026})
\doiurl{10.1016/j.jat.2026.106301}
\end{barticle}
\endbibitem

%%% 6
\bibitem[\protect\citeauthoryear{Cheney and Lei}{1994}]{CheneyLei1994}
\begin{bchapter}
\bauthor{\bsnm{Cheney}, \binits{E.W.}},
\bauthor{\bsnm{Lei}, \binits{J.}}:
\bctitle{Quasi-interpolation on irregular points}.
In: \beditor{\bsnm{Zahar}, \binits{R.V.M.}} (ed.)
\bbtitle{Approximation and Computation: A Festschrift in Honor of Walter
  Gautschi}.
\bsertitle{International Series of Numerical Mathematics},
vol. \bseriesno{119},
pp. \bfpage{121}--\blpage{135}.
\bpublisher{Birkh{\"a}user},
\blocation{Boston}
(\byear{1994}).
\doiurl{10.1007/978-1-4684-7415-2_8}
\end{bchapter}
\endbibitem

%%% 7
\bibitem[\protect\citeauthoryear{Cheney et~al.}{1992}]{CheneyLightXu1992}
\begin{bchapter}
\bauthor{\bsnm{Cheney}, \binits{E.W.}},
\bauthor{\bsnm{Light}, \binits{W.A.}},
\bauthor{\bsnm{Xu}, \binits{Y.}}:
\bctitle{On kernels and approximation orders}.
In: \beditor{\bsnm{Anastassiou}, \binits{G.A.}} (ed.)
\bbtitle{Approximation Theory}.
\bsertitle{Lecture Notes in Pure and Applied Mathematics},
vol. \bseriesno{138},
pp. \bfpage{227}--\blpage{242}.
\bpublisher{Marcel Dekker},
\blocation{New York}
(\byear{1992})
\end{bchapter}
\endbibitem

%%% 8
\bibitem[\protect\citeauthoryear{Cohen and
  DeVore}{2015}]{cohen_2015ActaNumer_approximation}
\begin{barticle}
\bauthor{\bsnm{Cohen}, \binits{A.}},
\bauthor{\bsnm{DeVore}, \binits{R.}}:
\batitle{Approximation of high-dimensional parametric {PDE}s}.
\bjtitle{Acta Numer.}
\bvolume{24},
\bfpage{1}--\blpage{159}
(\byear{2015})
\doiurl{10.1017/S0962492915000033}
\end{barticle}
\endbibitem

%%% 9
\bibitem[\protect\citeauthoryear{Cools et~al.}{2020}]{CoolsKuoNuyensSloan2020}
\begin{bchapter}
\bauthor{\bsnm{Cools}, \binits{R.}},
\bauthor{\bsnm{Kuo}, \binits{F.Y.}},
\bauthor{\bsnm{Nuyens}, \binits{D.}},
\bauthor{\bsnm{Sloan}, \binits{I.H.}}:
\bctitle{Lattice algorithms for multivariate approximation in periodic spaces
  with general weight parameters}.
In: \beditor{\bsnm{Brenner}, \binits{S.C.}},
\beditor{\bsnm{Shparlinski}, \binits{I.}},
\beditor{\bsnm{Shu}, \binits{C.-W.}},
\beditor{\bsnm{Szyld}, \binits{D.B.}} (eds.)
\bbtitle{75 Years of Mathematics of Computation}.
\bsertitle{Contemporary Mathematics},
vol. \bseriesno{754},
pp. \bfpage{93}--\blpage{113}.
\bpublisher{American Mathematical Society}
(\byear{2020}).
\doiurl{10.1090/conm/754/15150}
\end{bchapter}
\endbibitem

%%% 10
\bibitem[\protect\citeauthoryear{Cools et~al.}{2021}]{CoolsKuoNuyensSloan2021}
\begin{barticle}
\bauthor{\bsnm{Cools}, \binits{R.}},
\bauthor{\bsnm{Kuo}, \binits{F.Y.}},
\bauthor{\bsnm{Nuyens}, \binits{D.}},
\bauthor{\bsnm{Sloan}, \binits{I.H.}}:
\batitle{Fast component-by-component construction of lattice algorithms for
  multivariate approximation with {POD} and {SPOD} weights}.
\bjtitle{Math. Comput.}
\bvolume{90}(\bissue{328}),
\bfpage{787}--\blpage{812}
(\byear{2021})
\doiurl{10.1090/mcom/3586}
\end{barticle}
\endbibitem

%%% 11
\bibitem[\protect\citeauthoryear{de~Boor and Fix}{1973}]{deBoorFix1973}
\begin{barticle}
\bauthor{\bsnm{Boor}, \binits{C.}},
\bauthor{\bsnm{Fix}, \binits{G.J.}}:
\batitle{Spline approximation by quasiinterpolants}.
\bjtitle{J. Approx. Theory}
\bvolume{8},
\bfpage{19}--\blpage{45}
(\byear{1973})
\doiurl{10.1016/0021-9045(73)90029-4}
\end{barticle}
\endbibitem

%%% 12
\bibitem[\protect\citeauthoryear{Dick et~al.}{2022}]{dick_book2022}
\begin{bbook}
\bauthor{\bsnm{Dick}, \binits{J.}},
\bauthor{\bsnm{Kritzer}, \binits{P.}},
\bauthor{\bsnm{Pillichshammer}, \binits{F.}}:
\bbtitle{Lattice Rules: Numerical Integration, Approximation, and Discrepancy}.
\bsertitle{Springer Series in Computational Mathematics},
vol. \bseriesno{58}.
\bpublisher{Springer},
\blocation{Cham}
(\byear{2022}).
\doiurl{10.1007/978-3-031-09951-9}
\end{bbook}
\endbibitem

%%% 13
\bibitem[\protect\citeauthoryear{D{\~u}ng}{2016}]{Dung2016}
\begin{barticle}
\bauthor{\bsnm{D{\~u}ng}, \binits{D.}}:
\batitle{Sampling and cubature on sparse grids based on a {B}-spline
  quasi-interpolation}.
\bjtitle{Found. Comput. Math.}
\bvolume{16}(\bissue{5}),
\bfpage{1193}--\blpage{1240}
(\byear{2016})
\doiurl{10.1007/s10208-015-9274-8}
\end{barticle}
\endbibitem

%%% 14
\bibitem[\protect\citeauthoryear{D{\~u}ng and
  Micchelli}{2013}]{DungMicchelli2013}
\begin{barticle}
\bauthor{\bsnm{D{\~u}ng}, \binits{D.}},
\bauthor{\bsnm{Micchelli}, \binits{C.A.}}:
\batitle{Multivariate approximation by translates of the {K}orobov function on
  {S}molyak grids}.
\bjtitle{J. Complexity}
\bvolume{29}(\bissue{6}),
\bfpage{424}--\blpage{437}
(\byear{2013})
\doiurl{10.1016/j.jco.2013.06.002}
\end{barticle}
\endbibitem

%%% 15
\bibitem[\protect\citeauthoryear{D{\~u}ng et~al.}{2018}]{dung_book2018}
\begin{bbook}
\bauthor{\bsnm{D{\~u}ng}, \binits{D.}},
\bauthor{\bsnm{Temlyakov}, \binits{V.}},
\bauthor{\bsnm{Ullrich}, \binits{T.}}:
\bbtitle{Hyperbolic Cross Approximation}.
\bsertitle{Advanced Courses in Mathematics-CRM Barcelona}.
\bpublisher{Birkh{\"a}user},
\blocation{Switzerland}
(\byear{2018}).
\doiurl{10.1007/978-3-319-92240-9}
\end{bbook}
\endbibitem

%%% 16
\bibitem[\protect\citeauthoryear{Durrande et~al.}{2013}]{DurrandeEtAl2013}
\begin{barticle}
\bauthor{\bsnm{Durrande}, \binits{N.}},
\bauthor{\bsnm{Ginsbourger}, \binits{D.}},
\bauthor{\bsnm{Roustant}, \binits{O.}},
\bauthor{\bsnm{Carraro}, \binits{L.}}:
\batitle{{ANOVA} kernels and {RKHS} of zero mean functions for model-based
  sensitivity analysis}.
\bjtitle{J. Multivariate Anal.}
\bvolume{115},
\bfpage{57}--\blpage{67}
(\byear{2013})
\doiurl{10.1016/j.jmva.2012.08.016}
\end{barticle}
\endbibitem

%%% 17
\bibitem[\protect\citeauthoryear{Dyn et~al.}{1992}]{DynJacksonLevinRon1992}
\begin{barticle}
\bauthor{\bsnm{Dyn}, \binits{N.}},
\bauthor{\bsnm{Jackson}, \binits{I.R.H.}},
\bauthor{\bsnm{Levin}, \binits{D.}},
\bauthor{\bsnm{Ron}, \binits{A.}}:
\batitle{On multivariate approximation by integer translates of a basis
  function}.
\bjtitle{Israel J. Math.}
\bvolume{78},
\bfpage{95}--\blpage{130}
(\byear{1992})
\doiurl{10.1007/BF02801574}
\end{barticle}
\endbibitem

%%% 18
\bibitem[\protect\citeauthoryear{Ebert and
  Pillichshammer}{2021}]{EbertPillichshammer2021}
\begin{barticle}
\bauthor{\bsnm{Ebert}, \binits{A.}},
\bauthor{\bsnm{Pillichshammer}, \binits{F.}}:
\batitle{Tractability of approximation in the weighted {K}orobov space in the
  worst-case setting---a complete picture}.
\bjtitle{J. Complexity}
\bvolume{67},
\bfpage{101571}
(\byear{2021})
\doiurl{10.1016/j.jco.2021.101571}
\end{barticle}
\endbibitem

%%% 19
\bibitem[\protect\citeauthoryear{Franz and Wendland}{2023}]{FranzWendland2023}
\begin{barticle}
\bauthor{\bsnm{Franz}, \binits{T.}},
\bauthor{\bsnm{Wendland}, \binits{H.}}:
\batitle{Multilevel quasi-interpolation}.
\bjtitle{IMA J. Numer. Anal.}
\bvolume{43}(\bissue{5}),
\bfpage{2934}--\blpage{2964}
(\byear{2023})
\doiurl{10.1093/imanum/drac059}
\end{barticle}
\endbibitem

%%% 20
\bibitem[\protect\citeauthoryear{Gao et~al.}{2020}]{GaoFasshauerSunZhou2020}
\begin{barticle}
\bauthor{\bsnm{Gao}, \binits{W.W.}},
\bauthor{\bsnm{Fasshauer}, \binits{G.E.}},
\bauthor{\bsnm{Sun}, \binits{X.P.}},
\bauthor{\bsnm{Zhou}, \binits{X.}}:
\batitle{Optimality and regularization properties of quasi-interpolation:
  deterministic and stochastic approaches}.
\bjtitle{SIAM J. Numer. Anal.}
\bvolume{58}(\bissue{4}),
\bfpage{2059}--\blpage{2078}
(\byear{2020})
\doiurl{10.1137/19M1266496}
\end{barticle}
\endbibitem

%%% 21
\bibitem[\protect\citeauthoryear{Gao et~al.}{2026}]{GaoHuSunWang2026}
\begin{barticle}
\bauthor{\bsnm{Gao}, \binits{W.W.}},
\bauthor{\bsnm{Hu}, \binits{L.}},
\bauthor{\bsnm{Sun}, \binits{Z.J.}},
\bauthor{\bsnm{Wang}, \binits{C.W.}}:
\batitle{Quasi-interpolation using generalized {G}aussian kernels}.
\bjtitle{J. Comput. Appl. Math.}
\bvolume{483},
\bfpage{117438}
(\byear{2026})
\doiurl{10.1016/j.cam.2026.117438}
\end{barticle}
\endbibitem

%%% 22
\bibitem[\protect\citeauthoryear{Gao et~al.}{2024}]{GaoWangSunFasshauer2024}
\begin{barticle}
\bauthor{\bsnm{Gao}, \binits{W.W.}},
\bauthor{\bsnm{Wang}, \binits{J.C.}},
\bauthor{\bsnm{Sun}, \binits{Z.J.}},
\bauthor{\bsnm{Fasshauer}, \binits{G.E.}}:
\batitle{Quasi-interpolation for high-dimensional function approximation}.
\bjtitle{Numer. Math.}
\bvolume{156}(\bissue{5}),
\bfpage{1855}--\blpage{1885}
(\byear{2024})
\doiurl{10.1007/s00211-024-01435-6}
\end{barticle}
\endbibitem

%%% 23
\bibitem[\protect\citeauthoryear{Gao and Zhou}{2020}]{GaoZhou2020}
\begin{barticle}
\bauthor{\bsnm{Gao}, \binits{W.W.}},
\bauthor{\bsnm{Zhou}, \binits{X.}}:
\batitle{Multiscale radial kernels with high-order generalized {S}trang--{F}ix
  conditions}.
\bjtitle{Numer. Algorithms}
\bvolume{85}(\bissue{2}),
\bfpage{427}--\blpage{448}
(\byear{2020})
\doiurl{10.1007/s11075-019-00820-9}
\end{barticle}
\endbibitem

%%% 24
\bibitem[\protect\citeauthoryear{Gerstner and
  Griebel}{1998}]{gerstner_1998NumerAlgor_numerical}
\begin{barticle}
\bauthor{\bsnm{Gerstner}, \binits{T.}},
\bauthor{\bsnm{Griebel}, \binits{M.}}:
\batitle{Numerical integration using sparse grids}.
\bjtitle{Numer. Algorithms}
\bvolume{18}(\bissue{3}),
\bfpage{209}--\blpage{232}
(\byear{1998})
\doiurl{10.1023/A:1019129717644}
\end{barticle}
\endbibitem

%%% 25
\bibitem[\protect\citeauthoryear{Griebel and
  Hamaekers}{2014}]{GriebelHamaekers2014}
\begin{bchapter}
\bauthor{\bsnm{Griebel}, \binits{M.}},
\bauthor{\bsnm{Hamaekers}, \binits{J.}}:
\bctitle{Fast discrete fourier transform on generalized sparse grids}.
In: \beditor{\bsnm{Garcke}, \binits{J.}},
\beditor{\bsnm{Pfl{\"u}ger}, \binits{D.}} (eds.)
\bbtitle{Sparse Grids and Applications-Munich 2012}.
\bsertitle{Lecture Notes in Computational Science and Engineering},
vol. \bseriesno{97},
pp. \bfpage{75}--\blpage{107}.
\bpublisher{Springer},
\blocation{Cham}
(\byear{2014}).
\doiurl{10.1007/978-3-319-04537-5_4} .
\burl{https://doi.org/10.1007/978-3-319-04537-5_4}
\end{bchapter}
\endbibitem

%%% 26
\bibitem[\protect\citeauthoryear{Griebel and
  Knapek}{2009}]{griebel_2009MCoM_optimized}
\begin{barticle}
\bauthor{\bsnm{Griebel}, \binits{M.}},
\bauthor{\bsnm{Knapek}, \binits{S.}}:
\batitle{Optimized general sparse grid approximation spaces for operator
  equations}.
\bjtitle{Math. Comput.}
\bvolume{78}(\bissue{268}),
\bfpage{2223}--\blpage{2257}
(\byear{2009})
\doiurl{10.1090/S0025-5718-09-02248-0}
\end{barticle}
\endbibitem

%%% 27
\bibitem[\protect\citeauthoryear{Griebel
  et~al.}{1992}]{GriebelSchneiderZenger1992}
\begin{bchapter}
\bauthor{\bsnm{Griebel}, \binits{M.}},
\bauthor{\bsnm{Schneider}, \binits{M.}},
\bauthor{\bsnm{Zenger}, \binits{C.}}:
\bctitle{A combination technique for the solution of sparse grid problems}.
In: \beditor{\bsnm{Groen}, \binits{P.}},
\beditor{\bsnm{Beauwens}, \binits{R.}} (eds.)
\bbtitle{Iterative Methods in Linear Algebra},
pp. \bfpage{263}--\blpage{281}.
\bpublisher{Elsevier},
\blocation{Amsterdam}
(\byear{1992})
\end{bchapter}
\endbibitem

%%% 28
\bibitem[\protect\citeauthoryear{Hubbert
  et~al.}{2023}]{HubbertJagerLevesley2023}
\begin{barticle}
\bauthor{\bsnm{Hubbert}, \binits{S.}},
\bauthor{\bsnm{J{\"a}ger}, \binits{J.}},
\bauthor{\bsnm{Levesley}, \binits{J.}}:
\batitle{Convergence of sparse grid {G}aussian convolution approximation for
  multi-dimensional periodic functions}.
\bjtitle{Appl. Comput. Harmon. Anal.}
\bvolume{62},
\bfpage{453}--\blpage{474}
(\byear{2023})
\doiurl{10.1016/j.acha.2022.10.005}
\end{barticle}
\endbibitem

%%% 29
\bibitem[\protect\citeauthoryear{Jeong et~al.}{2021}]{JeongKerseyYoon2021}
\begin{barticle}
\bauthor{\bsnm{Jeong}, \binits{B.}},
\bauthor{\bsnm{Kersey}, \binits{S.N.}},
\bauthor{\bsnm{Yoon}, \binits{J.}}:
\batitle{Approximation of multivariate functions on sparse grids by
  kernel-based quasi-interpolation}.
\bjtitle{SIAM J. Sci. Comput.}
\bvolume{43}(\bissue{2}),
\bfpage{953}--\blpage{979}
(\byear{2021})
\doiurl{10.1137/20M1318055}
\end{barticle}
\endbibitem

%%% 30
\bibitem[\protect\citeauthoryear{Joe and Kuo}{2003}]{JoeKuo2003}
\begin{barticle}
\bauthor{\bsnm{Joe}, \binits{S.}},
\bauthor{\bsnm{Kuo}, \binits{F.Y.}}:
\batitle{Remark on algorithm 659: Implementing {Sobol}'s quasirandom sequence
  generator}.
\bjtitle{ACM Trans. Math. Software}
\bvolume{29}(\bissue{1}),
\bfpage{49}--\blpage{57}
(\byear{2003})
\doiurl{10.1145/641876.641879}
\end{barticle}
\endbibitem

%%% 31
\bibitem[\protect\citeauthoryear{Kaarnioja et~al.}{2022}]{KaarniojaEtAl2022}
\begin{barticle}
\bauthor{\bsnm{Kaarnioja}, \binits{V.}},
\bauthor{\bsnm{Kazashi}, \binits{Y.}},
\bauthor{\bsnm{Kuo}, \binits{F.Y.}},
\bauthor{\bsnm{Nobile}, \binits{F.}},
\bauthor{\bsnm{Sloan}, \binits{I.H.}}:
\batitle{Fast approximation by periodic kernel-based lattice-point
  interpolation with application in uncertainty quantification}.
\bjtitle{Numer. Math.}
\bvolume{150},
\bfpage{33}--\blpage{77}
(\byear{2022})
\doiurl{10.1007/s00211-021-01242-3}
\end{barticle}
\endbibitem

%%% 32
\bibitem[\protect\citeauthoryear{K{\"a}mmerer}{2013}]{Kammerer2013}
\begin{barticle}
\bauthor{\bsnm{K{\"a}mmerer}, \binits{L.}}:
\batitle{Reconstructing hyperbolic cross trigonometric polynomials by sampling
  along rank-1 lattices}.
\bjtitle{SIAM J. Numer. Anal.}
\bvolume{51}(\bissue{5}),
\bfpage{2773}--\blpage{2796}
(\byear{2013})
\doiurl{10.1137/120871183}
\end{barticle}
\endbibitem

%%% 33
\bibitem[\protect\citeauthoryear{K{\"a}mmerer
  et~al.}{2015}]{KammererPottsVolkmer2015}
\begin{barticle}
\bauthor{\bsnm{K{\"a}mmerer}, \binits{L.}},
\bauthor{\bsnm{Potts}, \binits{D.}},
\bauthor{\bsnm{Volkmer}, \binits{T.}}:
\batitle{Approximation of multivariate periodic functions by trigonometric
  polynomials based on rank-1 lattice sampling}.
\bjtitle{J. Complexity}
\bvolume{31}(\bissue{4}),
\bfpage{543}--\blpage{576}
(\byear{2015})
\doiurl{10.1016/j.jco.2015.02.004}
\end{barticle}
\endbibitem

%%% 34
\bibitem[\protect\citeauthoryear{K{\"a}mmerer and
  Volkmer}{2019}]{kammerer_2019JAT_approximation}
\begin{barticle}
\bauthor{\bsnm{K{\"a}mmerer}, \binits{L.}},
\bauthor{\bsnm{Volkmer}, \binits{T.}}:
\batitle{Approximation of multivariate periodic functions based on sampling
  along multiple rank-1 lattices}.
\bjtitle{J. Approx. Theory}
\bvolume{246},
\bfpage{1}--\blpage{27}
(\byear{2019})
\doiurl{10.1016/j.jat.2019.05.001}
\end{barticle}
\endbibitem

%%% 35
\bibitem[\protect\citeauthoryear{Kempf and Wendland}{2023}]{KempfWendland2023}
\begin{barticle}
\bauthor{\bsnm{Kempf}, \binits{R.}},
\bauthor{\bsnm{Wendland}, \binits{H.}}:
\batitle{High-dimensional approximation with kernel-based multilevel methods on
  sparse grids}.
\bjtitle{Numer. Math.}
\bvolume{154},
\bfpage{485}--\blpage{519}
(\byear{2023})
\doiurl{10.1007/s00211-023-01363-x}
\end{barticle}
\endbibitem

%%% 36
\bibitem[\protect\citeauthoryear{Kolomoitsev}{2022}]{Kolomoitsev2022}
\begin{barticle}
\bauthor{\bsnm{Kolomoitsev}, \binits{Y.}}:
\batitle{Approximation by quasi-interpolation operators and {S}molyak's
  algorithm}.
\bjtitle{J. Complexity}
\bvolume{69},
\bfpage{101601}
(\byear{2022})
\doiurl{10.1016/j.jco.2021.101601}
\end{barticle}
\endbibitem

%%% 37
\bibitem[\protect\citeauthoryear{Kolomoitsev
  et~al.}{2023}]{KolomoitsevLomakoTikhonov2023}
\begin{barticle}
\bauthor{\bsnm{Kolomoitsev}, \binits{Y.}},
\bauthor{\bsnm{Lomako}, \binits{T.}},
\bauthor{\bsnm{Tikhonov}, \binits{S.}}:
\batitle{Sparse grid approximation in weighted wiener spaces}.
\bjtitle{J. Fourier Anal. Appl.}
\bvolume{29}(\bissue{2}),
\bfpage{19}
(\byear{2023})
\doiurl{10.1007/s00041-023-09994-2}
\end{barticle}
\endbibitem

%%% 38
\bibitem[\protect\citeauthoryear{Kolomoitsev and
  Prestin}{2021}]{KolomoitsevPrestin2021}
\begin{barticle}
\bauthor{\bsnm{Kolomoitsev}, \binits{Y.}},
\bauthor{\bsnm{Prestin}, \binits{J.}}:
\batitle{Approximation properties of periodic multivariate quasi-interpolation
  operators}.
\bjtitle{J. Approx. Theory}
\bvolume{270},
\bfpage{105631}
(\byear{2021})
\doiurl{10.1016/j.jat.2021.105631}
\end{barticle}
\endbibitem

%%% 39
\bibitem[\protect\citeauthoryear{K{\"u}hne et~al.}{2020}]{kuhne_2020JCP_cp2k}
\begin{botherref}
\oauthor{\bsnm{K{\"u}hne}, \binits{T.D.}},
\oauthor{\bsnm{Iannuzzi}, \binits{M.}},
\oauthor{\bsnm{Del~Ben}, \binits{M.}},
\oauthor{\bsnm{Rybkin}, \binits{V.V.}},
\oauthor{\bsnm{Seewald}, \binits{P.}},
\oauthor{\bsnm{Stein}, \binits{F.}},
\oauthor{\bsnm{Laino}, \binits{T.}},
\oauthor{\bsnm{Khaliullin}, \binits{R.Z.}},
\oauthor{\bsnm{Sch{\"u}tt}, \binits{O.}},
\oauthor{\bsnm{Schiffmann}, \binits{F.}}, et al.:
Cp2k: An electronic structure and molecular dynamics software
  package-{Q}uickstep: Efficient and accurate electronic structure
  calculations.
J. Chem. Phys.
\textbf{152}(19)
(2020)
\doiurl{10.1063/5.0007045}
\end{botherref}
\endbibitem

%%% 40
\bibitem[\protect\citeauthoryear{Kuo}{2003}]{Kuo2003}
\begin{barticle}
\bauthor{\bsnm{Kuo}, \binits{F.Y.}}:
\batitle{Component-by-component constructions achieve the optimal rate of
  convergence for multivariate integration in weighted korobov and sobolev
  spaces}.
\bjtitle{J. Complexity}
\bvolume{19}(\bissue{3}),
\bfpage{301}--\blpage{320}
(\byear{2003})
\doiurl{10.1016/S0885-064X(03)00006-2}
\end{barticle}
\endbibitem

%%% 41
\bibitem[\protect\citeauthoryear{Kuo et~al.}{2021}]{kuo_2021MCoM_function}
\begin{barticle}
\bauthor{\bsnm{Kuo}, \binits{F.Y.}},
\bauthor{\bsnm{Migliorati}, \binits{G.}},
\bauthor{\bsnm{Nobile}, \binits{F.}},
\bauthor{\bsnm{Nuyens}, \binits{D.}}:
\batitle{Function integration, reconstruction and approximation using rank-1
  lattices}.
\bjtitle{Math. Comput.}
\bvolume{90}(\bissue{330}),
\bfpage{1861}--\blpage{1897}
(\byear{2021})
\doiurl{10.1090/mcom/3595}
\end{barticle}
\endbibitem

%%% 42
\bibitem[\protect\citeauthoryear{Kuo et~al.}{2006}]{KuoSloanWozniakowski2006}
\begin{bchapter}
\bauthor{\bsnm{Kuo}, \binits{F.Y.}},
\bauthor{\bsnm{Sloan}, \binits{I.H.}},
\bauthor{\bsnm{Wo{\'z}niakowski}, \binits{H.}}:
\bctitle{Lattice rules for multivariate approximation in the worst case
  setting}.
In: \beditor{\bsnm{Niederreiter}, \binits{H.}},
\beditor{\bsnm{Talay}, \binits{D.}} (eds.)
\bbtitle{Monte Carlo and Quasi-Monte Carlo Methods 2004},
pp. \bfpage{289}--\blpage{330}.
\bpublisher{Springer},
\blocation{Berlin}
(\byear{2006}).
\doiurl{10.1007/3-540-31186-6_18}
\end{bchapter}
\endbibitem

%%% 43
\bibitem[\protect\citeauthoryear{Lippert
  et~al.}{2023}]{LippertPottsUllrich2023}
\begin{barticle}
\bauthor{\bsnm{Lippert}, \binits{L.}},
\bauthor{\bsnm{Potts}, \binits{D.}},
\bauthor{\bsnm{Ullrich}, \binits{T.}}:
\batitle{Fast hyperbolic wavelet regression meets {ANOVA}}.
\bjtitle{Numer. Math.}
\bvolume{154},
\bfpage{155}--\blpage{207}
(\byear{2023})
\doiurl{10.1007/s00211-023-01358-8}
\end{barticle}
\endbibitem

%%% 44
\bibitem[\protect\citeauthoryear{McDiarmid}{1989}]{McDiarmid1989}
\begin{bchapter}
\bauthor{\bsnm{McDiarmid}, \binits{C.}}:
\bctitle{On the method of bounded differences}.
In: \beditor{\bsnm{Siemons}, \binits{J.}} (ed.)
\bbtitle{Surveys in Combinatorics, 1989}.
\bsertitle{London Mathematical Society Lecture Note Series},
vol. \bseriesno{141},
pp. \bfpage{148}--\blpage{188}.
\bpublisher{Cambridge University Press}
(\byear{1989}).
\doiurl{10.1017/CBO9781107359949.008}
\end{bchapter}
\endbibitem

%%% 45
\bibitem[\protect\citeauthoryear{Narcowich et~al.}{2007}]{NarcowichSunWard2007}
\begin{barticle}
\bauthor{\bsnm{Narcowich}, \binits{F.J.}},
\bauthor{\bsnm{Sun}, \binits{X.P.}},
\bauthor{\bsnm{Ward}, \binits{J.D.}}:
\batitle{Approximation power of {RBF}s and their associated {SBF}s: A
  connection}.
\bjtitle{Adv. Comput. Math.}
\bvolume{27}(\bissue{1}),
\bfpage{107}--\blpage{124}
(\byear{2007})
\doiurl{10.1007/s10444-005-7506-1}
\end{barticle}
\endbibitem

%%% 46
\bibitem[\protect\citeauthoryear{Novak
  et~al.}{2004}]{NovakSloanWozniakowski2004}
\begin{barticle}
\bauthor{\bsnm{Novak}, \binits{E.}},
\bauthor{\bsnm{Sloan}, \binits{I.H.}},
\bauthor{\bsnm{Wo{\'z}niakowski}, \binits{H.}}:
\batitle{Tractability of approximation for weighted {K}orobov spaces on
  classical and quantum computers}.
\bjtitle{Found. Comput. Math.}
\bvolume{4}(\bissue{2}),
\bfpage{121}--\blpage{156}
(\byear{2004})
\doiurl{10.1007/s10208-002-0074-6}
\end{barticle}
\endbibitem

%%% 47
\bibitem[\protect\citeauthoryear{Nuyens and Cools}{2006}]{NuyensCools2006}
\begin{barticle}
\bauthor{\bsnm{Nuyens}, \binits{D.}},
\bauthor{\bsnm{Cools}, \binits{R.}}:
\batitle{Fast algorithms for component-by-component construction of rank-1
  lattice rules in shift-invariant reproducing kernel hilbert spaces}.
\bjtitle{Math. Comput.}
\bvolume{75}(\bissue{254}),
\bfpage{903}--\blpage{920}
(\byear{2006})
\doiurl{10.1090/S0025-5718-06-01785-6}
\end{barticle}
\endbibitem

%%% 48
\bibitem[\protect\citeauthoryear{Potts and
  Schmischke}{2021}]{PottsSchmischke2021}
\begin{barticle}
\bauthor{\bsnm{Potts}, \binits{D.}},
\bauthor{\bsnm{Schmischke}, \binits{M.}}:
\batitle{Approximation of high-dimensional periodic functions with
  {F}ourier-based methods}.
\bjtitle{SIAM J. Numer. Anal.}
\bvolume{59}(\bissue{5}),
\bfpage{2393}--\blpage{2429}
(\byear{2021})
\doiurl{10.1137/20M1354921}
\end{barticle}
\endbibitem

%%% 49
\bibitem[\protect\citeauthoryear{Rieger and
  Wendland}{2024}]{rieger_2024SINUM_approximability}
\begin{barticle}
\bauthor{\bsnm{Rieger}, \binits{C.}},
\bauthor{\bsnm{Wendland}, \binits{H.}}:
\batitle{On the approximability and curse of dimensionality of certain classes
  of high-dimensional functions}.
\bjtitle{SIAM J. Numer. Analy.}
\bvolume{62}(\bissue{2}),
\bfpage{842}--\blpage{871}
(\byear{2024})
\doiurl{10.1137/22M1525193}
\end{barticle}
\endbibitem

%%% 50
\bibitem[\protect\citeauthoryear{Rieger and
  Wendland}{2026}]{RiegerWendland2026}
\begin{barticle}
\bauthor{\bsnm{Rieger}, \binits{C.}},
\bauthor{\bsnm{Wendland}, \binits{H.}}:
\batitle{Constructive approximation of high-dimensional functions with small
  efficient dimension with applications in uncertainty quantification}.
\bjtitle{SIAM J. Numer. Anal.}
\bvolume{64}(\bissue{1}),
\bfpage{1}--\blpage{28}
(\byear{2026})
\doiurl{10.1137/24M171231X}
\end{barticle}
\endbibitem

%%% 51
\bibitem[\protect\citeauthoryear{Smith}{2013}]{smith_Book_2013}
\begin{bbook}
\bauthor{\bsnm{Smith}, \binits{R.C.}}:
\bbtitle{Uncertainty Quantification: Theory, Implementation, and Applications}.
\bsertitle{Computational Science and Engineering},
vol. \bseriesno{12}.
\bpublisher{SIAM},
\blocation{Philadelphia, PA}
(\byear{2013}).
\doiurl{10.1137/1.9781611973228}
\end{bbook}
\endbibitem

%%% 52
\bibitem[\protect\citeauthoryear{Sobol}{1967}]{Sobol1967}
\begin{barticle}
\bauthor{\bsnm{Sobol}, \binits{I.M.}}:
\batitle{On the distribution of points in a cube and the approximate evaluation
  of integrals}.
\bjtitle{USSR Comput. Math. Math. Phys.}
\bvolume{7}(\bissue{4}),
\bfpage{86}--\blpage{112}
(\byear{1967})
\doiurl{10.1016/0041-5553(67)90144-9}
\end{barticle}
\endbibitem

%%% 53
\bibitem[\protect\citeauthoryear{Sun et~al.}{2025}]{sun_2025IMAJNA_scaled}
\begin{botherref}
\oauthor{\bsnm{Sun}, \binits{Z.J.}},
\oauthor{\bsnm{Gao}, \binits{W.W.}},
\oauthor{\bsnm{Sun}, \binits{X.P.}}:
Scaled zonal kernel quasi-interpolation on spheres.
IMA J. Numer. Anal.
\textbf{draf104}
(2025)
\doiurl{10.1093/imanum/draf104}
\end{botherref}
\endbibitem

%%% 54
\bibitem[\protect\citeauthoryear{Sun et~al.}{2026}]{sun_2026SINUM_monte}
\begin{barticle}
\bauthor{\bsnm{Sun}, \binits{Z.J.}},
\bauthor{\bsnm{Lv}, \binits{M.Y.}},
\bauthor{\bsnm{Sun}, \binits{X.P.}}:
\batitle{Monte carlo quasi-interpolation of spherical data}.
\bjtitle{SIAM J. Numer. Anal.}
\bvolume{64}(\bissue{5}),
\bfpage{2005}--\blpage{2027}
(\byear{2026})
\doiurl{10.1137/25M1810490}
\end{barticle}
\endbibitem

%%% 55
\bibitem[\protect\citeauthoryear{Usta and Levesley}{2018}]{UstaLevesley2018}
\begin{barticle}
\bauthor{\bsnm{Usta}, \binits{F.}},
\bauthor{\bsnm{Levesley}, \binits{J.}}:
\batitle{Multilevel quasi-interpolation on a sparse grid with the {Gaussian}}.
\bjtitle{Numer. Algorithms}
\bvolume{77}(\bissue{3}),
\bfpage{793}--\blpage{808}
(\byear{2018})
\doiurl{10.1007/s11075-017-0340-y}
\end{barticle}
\endbibitem

%%% 56
\bibitem[\protect\citeauthoryear{Wang
  et~al.}{2022}]{wang_2022NatComputSci_efficient}
\begin{barticle}
\bauthor{\bsnm{Wang}, \binits{D.D.}},
\bauthor{\bsnm{Wang}, \binits{Y.Z.}},
\bauthor{\bsnm{Chang}, \binits{J.H.}},
\bauthor{\bsnm{Zhang}, \binits{L.F.}},
\bauthor{\bsnm{Wang}, \binits{H.}},
\bauthor{\bsnm{E}, \binits{W.N.}}:
\batitle{Efficient sampling of high-dimensional free energy landscapes using
  adaptive reinforced dynamics}.
\bjtitle{Nat. Comput. Sci.}
\bvolume{2}(\bissue{1}),
\bfpage{20}--\blpage{29}
(\byear{2022})
\doiurl{10.1038/s43588-021-00173-1}
\end{barticle}
\endbibitem

%%% 57
\bibitem[\protect\citeauthoryear{Wasilkowski and
  Wo{\'z}niakowski}{1999}]{WasilkowskiWozniakowski1999}
\begin{barticle}
\bauthor{\bsnm{Wasilkowski}, \binits{G.W.}},
\bauthor{\bsnm{Wo{\'z}niakowski}, \binits{H.}}:
\batitle{Weighted tensor product algorithms for linear multivariate problems}.
\bjtitle{J. Complexity}
\bvolume{15}(\bissue{3}),
\bfpage{402}--\blpage{447}
(\byear{1999})
\doiurl{10.1006/jcom.1999.0512}
\end{barticle}
\endbibitem

%%% 58
\bibitem[\protect\citeauthoryear{Wu and Liu}{2005}]{wu_2005Adv_generalized}
\begin{barticle}
\bauthor{\bsnm{Wu}, \binits{Z.M.}},
\bauthor{\bsnm{Liu}, \binits{J.P.}}:
\batitle{Generalized {S}trang-{F}ix condition for scattered data
  quasi-interpolation}.
\bjtitle{Adv. Comput. Math.}
\bvolume{23}(\bissue{1}),
\bfpage{201}--\blpage{214}
(\byear{2005})
\doiurl{10.1007/s10444-004-1832-6}
\end{barticle}
\endbibitem

%%% 59
\bibitem[\protect\citeauthoryear{Xiao and Boyd}{2014}]{XiaoBoyd2014}
\begin{barticle}
\bauthor{\bsnm{Xiao}, \binits{J.}},
\bauthor{\bsnm{Boyd}, \binits{J.P.}}:
\batitle{Periodized radial basis functions, part {I}: Theory}.
\bjtitle{Appl. Numer. Math.}
\bvolume{86},
\bfpage{43}--\blpage{73}
(\byear{2014})
\doiurl{10.1016/j.apnum.2014.07.003}
\end{barticle}
\endbibitem

%%% 60
\bibitem[\protect\citeauthoryear{Zeng et~al.}{2009}]{ZengKritzerHickernell2009}
\begin{barticle}
\bauthor{\bsnm{Zeng}, \binits{X.}},
\bauthor{\bsnm{Kritzer}, \binits{P.}},
\bauthor{\bsnm{Hickernell}, \binits{F.J.}}:
\batitle{Spline methods using integration lattices and digital nets}.
\bjtitle{Constr. Approx.}
\bvolume{30}(\bissue{3}),
\bfpage{529}--\blpage{555}
(\byear{2009})
\doiurl{10.1007/s00365-009-9072-0}
\end{barticle}
\endbibitem

\end{thebibliography}

\end{document}